\documentclass[a4paper,11pt]{amsart}
\input{Macros}
\usepackage[frenchb, english]{babel}
\usepackage{tabu}

\newcommand{\fm}{\mathfrak{m}}
\newcommand{\fn}{\mathfrak{n}}
\newcommand{\fp}{\mathfrak{p}}
\newcommand{\fq}{\mathfrak{q}}

\DeclareMathOperator{\Char}{char}
\DeclareMathOperator{\src}{Src}
\DeclareMathOperator{\spr}{Spr}
\DeclareMathOperator{\arr}{Arrow}

\title{Gluing theory for slc surfaces and threefolds in positive characteristic}
\author{Quentin Posva}
\date{}

\address{University of Utah, Department of Mathematics, JWB 311, 84112 Salt Lake City, Utah, USA}
\email{posva@math.utah.edu}

\begin{document}

\maketitle

\begin{quote}
\textsc{Abstract.} We develop a gluing theory in the sense of Koll\'{a}r for slc surfaces and threefolds in positive characteristic. For surfaces we are able to deal with every positive characteristic $p$, while for threefolds we assume that $p>5$. Along the way we study nodes in characteristic $2$ and establish a theory of sources and springs \emph{à la} Koll\'{a}r for threefolds. We also give applications to the topology of lc centers on slc threefolds, and to the projectivity of the moduli space of stable surfaces in characteristic $p>5$.
\end{quote}

{\let\thefootnote\relax\footnote{{ 
\emph{MSC numbers:} 14J17, 14J29, 14J30.
}
}}

\tableofcontents

\section{Introduction}
Varieties with semi-log canonical (slc) singularities were first introduced by Koll\'{a}r and Shepherd-Barron in \cite{Kollar_Shepherd_Barron_3folds_and_deformations_of_surfaces_singularities} to study the moduli functor of surfaces of general type over the complex numbers. They showed that slc singularities appear on the special fibers of stable degenerations of smooth surfaces of general type, and that the moduli functor of complex slc surfaces of general type is coarsely represented by a separated algebraic space. It was proved a few years later that this coarse moduli space is in fact projective \cite{Kollar_Projectivity_complete_moduli, Alexeev_Boundedness_and_K2_for_log_surfaces}. These results showed the importance of canonically polarized slc varieties (also called stable varieties) for a modular meaningful compactification of the moduli functor of smooth varieties of general type. Indeed, it has been shown that the moduli functor $\overline{\sM}_{n,v}$ of stable varieties (of any fixed dimension $n$ and volume $v$) over a field $k$ of characteristic $0$, is coarsely represented by a projective algebraic space of finite type over $k$
\cite{Viehweg_Qproj_Moduli_for_polarized_manifolds,
Hassett_Kovacs_Reflexive_pullbacks_and_base_extension,
Karu_Minimal_models_and_boundedness_of_stable_varieties,
Abramovich_Hassett_Stable_varieties_with_a_twist,
Kollar_Singularities_of_the_minimal_model_program,
Fujino_Semipositivity_theorems_for_moduli_problems,
Hacon_McKernan_Xu_Boundedness_of_moduli_of_varieties_of_general_type,
Kovacs_Patakfalvi_Projectivity_of_moduli_space_of_stable_log_varieties}.
See the forthcoming book \cite{Kollar_Families_of_varieties_of_general_type} for an exposition of the proof and of many related results.

The study of slc singularities and the construction of the compact coarse moduli space of $\overline{\sM}_{n,v}$ owe much to the recent development of the Minimal Model Program (MMP), see e.g. \cite{Karu_Minimal_models_and_boundedness_of_stable_varieties,
Hacon_McKernan_Xu_Boundedness_of_moduli_of_varieties_of_general_type,
Kollar_Singularities_of_the_minimal_model_program}. Here a technical issue seems to arise: the methods of the MMP work best for normal varieties while they might fail for normal crossing singularities \cite{Kollar_Two_examples_of_nc_surfaces}, and slc singularities are typically not normal as they might have nodes in codimension one (the precise definition is recalled in \autoref{section: demi-normal schemes}). An effective way to solve this issue is to study the normalization, which has log canonical (lc) singularities, and then descend the information along the normalization morphism. This idea is exploited to prove the valuative criterion of properness of $\overline{\sM}_{n,v}$, see \autoref{section:applications_to_moduli_spaces} for details, and has been fruitful to study abundance of stable varieties, \cite[\S 12]{Flips_and_abundance_for_3folds} and \cite{Hacon_Xu_Existence_of_lc_closures, Hacon_Xu_Finiteness_of_B_representations_and_slc_abundance,
Fujino_Gongyo_Log_pluricanonical_representations_and_abundance}. These results rely on a theory that allows one to go back and forth between slc varieties and their normalization in a systematic way. The main goal of this article is to establish such a theory for surfaces and threefolds in positive characteristic.

\subsection{Gluing theory for slc varieties}

Let $X$ be an slc variety, with normalization $\nu\colon \bar{X}\to X$, and conductor divisors $D\subset X$ and $\bar{D}\subset \bar{X}$. The morphism $\bar{D}\to D$ generically looks like the normalization of the node, hence it is generically Galois of degree $2$ (at least in characteristic $\neq 2$). This induces a Galois involution $\tau$ on the normalization of $\bar{D}$. Then one should be able to reconstruct $X$ by gluing together the points that belong to the same $\tau$-equivalence classes. This is rather easy to accomplish with the theory of quotients by finite equivalence relations. What is difficult is to decide which triplets $(\bar{X},\bar{D},\tau)$ arise as normalization of slc varieties. In the case of stable varieties over a field of characteristic $0$, the answer is given by Koll\'{a}r's gluing theory, developed in \cite[\S 5,\S 9]{Kollar_Singularities_of_the_minimal_model_program}. An overview of this theory is given in \autoref{section: demi-normal schemes}, \autoref{section:Quotients_by_finite_equivalence_relations} and \autoref{section:Kollar_gluing_theory}.

In this article, we establish a gluing theory for slc surfaces and threefolds in positive characteristic. As the case of surfaces and the case of threefolds are quite different, with respect to both methods and results, we shall present them separately.

\subsubsection{Gluing for surfaces}
For surfaces we prove the following theorem:
\begin{theorem_intro}[\autoref{theorem:gluing_for_surfaces_p_different_from_2} and \autoref{theorem:gluing_for_surfaces_in_char_2}, see also \autoref{theorem:gluing_for_germs_of_surfaces}]
Let $k$ be a field of positive characteristic.
	\begin{enumerate}
		\item If $\Char k> 2$, then normalization gives a bijection
		\begin{equation*}
		\begin{pmatrix}
		\text{Slc surface pairs } (S,\Delta)\\
		\text{of finite type over }k
		\end{pmatrix}
		\overset{1:1}{\longrightarrow}
		\begin{pmatrix}
		\text{Lc surface pairs } (\bar{S},\bar{D}+\bar{\Delta})  \\
		\text{of finite type over }k \\
		\text{plus an involution } \tau\text{ of the pair }(\bar{D}^n,\Diff_{\bar{D}^n}\bar{\Delta})\\
		\text{that is generically fixed point free on every component }.
		\end{pmatrix}
		\end{equation*}
		
		\item If $\Char k=2$, then normalization gives a bijection
		\begin{equation*}
		\begin{pmatrix}
		\text{Slc surface pairs }(S,\Delta)\\
		\text{of finite type over }k
		\end{pmatrix}
		\overset{1:1}{\longrightarrow}
		\begin{pmatrix}
		\text{Lc surface pairs } (\bar{S},\bar{D}_\emph{Gal}+\bar{D}_\emph{ins}+\bar{\Delta})  \\
		\text{of finite type over }k \\
		\text{where }\bar{D}_\emph{Gal}, \bar{D}_\emph{ins} \text{ and }\bar{\Delta}\text{ have no common component,}\\
		\text{plus an involution } \tau\text{ of the pair }(\bar{D}_\emph{Gal}^n,\Diff_{\bar{D}_\emph{Gal}^n}(\bar{\Delta}+\bar{D}_\emph{ins})) \\
		\text{that is generically fixed point free on every component.}
		\end{pmatrix}
		\end{equation*}
	\end{enumerate}
Here it is understood that the divisors $\bar{D},\bar{D}_\emph{Gal},\bar{D}_\emph{ins}$ are reduced and that the divisors $\bar{\Delta}$ have rational coefficients in $(0;1]$.
\end{theorem_intro}

We emphasize that the surface pairs in the above theorem are not assumed to be proper over the base field.  Actually, we also prove a semi-local version of the above theorem in \autoref{theorem:gluing_for_germs_of_surfaces}. This seems to be a special feature of the lc surface case: compare with \cite[5.13]{Kollar_Singularities_of_the_minimal_model_program} and the counterexamples in \cite[\S 9.4]{Kollar_Singularities_of_the_minimal_model_program}.

If $\Char k>2$ the proof follows from the results of \cite[\S 5.3]{Kollar_Singularities_of_the_minimal_model_program} and the observation that the equivalence relation on $\bar{S}$ generated by $\tau$ is finite by adjunction and dimensional reasons. Let me mention that in the semi-local case, the construction of the inverse map to the normalization is slightly indirect: see the paragraph before \autoref{theorem:gluing_for_germs_of_surfaces} for some comments. The case $\Char k=2$ is more interesting, as a special case of normalization appears: the morphism between the conductors $\bar{D}\to D$ might be purely inseparable, see \autoref{example: inseparable nodes}. We are able to analyse these cases in detail and in every dimension: see \autoref{section:geometry_of_demi_normal_varieties} below.

\subsubsection{Gluing for threefolds}
Gluing theory for threefold pairs $(X,\Delta)$ is more challenging. We follow the strategy of Koll\'{a}r in \cite[\S 4,\S 5]{Kollar_Singularities_of_the_minimal_model_program}. The $\tau$-equivalence classes preserve the lc stratification of $(X,\Delta)$, and become more and more complicated as we look at higher codimension lc centers. To keep track of the gluing process, it is therefore desirable to understand the structure of these high codimension lc centers (see \autoref{remark:proof_in_dim_2}). Instead of looking directly at the lc centers on $X$, Koll\'{a}r's idea is to consider a crepant $\bQ$-factorial dlt blow-up of $(X,\Delta)$: the advantange is that the lc centers are nicer on this biraitonal model (see \autoref{proposition: lc centers on dlt 3-folds}). The following theorem describes precisely the relation between lc centers on $X$ and on the model:

\begin{theorem_intro}[\autoref{theorem:spring_and_sources} and \autoref{cor:adjunction_for_reduced_boundary}]\label{theorem_intro:spring_and_sources}
Let $f\colon (Y,\Delta_Y)\to (X,\Delta=\Delta^{=1}+\Delta^{<1})$ be crepant $\bQ$-factorial dlt blow-up of a quasi-projective lc threefold pair over a perfect field of characteristic $>5$. Let $Z\subset X$ be a lc center contained in $\Delta^{=1}$ with normalization $Z^n\to Z$.

Let $(S,\Delta_S:=\Diff^*_{S}\Delta_Y)\subset Y$ be a minimal lc center over $Z$, with Stein factorization $f^n_S\colon S\to Z_S\to Z^n$. 
Then:
	\begin{enumerate}
		\item The crepant birational equivalence class of $(S,\Delta_S)$ over $Z$ does not depend on the choice of $S$ or $Y$. We call it the \textbf{source} of $Z$, and denote it by $\src(Z, X,\Delta)$.
		\item The isomorphism class of $Z_S$ over $Z$ does not depend on the choice of $S$ or $Y$. We call it the \textbf{spring} of $Z$, and denote it by $\spr(Z,X,\Delta)$.
		\item $(S,\Delta_S)$ is dlt, $K_S+\Delta_S\sim_{\bQ,Z}0$ and $(S,\Delta_S)$ is klt on the general fiber above $Z$.
		\item The field extension $k(Z)\subset k(Z_S)$ is Galois and $\Bir^c_Z(S,\Delta_S)\twoheadrightarrow\Gal(Z_S/Z)$.
		\item For $m>0$ divisible enough, there are well-defined Poincaré isomorphisms
				$$\omega_Y^{[m]}(m\Delta_Y)|_S\cong \omega_S^{[m]}(m\Delta_S).$$
		\item If $W\subset (\Delta^{=1})^n$ is an irreducible closed subvariety such that $n(W)=Z$, where $n\colon (\Delta^{=1})^n\to \Delta^{=1}$ is the normalization, then
				$$\src(W,(\Delta^{=1})^n,\Diff_{(\Delta^{=1})^n}\Delta^{<1})\overset{\emph{cbir}}{\sim} \src(Z, X,\Delta)$$
			and
				$$\spr(W,(\Delta^{=1})^n,\Diff_{(\Delta^{=1})^n}\Delta^{<1})\cong \spr(Z, X,\Delta).$$
	\end{enumerate}
\end{theorem_intro}

This theorem is analog to \cite[4.45]{Kollar_Singularities_of_the_minimal_model_program} and indeed our proof is similar to Koll\'{a}r's one. Notice however that we only consider lc centers that are contained in the reduced boundary $\Delta^{=1}$: see \autoref{remark:sources_for_general_lc_centers} for what we can say about the other lc centers. The main step of the proof is to study the relation between the lc centers of $(Y,\Delta_Y)$ that are minimal above $Z$. This is done in \autoref{section: geometry of lc centers}, using the notion of weak $\bP^1$-links (see \autoref{definition:weak-P1_link}), which is a slight generalization of \cite[4.36]{Kollar_Singularities_of_the_minimal_model_program}. The proofs are once again similar to those of Koll\'{a}r, but some complications arise: for example the analog of \cite[4.37]{Kollar_Singularities_of_the_minimal_model_program} in positive characteristic requires some thought, since the torsion-freeness of some higher pushforward is not available in positive characteristic. This is bypassed in the proof of \autoref{prop:minimal_centers_are_linked} using the MMP and connectedness theorems.  

Once the theory of higher codimension adjunction is established, we apply it to the gluing problems for lc threefolds. We obtain:
\begin{theorem_intro}[\autoref{theorem:gluing_for_threefolds}]
Let $k$ be a perfect field of characteristic $> 5$. Then normalization gives a bijection
	\begin{equation*}
		\begin{pmatrix}
		\text{Proper slc threefold pairs}\\
		(X,\Delta) \text{ such that}\\
		K_X+\Delta\text{ is ample}
		\end{pmatrix}
		\overset{1:1}{\longrightarrow}
		\begin{pmatrix}
		\text{Proper lc threefold pairs } (\bar{X},\bar{D}+\bar{\Delta})  \\
		\text{plus an involution }\tau\text{ of }(\bar{D}^n,\Diff_{\bar{D}^n}\bar{\Delta}) \\
		\text{that is generically fixed point free on each component}\\
		\text{such that }K_{\bar{X}}+\bar{D}+\bar{\Delta} \text{ is ample.}
		\end{pmatrix}
		\end{equation*}
\end{theorem_intro}
Contrarily to the surface case, the properness assumption is necessary. Otherwise the involution $\tau$ might generate infinite equivalence classes: see for example \cite[9.57]{Kollar_Singularities_of_the_minimal_model_program}.

The method of proof is similar to \cite[5.33, 5.37, 5.38]{Kollar_Singularities_of_the_minimal_model_program}. Let us emphasize that quotients are actually easier to construct in positive characteristic than in characteristic $0$, thanks to \cite[Theorem 6]{Kollar_Quotients_by_finite_equivalence_relations}: one only has to prove that the gluing relation is finite. The theory of sources and springs gives an appropriate set-up to study the gluing problem, but contrarily to the surface case, finiteness is non-trivial. In characteristic $0$, the proof of finiteness relies on the theory of pluricanonical representations (also called B-representations by some authors), see \cite[\S 10.5]{Kollar_Singularities_of_the_minimal_model_program}. This is well-understood in characteristic $0$, but to my knowledge completely open in general over a field of positive characteristic. Fortunately the gluing problem for threefolds involves only simple pluricanonical representations, which are discussed in \autoref{section:pluricanonical_representations}. 

Before presenting some applications of this theory, let us say a word about the generalization to higher dimensions. The theory of sources and springs for threefolds (except for the Galois property) ultimately relies on existence of crepant $\bQ$-factorial blow-up for threefolds and on the MMP for surfaces: recent developments in birational geometry, e.g.
\cite{Waldron_MMP_for_3_folds_in_char_>5,
Hacon_Witaszek_On_the_relative_MMP_for_threefolds_in_low_char, Bernasconi_Kollar_vanishing_theorems_for_threefolds_in_char_>5},
make us confident that these tools will soon be available in one dimension higher at least. Two other aspects of the proof are more problematic: the Galois property and pluricanonical representations. For threefolds, we prove the Galois property using the classification of surface lc singularities to make sure that inseparable extensions do not appear (see Step 3 in the proof of \autoref{theorem:spring_and_sources}). A finer approach is needed in higher dimension. As for pluricanonical representations, already the case of surfaces of Kodaira dimension $1$ is challenging. 

\subsection{Applications}
Let us present some applications of our theorems. 

\subsubsection{Geometry of demi-normal varieties}\label{section:geometry_of_demi_normal_varieties}
The study of nodal singularities leads to several interesting consequence for the geometry of demi-normal schemes.

In characteristic $2$, the normalization of a demi-normal scheme $\bar{X}\to X$ might create a conductor morphism $\bar{D}\to D$ that is purely inseparable. To understand these cases, which we baptise \emph{inseparable nodes}, we analyse in detail the normalization theory of nodes in \autoref{section:normalization_of_nodes}. It turns out that inseparable nodes in any dimension are completely determined by the extension of functions fields given by $\bar{D}\to D$, namely:

\begin{theorem_intro}[\autoref{theorem: bijection for inseparable nodes}]
Let $k$ be a field of characteristic $2$. Then normalization gives a bijection
	\begin{equation*}
		\begin{pmatrix}
		\text{Demi-normal varieties over }k\\
		\text{with only inseparable nodes}
		\end{pmatrix}
		\overset{1:1}{\longrightarrow}
		\begin{pmatrix}
		\text{Triples }\big(X,\sum_i D_i, k\subset k_i'\subset k(D_i)\big)\\
		\text{where }X\text{ is a normal variety over }k,\\
		 \sum_iD_i \text{ is a reduced Weil divisor},\\ 
		 k_i'\subset k(D_i)\text{ are degree }2\\
		 \text{inseparable extensions over }k
		\end{pmatrix}
	\end{equation*}
whose inverse is given by \autoref{construction: construction of inseparable node}.
\end{theorem_intro}
We also study the \'{e}tale theory of inseparable nodes in \autoref{section:nodal_is_etale} and show that the distinction between separable and inseparable nodes holds \'{e}tale-locally.

\bigskip
The study of the gluing formalism in \autoref{section:construction_of_separable_nodes} and \autoref{section:inseparable_nodes} can be applied to the following compactification property:
\begin{theorem_intro}[\autoref{thm:embedding_in_demi_normal}]
Let $X$ be a demi-normal scheme that is separated and of finite type over a field $k$ of positive characteristic. Then $X$ embedds as an open subset of a demi-normal scheme that is proper over $k$.
\end{theorem_intro}
This is a positive characteristic analog of the main result of \cite{Berquist_Embedding_of_demi_normal_varieties}, and the proof is similar in spirit.

\bigskip
It is possible to study slc surfaces by means of partial resolutions, as it is done in \cite[\S 4]{Kollar_Shepherd_Barron_3folds_and_deformations_of_surfaces_singularities} for complex surfaces. While we do not use this approach, we record that gluing theory can be used to prove the existence of partial resolutions in positive characteristic $\neq 2$:

\begin{theorem_intro}[\autoref{theorem:semi_resolution_for_surfaces}]
Let $S$ be a demi-normal surface over an arbitrary field of characteristic $\neq 2$. Then $S$ has a good semi-resolution with slc singularities (see \autoref{definition:semi_resolution}).
\end{theorem_intro}
This result was part of the folklore: it is stated in \cite[4.2]{Kollar_Projectivity_complete_moduli} without restriction on the characteristic with a reference to \cite[1.4.3]{van_Straten_weakly_normal_surfaces}, where only the complex case is treated. In the complex case, a sketch of proof using gluing theory is given in \cite[12.2.3]{Flips_and_abundance_for_3folds}. Our contribution is to make this result precise and available in a greater generality. The restriction on the characteristic is made to avoid inseparable nodes: see \autoref{remark:semi_resolution_in_char_2} and \autoref{proposition:semi_resolution_in_char_2} for some partial results in characteristic $2$.

\subsubsection{Lc centers of threefolds}
The theory of sources and springs has the following consequence for the topology of lc centers:
\begin{theorem_intro}[\autoref{corollary:intersection_of_lc_centers} and \autoref{proposition:minimal_lc_centers_are_normal_up_to_homeo}]
Let $(X,\Delta)$ be a quasi-projective slc threefold over a perfect field of characteristic $>5$. Then:
	\begin{enumerate}
		\item Intersections of lc centers are union of lc centers. 
		\item Minimal lc centers are normal up to universal homeomorphism.
	\end{enumerate}
\end{theorem_intro}

\begin{remark}
Filipazzi and Waldron have independently obtained a more general result for $3$-folds in characteristic $>5$ in \cite[Corollary 1.5]{Filipazzi_Waldron_Connectedness_principle_for_3folds_pos_char}.\end{remark}

Let us mention that the question of normality of plt centers (a special case of minimal lc centers) has been studied extensively. There are examples of non-normal plt centers on threefolds in characteristic $2$ \cite{Cascini_Tanaka_Plt_threefolds_with_non_normal_centres_in_char_2}. More generally, if the characteristic is too small compared to the dimension, there are examples of non-normal plt centers in arbitrary big dimensions \cite[Theorem 1.1]{Bernasconi_Non_normal_plt_centers_in_pos_char}. On the positive side, plt centers on threefolds are normal in characteristic $>5$ \cite[Theorem 3.11, Proposition 4.1]{Hacon_Xu_On_the_3dim_MMP_in_pos_char} and normal up to universal homeomorphism in general \cite[Theorem 1.2]{Hacon_Witaszek_On_the_relative_MMP_for_threefolds_in_low_char}.

\subsubsection{Moduli space of stable surfaces}
We also give an application of the gluing theory of threefolds, to the moduli theory of stable surfaces in positive characteristic. By contrast to the characteristic $0$ theory, not much is known about the moduli functor of stable varieties in positive characteristic beyond the classical case of stable curves. For example, the existence of a coarse moduli space in arbitrary dimension is not known. Nevertheless, the moduli functor $\overline{\sM}_{2,v,k}$ of stable surfaces over an algebraically closed field $k$ of characteristic $>5$ is known to be represented by a separated algebraic space of finite type over $k$ \cite[Corollary 9.8]{Patakfalvi_Projectivity_moduli_space_of_surfaces_in_pos_char}. Modulo a weak version of semi-stable reduction and a technical adjunction condition (denoted respectively by \textbf{(SSR)} and \textbf{(S2)}, see \autoref{section:applications_to_moduli_spaces} for details), we prove the valuative criterion of properness for $\overline{\sM}_{2,v,k}$. Combining with the main result of \cite{Patakfalvi_Projectivity_moduli_space_of_surfaces_in_pos_char} we obtain:
\begin{theorem_intro}[\cite{Patakfalvi_Projectivity_moduli_space_of_surfaces_in_pos_char} and \autoref{theorem:Properness_of_moduli_of_stable_surfaces}]
Let $k$ be an algebraically closed field of characteristic $>5$ and $v$ a rational number. Assume that conditions \textbf{(S2)} and \textbf{(SSR)} hold. Then $\overline{\sM}_{2,v,k}$ admits a projective coarse moduli space.
\end{theorem_intro}

\subsection{Acknowledgements}
I would like to thank to my supervisor Zsolt Patakfalvi for his time, his support and several fruitful conversations. I am also thankful to the members of the CAG group at EPFL for moral support during the academic year 2020--2021, and to J\'{a}nos Koll\'{a}r for kindly answering my questions about the theory of sources and springs. Financial support was provided by the European Research Council under grant number 804334.

\section{Preliminaries}

\subsection{Conventions and notations}
We work over a field $k$ of positive characteristic. The assumptions on $k$ vary through the paper, and will be precised at the appropriate places. 

A \emph{variety} is a connected separated reduced equidimensional scheme of finite type over $k$. Note that a variety in our sense might be reducible. A \emph{curve} (resp. a \emph{surface}, resp. a \emph{threefold}) is a variety of dimension one (resp. two, resp. three).

Let $X$ be a scheme. A coherent $\sO_X$-module $\sF$ is $S_i$ if it satisfies the condition $\depth_{\sO_{X,x}}\sF_x\geq \min\{i,\dim\sF_x\}$ for all $x\in X$.

If $X$ is a reduced Noetherian scheme, its \emph{normalization} is defined to be its relative normalization along the structural morphism $\bigsqcup_\eta\Spec(k(\eta))\to X$ where $\eta$ runs through the generic points of $X$. Recall that $X$ is normal if and only if it is regular in codimension one and $\sO_X$ is $S_2$.

If $X$ is normal, a torsion-free coherent module $\sF$ is $S_2$ if and only if it is reflexive, if and only if $\sF=j_*\sF_U$ for any dense open subset $j\colon U\subset X$ satisfying $\codim X\setminus U\geq 2$. If $X$ is only $S_1$ (for example if it is reduced) then for an $S_1$ coherent $\sF$, being $S_2$ is equivalent to the equality $\sF=j_*\sF_U$ for any $U$ as above \cite[1.8]{Hartshorne_Generalized_divisors_and_biliaison}. Note that if $X$ is $S_1$ then any reflexive sheaf is also $S_1$ \cite[0AV5]{Stacks_Project}.

Given a normal variety $X$, we denote by $K_X$ any Weil divisor on $X$ associated to the invertible sheaf $\omega_{X_\text{reg}}$. A $\bQ$-Weil divisor $D$ on $X$ is \emph{$\bQ$-Cartier} if for some $m>0$ the reflexive sheaf $\sO_X(mD)$ is invertible.

A \emph{pair} $(X,\Delta)$ is the data of a demi-normal variety $X$ together with a $\bQ$-Weil divisor $\Delta$ whose coefficients belongs to $[0;1]$ and without components in common with $\Sing(X)$, such that $K_X+\Delta$ is $\bQ$-Cartier. The divisor $\Delta$ is sometimes called the \emph{boundary} of the pair. Given $\Delta$, we often consider the following divisors:
		$$\Delta^{=1}=\lfloor \Delta \rfloor:=\sum_{E \; : \; \coeff_E\Delta=1}E,\qquad \Delta^{<1}:=\Delta-\Delta^{=1}.$$
		
Let $(X,\Delta)$ and $(X',\Delta')$ be two pairs. A \emph{log isomorphism} of these pairs is an isomorphism $\varphi\colon X\cong X'$ such that $\varphi(\Delta)=\Delta'$ as $\bQ$-divisors. If $(X',\Delta')=(X,\Delta)$ and $\varphi$ has order $2$, we say that $\varphi$ is a \emph{log involution} of $(X,\Delta)$. In this case, if no component of $X$ belongs to the pointwise fixed locus of $\varphi$, we say that $\varphi$ is \emph{generically fixed point free}.

We follow the standard terminology of \cite[\S 2.1]{Kollar_Singularities_of_the_minimal_model_program} for the birational geometry of pairs $(X,\Delta)$. In particular, we refer the reader to \emph{loc. cit.} for the notions of \emph{discrepancy} of divisors, and for those of \emph{log canonical} (\emph{lc}), \emph{Kawamata log terminal} (\emph{klt}), \emph{divisorial log terminal} (\emph{dlt}), \emph{canonical} and \emph{terminal} pairs. In these cases, $X$ is actually normal. The definition of \emph{semi-log canonical pair} (\emph{slc}) will be recalled in the next section. All these notions are preserved by \'{e}tale base-changes \cite[2.15]{Kollar_Singularities_of_the_minimal_model_program} (the slc case follows from the lc case, since normalization commutes with \'{e}tale base-change).

In dimension two, there is a (seemingly) more general notion of \emph{numerically lc surface pair}, defined in \cite[2.27]{Kollar_Singularities_of_the_minimal_model_program}.

Let $(X,\Delta)$ be a pair. An \emph{lc center} of $(X,\Delta)$ is a closed subset $Z\subset X$ that is the image of a divisor of discrepancy $-1$ with respect to $(X,\Delta)$. More precisely, there exists a proper birational morphism from a normal variety $f\colon Y\to X$ and a prime divisor $E\subset Y$ such that $a(E;X,\Delta)=-1$ and $f(E)=Z$. It is convenient to allow $X$ itself to be considered as a lc center.

Let $(X,\Delta+D)$ be a pair, where $D$ is a reduced divisor with normalisation $D^n$. Then there is a canonically-defined $\bQ$-divisor $\Diff_{D^n}\Delta$ on $D^n$ such that restriction on $D^n$ induces an isomorphism $\omega_X^{[m]}(m\Delta+mD)|_{D^n}\cong \omega_{D^n}(m\Diff_{D^n}\Delta)$ for $m$ divisible enough. Singularities of $(X,\Delta+D)$ along $D$ and singularities of $(D^n,\Diff_{D^n}\Delta)$ are related by so-called adjunction theorems. We refer to \cite[\S 4.1]{Kollar_Singularities_of_the_minimal_model_program} for fundamental theorems of adjunction theory.

Let $X$ be a variety and $C\subset X$ a proper curve. For a Cartier divisor $L$ on $X$, we define the intersection number $L\cdot_k C=\deg_k\sO_C(L|_C)$. By linearity, this definition extends to $\bQ$-Cartier divisors.

Linear equivalence of Weil divisors is denoted by $D\sim D'$, and $\bQ$-linear equivalence of $\bQ$-Weil divisors is denoted by $D\sim_{\bQ}D'$. If $X$ is endowed with a dominant morphism $f\colon X\to S$, we write $D\sim_{\bQ,f}D'$ if there is a $\bQ$-Cartier divisor $E$ on $S$ such that $D\sim_{\bQ}D'+f^*E$.

An \'{e}tale morphism of pointed schemes $(Y,y)\to (X,x)$ is called \emph{elementary} if it induces an isomorphism $k(y)\cong k(x)$ (here we follow the convention of \cite[02LE]{Stacks_Project}).

\subsection{Quotients by finite equivalence relations}\label{section:Quotients_by_finite_equivalence_relations}
The theory of quotients by finite equivalence relations is developed in \cite{Kollar_Quotients_by_finite_equivalence_relations} and \cite[\S 9]{Kollar_Singularities_of_the_minimal_model_program}. For convenience, we recall the basic definitions and constructions that we will need.

Let $S$ be a base scheme, and $X,R$ two reduced $S$-schemes. An $S$-morphism $\sigma=(\sigma_1,\sigma_2)\colon R\to X\times_S X$ is a \textbf{set theoretic equivalence relation} if, for every geometric point $\Spec K\to S$, the induced map
		$$\sigma(\Spec K)\colon \Hom_S(\Spec K,R)\to \Hom_S(\Spec K,X)\times \Hom_S(\Spec K,X)$$
	is injective and an equivalence relation of the set $\Hom_S(\Spec K,X)$. We say in addition that $\sigma\colon R\to X\times_S X$ is \textbf{finite} if both $\sigma_i\colon R\to X$ are finite morphisms.
	
	Assume that $G$ is a groupoid, that is a category where all the arrows are isomorphisms. An action of $G$ on $X$ over $S$ is a functor $F\colon G\to *_{\Aut_SX}$, where the target is the groupoid with one element induced by the abstract group $\Aut_SX$. Given such an action, for each $g\in \arr(G)$ we let $\Gamma(g)\subset X\times_S X$ be the graph of the $S$-automorphism $F(g)$. Then the union $\bigcup_g\Gamma(g)\subset X\times_S X$ is a set theoretic equivalence relation. Conversely, a set theoretic equivalence relation $R\subset X\times_S X$ is called a \textbf{groupoid} if it is of this form.
	
	Suppose that $\sigma\colon R\hookrightarrow X\times_S X$ is a reduced closed subscheme. Then there is a minimal set theoretic equivalence relation generated by $R$: see \cite[9.3]{Kollar_Singularities_of_the_minimal_model_program}. Even if both $\sigma_i\colon R\to X$ are finite morphisms, the resulting relation may not be finite: achieving transitivity can create infinite equivalence classes. 
	
	An important case we will discuss is the following: let $X$ be a variety over a field $k$, $D\subset X$ a reduced divisor with normalization $n\colon D^n\to D$ and $\tau\colon D^n\cong D^n$ an involution. The \textbf{equivalence relation induced by $\tau$} is the smallest set theoretic equivalence relation $R(\tau)\to X\times_k X$ induced by the closure of the set of those $(x,y)\in X\times_k X$ such that
			$$\exists x',y'\in D^n \text{ such that } n(x')=x,\ n(y')=y,\ \tau(x')=y'.$$
A central task will be to identify conditions on $(X,D,\tau)$ that guarantee that $R(\tau)$ is a finite equivalence relation.

	Let $(\sigma_1,\sigma_2)\colon R\to X\times_S X$ be a finite set theoretic equivalence relation. A \textbf{geometric quotient} of this relation is an $S$-morphism $q\colon X\to Y$ such that
		\begin{enumerate}
			\item $q\circ \sigma_1=q\circ \sigma_2$,
			\item $(Y,q\colon X\to Y)$ is initial in the category of algebraic spaces for the property above: if $p\colon X\to Z$ is such that $p\circ \sigma_1=p\circ\sigma_2$ there exists a unique $\phi\colon Y\to Z$ such that $p=\phi\circ q$; and
			\item $q$ is finite.
		\end{enumerate}
Clearly the quotient $(Y,q)$ is unique (up to unique isomorphism) if it exists. It may happen that $Y$ exists only as an algebraic space. 
		
The most important result for us is that quotients by finite equivalence relation usually exist in positive characteristic: if $X$ is essentially of finite type over a field $k$ of positive characteristic and $R\rightrightarrows X$ is a finite set theoretic equivalence relation, then the geometric quotient $X/R$ exists and is a $k$-scheme \cite[Theorem 6, Corollary 48]{Kollar_Quotients_by_finite_equivalence_relations}.

\subsection{Demi-normal schemes}\label{section: demi-normal schemes}
Let us recall the definition of a node, following \cite[1.41]{Kollar_Singularities_of_the_minimal_model_program}.

\begin{definition}\label{definition: node}
A one-dimensional Noetherian local ring $(R,\fm)$ is called a \textbf{node} if there exists a ring isomorphism $R\cong S/(f)$, where $(S,\fn)$ is a regular two-dimensional local ring and $f\in \fn^2$ is an element that is not a square in $\fn^2/\fn^3$.
\end{definition}

\begin{definition}
A locally Noetherian reduced scheme (or ring) is called \textbf{nodal} if its codimension one local rings are regular or nodal. It is called \textbf{demi-normal} if it is $S_2$ and nodal.
\end{definition}

Let $X$ be a reduced scheme with normalization $\pi\colon \bar{X}\to X$. The conductor ideal of the normalization is defined as
		$$\fI:=\im\left(\ev_1\colon \sHom_X(\pi_*\sO_{\bar{X}},\sO_X)\\longrightarrow\sO_X\right).$$
It is an ideal in both $\sO_X$ and $\sO_{\bar{X}}$. We let
		$$D:=\Spec_X\sO_X/\fI,\quad \bar{D}:=\Spec_{\bar{X}}\sO_{\bar{X}}/\fI,$$
and call them the \textit{conductor subschemes}.

\begin{lemma}\label{lemma: basic properties of normalization of demi-normal scheme}
Notations as above. Assume that $X$ is essentially of finite type over a field. Then:
	\begin{enumerate}
		\item $D$ and $\bar{D}$ are reduced of pure codimension $1$.
		\item If $\eta\in D$ is a generic point such that $\Char k(\eta)\neq 2$, the morphism $\bar{D}\to D$ is \'{e}tale of degree $2$ in a neighborhood of $\eta$. If $\Char k(\eta)=2$, then $\bar{D}\to D$ might be purely inseparable (see \autoref{example: inseparable nodes}).
		\item $X$ has a dualizing sheaf which is invertible in codimension one. In particular, it has a well-defined canonical divisor $K_X$.
		\item Let $\Delta$ be a $\bQ$-divisor on $X$ with no component supported on $D$, and such that $K_X+\Delta$ is $\bQ$-Cartier. If $\bar{\Delta}$ denotes the divisorial part of $\pi^{-1}(\Delta)$, then there is a canonical isomorphism
				$$\pi^*\omega_X^{[m]}(m\Delta)\cong \omega_{\bar{X}}^{[m]}(m\bar{D}+m\bar{\Delta})$$
		for $m$ divisible enough.
	\end{enumerate}
\end{lemma}
\begin{proof}
Since $X$ is essentially of finite type over a field, it has a dualizing sheaf $\omega_X$ and \autoref{lemma: commutative algebra definition of a node} below shows that it is invertible in codimension one. The rest follows from \cite[5.2,5.7]{Kollar_Singularities_of_the_minimal_model_program}.
\end{proof}

\begin{definition}
We say that $(X,\Delta)$ is an \textbf{slc pair} if: $X$ is demi-normal, $\Delta$ is a $\bQ$-divisor with no components along $D$, $K_X+\Delta$ is $\bQ$-Cartier, and the normalization $(\bar{X}, \bar{D}+\bar{\Delta})$ is an lc pair.
\end{definition}

In most cases the conductor $\bar{D}\subset \bar{X}$ comes with an additional structure that remembers the morphism $\bar{D}\to D$:

\begin{lemma}\label{lemma:normalization_gives_involution}
Let $X$ be a demi-normal scheme with normalization $\pi\colon \bar{X}\to X$. Assume that the morphism of conductors $\bar{D}\to D$ is \'{e}tale over every generic point. Then:
	\begin{enumerate}
		\item the induced morphism of normalizations $\bar{D}^n\to D^n$ is the geometric quotient by a Galois involution $\tau$;
		\item if $K_X+\Delta$ is $\bQ$-Cartier, then $\tau$ is a log involution of $(\bar{D}^n,\Diff_{\bar{D}^n}\bar{\Delta})$.
	\end{enumerate}
\end{lemma}
\begin{proof}
Since $\bar{D}\to D$ is generically \'{e}tale, the field extension $k(D)\subset k(\bar{D})$ is separable of degree $2$, hence Galois. The non-trivial field automorphism gives a well-defined morphism $\tau$ on $\bar{D}^n$, and it is easy to see that $D^n=\bar{D}^n/\tau$ (see also \autoref{corollary:Crepant_map_induces_crepant_map_on_lc_centers} below). This proves the first point. 

Assume that $K_X+\Delta$ is Cartier. Since $\tau\colon \bar{D}^n\to \bar{D}^n$ commutes with the projection to $D$, the pullback of $K_X+\Delta$ to $\bar{D}^n$ is $\tau$-invariant. We conclude using \autoref{lemma: basic properties of normalization of demi-normal scheme} and adjunction.
\end{proof}

\subsection{Introduction to Koll\'{a}r's gluing theory}\label{section:Kollar_gluing_theory}
We give a short introduction to Koll\'{a}r's gluing theory. Let $X$ be a demi-normal variety defined over a field $k$ of arbitrary characteristic, with normalization $\pi\colon \bar{X}\to X$. Let $D\subset X$ and $\bar{D}\subset \bar{X}$ be as in \autoref{section: demi-normal schemes}. If $\Char k\neq 2$, we have seen in \autoref{lemma:normalization_gives_involution} that $\bar{D}^n$ comes with an involution $\tau$ and that the different $\Diff_{\bar{D}^n}(0)$ is $\tau$-invariant. Hence normalization gives a map
		\begin{equation}\label{eqn: gluing equivalence I}
		(\Char k\neq 2) \quad
		\begin{pmatrix}
		X \text{ demi-normal and}\\
		\text{proper over }k
		\end{pmatrix}
		\longrightarrow
		\begin{pmatrix}
		\text{Proper normal pair }(\bar{X},\bar{D})\text{ with}\\
		\text{an involution }\tau\text{ of } (\bar{D}^n,\Diff_{\bar{D}^n}(0))
		\end{pmatrix}
		\end{equation}
Koll\'{a}r's gluing theory constructs, under additional assumptions, an inverse map to \autoref{eqn: gluing equivalence I}. More precisely, Koll\'{a}r shows \cite[5.13]{Kollar_Singularities_of_the_minimal_model_program} that normalization of demi-normal proper varieties induces a bijection	
	\begin{equation}\label{eqn: gluing equivalence II}
		(\Char k = 0) \quad
		\begin{pmatrix}
		\text{Proper slc pairs}\\
		(X,\Delta) \text{ such that}\\
		K_X+\Delta\text{ is ample}
		\end{pmatrix}
		\overset{1:1}{\longrightarrow}
		\begin{pmatrix}
		\text{Proper lc pairs } (\bar{X},\bar{D}+\bar{\Delta}) \text{ plus} \\
		\text{a generically fixed point free}\\
		\text{involution }\tau\text{ of }(\bar{D}^n,\Diff_{\bar{D}^n}\bar{\Delta}) \\
		\text{such that }K_{\bar{X}}+\bar{D}+\bar{\Delta} \text{ is ample.}
		\end{pmatrix}
		\end{equation}
We sketch how the inverse map is constructed. Let $(\bar{X},\bar{D},\tau)$ be a triplet such as in the right-hand side of \autoref{eqn: gluing equivalence II} (we take $\bar{\Delta}=0$ for simplicity). There are two questions to solve: how to construct $X$, and how to show that $K_X$ is $\bQ$-Cartier.
	\begin{enumerate}
		\item \textbf{Construction of $X$.} Similarly to a nodal curve that can be reconstructed from its normalization by gluing two points, the variety $X$ should be obtained by identifying the points of $\bar{D}$ that are conjugate under $\tau$. Since the normalization $\bar{X}\to X$ ought to be finite, the induced equivalence relation $R(\tau)\subset \bar{X}\times \bar{X}$ should be finite. However $\tau$ is defined on $\bar{D}^n$ and it is not clear that the equivalence classes on $\bar{D}\subset \bar{X}$ are finite. Indeed, they need not be if one drops the hypothesis that $(\bar{X},\bar{D})$ is log canonical or that $\tau$ respects the log canonical stratification of $(\bar{D}^n,\Diff_{\bar{D}^n}0)$ (see \cite[5.17]{Kollar_Singularities_of_the_minimal_model_program}). 
		
		Koll\'{a}r's solution is to take advantage of the log canonical stratification and to proceed by induction on the lc centers of $(\bar{X},\bar{D})$. While the singularities of the lc centers $Z\subset \bar{X}$ of high codimension can be complicated, the picture is more transparent if we take a dlt crepant blow-up $(Y,D_Y)\to (\bar{X},\bar{D})$. Then the lc centers of $(\bar{X},\bar{D})$ are images of the strata of $D_Y^{=1}$, and one shows \cite[4.45]{Kollar_Singularities_of_the_minimal_model_program} that the crepant birational class of a minimal stratum $(S,D_S)$ above a fixed lc center $Z$ of $(\bar{X},\bar{D})$, is independant of the choice of the resolution $Y$ and of the choice of $S$. Then one observes \cite[5.36--37]{Kollar_Singularities_of_the_minimal_model_program} that the equivalence classes generated by $\tau$ on $Z$ are governed by the group of birational crepant self-maps of $(S,D_S)$, more precisely by the representation of this group on the global sections of the invertible sheaf associated to $K_S+\Delta_S$. Then one uses the theory of pluricanonical representations \cite[\S 10.5]{Kollar_Singularities_of_the_minimal_model_program} to obtain finiteness.
		
		The relation between $(S,D_S)$ and the Stein factorization $Z_S$ of $S\to Z$ is subtle. One can think of $(S,D_S)$ as a higher codimension version of adjunction for divisors. Koll\'{a}r coined the term \emph{source} for the crepant birational class of $(S,D_S)$ and the term \emph{spring} for the isomorphism class of $Z_S$, and studied them extensively in \cite{Kollar_Singularities_of_the_minimal_model_program}. 
		
		In characteristic $0$, finiteness of $R(\tau)$ is not sufficient to guarantee the existence of the quotient $\bar{X}/R(\tau)$. However it is sufficient in positive characteristic \cite[Theorem 6]{Kollar_Quotients_by_finite_equivalence_relations}, so we will not elaborate on this issue.

		\item \textbf{Descent of $K_{\bar{X}}+\bar{D}$.} Once $X$ is constructed, we would like to descend the $\bQ$-Cartier divisor $K_{\bar{X}}+\bar{D}$. Koll\'{a}r's strategy is to descend the total space $T$ of $K_{\bar{X}}+\bar{D}$. Indeed, the equivalence relation $R(\tau)$ lifts to an equivalence relation $R_T$ on $T$, which is shown to be finite using a strategy similar to the one explained above (see \cite[5.38]{Kollar_Singularities_of_the_minimal_model_program}). 
		
	\end{enumerate}

\section{Nodes}
In this section, we study nodal singularities, with a particular emphasis on nodes in characteristic $2$.

\subsection{Normalization of nodes}\label{section:normalization_of_nodes}
We take a look at the normalization of nodes, following \cite[\S 3]{Tanaka_Abundance_for_slc_surfaces}.
Let $(R,\fm)$ be a nodal Noetherian local ring, and pick a presentation $R\cong S/(f)$ as in \autoref{definition: node}. We can choose a set of local parameters $\fn=(x,y)$ for $S$ such that
		$$f=ax^2+bxy+cy^2+g$$
for $a,b\in S$, $c\in S^\times$ and $g\in \fn^3$ \cite[3.2, 3.3]{
Tanaka_Abundance_for_slc_surfaces}. Denote by $\bar{x},\bar{y}$ the images of $x,y$ in $R$.

\begin{proposition}\label{proposition: normalization of nodes}
Notations as above. Assume that $(R,\fm)$ is not an integral domain. Then:
	\begin{enumerate}
		\item the normalization of $R$ is a product $T_1\times T_2$ of local regular one-dimensional rings;
		\item $\fm$ is the conductor of the normalization;
		\item $R\hookrightarrow T_1\times T_2$ induces the diagonal embedding
				$$k(\fm)\hookrightarrow \left( T_1/\fm T_1\right) \times \left(T_2/\fm T_2\right)\cong k(\fm)\oplus k(\fm).$$
	\end{enumerate}
Assume now that $(R,\fm)$ is an integral domain. Then:
	\begin{enumerate}
		\item the normalization of $(R,\fm)$ is given by
				$$R\hookrightarrow R\left[\frac{\bar{y}}{\bar{x}}\right]=:T.$$
				Moreover we have $T=R+R\cdot \frac{\bar{y}}{\bar{x}}$.
		\item the conductor of the normalization is $\fm$, and $\fm T=\bar{x}T$.
		\item We have
				$$T/\fm T\cong \frac{k(\fm)[Z] }{(\bar{a}+\bar{b}Z+\bar{c}Z^2)}$$
				where $\bar{a},\bar{b},\bar{c}$ are the image of $a,b,c$ through $S\to k(\fm)=R/\fm$, and $Z$ is the image of $\frac{\bar{y}}{\bar{x}}$. In particular $T$ is local with maximal ideal $\fm T$ as soon as $\bar{a}+\bar{b}Z+\bar{c}Z^2$ is irreducible in $k(\fm)[Z]$.
	\end{enumerate}
\end{proposition}
\begin{proof}
These results are contained in the statements and the proofs of \cite[3.4, 3.5]{Tanaka_Abundance_for_slc_surfaces}.
\end{proof}

\begin{corollary}\label{corollary: topology_of_normalization_of_nodes_II}
Let $(R,\fm)$ be a non-integral nodal ring and $T$ be its normalization. Denote the conductors by $D:=V(\fm)\subset \Spec{R}$ and $\bar{D}:=V(\fm T)\subset\Spec{T}$. Then $\sO_D$ is equal to the fixed sub-ring of an involution on $\sO_{\bar{D}}$.
\end{corollary}
\begin{proof}
By \autoref{proposition: normalization of nodes}, $\sO_D\hookrightarrow \sO_{\bar{D}}\cong \sO_D\oplus\sO_D$ is the diagonal embedding, so the involution is given by the permutation of direct summands.
\end{proof}

\begin{corollary}\label{corollary: topology of normalization of nodes}
Let $(R,\fm)$ be an integral nodal ring and $T$ be its normalization. Denote the conductors by $D:=V(\fm)\subset \Spec{R}$ and $\bar{D}:=V(\fm T)\subset\Spec{T}$. Then:
	\begin{enumerate}
		\item $\bar{D}\to D$ is two-to-one if and only if $\bar{a}+\bar{b}Z+\bar{c}Z^2$ splits in $k(\fm)[Z]$.
		\item $\bar{D}\to D$ is bijective and Galois if and only if $\bar{a}+\bar{b}Z+\bar{c}Z^2$ is irreducible and separable over $k(\fm)$.
		\item\label{condition: inseparable node} $\bar{D}\to D$ is bijective and purely inseparable if and only if $\bar{a}+\bar{b}Z+\bar{c}Z^2$ is irreducible and inseparable over $k(\fm)$.
	\end{enumerate}
Moreover exactly one of these cases occurs. In the first two cases, $\sO_D$ is recovered as the fixed sub-ring of an involution on $\sO_{\bar{D}}$.
\end{corollary}
\begin{proof}
We need to show that $\bar{a}+\bar{b}Z+\bar{c}Z^2$ cannot be a square in $k(\fm)[Z]$. If this was the case, then $\sO_{\bar{D}}=T/\fm T$ would not be reduced, which contradicts \autoref{lemma: basic properties of normalization of demi-normal scheme}.

In the first case we have $T/\fm T\cong k(\fm)\oplus k(\fm)$ in which $k(\fm)$ embeds diagonally, and the involution is given by exchanging the direct summands.

In the second case $T/\fm T$ is a Galois field extension of $k(\fm)$ of degree $2$, and the involution is the unique non-trivial element of the Galois group.
\end{proof}

\begin{definition}\label{definition:(in)separable_node}
A nodal Noetherian local ring $(R,\fm)$ is an \textbf{inseparable node} if the condition of \autoref{corollary: topology of normalization of nodes}.\autoref{condition: inseparable node} is satisfied. In the other cases, we call it a \textbf{separable node}.
\end{definition}

\begin{remark}\label{remark:no_curve_with_inseparable_node_over_perfect_field}
Keeping the notations of \autoref{proposition: normalization of nodes}, $(R,\fm)$ is an inseparable node if and only if $R$ is a domain, $k(\fm)$ has characteristic $2$, $\bar{b}=0$ and $\bar{a}/\bar{c}\notin k(\fm)^2$. In particular the residue field of an inseparable node is not perfect.
\end{remark}

We study the pluricanonical sections that descend the normalization morphism, aiming to give a slight generalization of \cite[5.8, 5.18]{Kollar_Singularities_of_the_minimal_model_program}.

\begin{lemma}\label{lemma:trace_map}
Notations as in \autoref{corollary: topology_of_normalization_of_nodes_II} and \autoref{corollary: topology of normalization of nodes}. Then:
	\begin{enumerate}
		\item The kernel of the Grothendieck trace map $\omega_{\bar{D}}\to\omega_D$ is one-dimensional over $k(\fm)$.
		\item Assume $\bar{D}\to D$ is separable with induced involution $\tau$ on $\bar{D}$. Then the kernel of the Grothendieck trace map is the sub-module of $\tau$-anti-invariant sections. (If $\Char k(\fm)=2$, we interpret $\tau$-anti-invariant sections as the $\tau$-invariant ones.)
	\end{enumerate}
\end{lemma}
\begin{proof}
In any case, since we can take $\sO_D$ and $\sO_{\bar{D}}$ as the dualizing sheaves of $D$ and $\bar{D}$, the Grothendieck trace map is obtained by applying $\Hom_{\sO_D}(\bullet,\sO_D)$ to the map $\iota\colon \sO_D\to \sO_{\bar{D}}$. We have to show that the kernel of the induced map $\Hom(\iota)$ is one-dimensional over $k$, and that it is the sub-vector space of $\tau$-invariant sections when the involution $\tau$ exists. We distinguish the three possible cases given by \autoref{corollary: topology_of_normalization_of_nodes_II} and \autoref{corollary: topology of normalization of nodes}, and prove our assertions for each of them. For ease of notation, we write $k:=k(\fm)=\sO_D$.
	\begin{enumerate}
		\item[(i)] The morphism $\bar{D}\to D$ is two-to-one. Then $\sO_{\bar{D}}=k\oplus k$, the involution $\tau$ exchanges the two summands. The morphism $\bar{D}\to D$ corresponds to the diagonal embedding $\iota\colon k\hookrightarrow k\oplus k$. The Grothendieck trace is given by
				$$\Hom(\iota)\colon \Hom_k(k\oplus k,k)\to \Hom_k(k,k), \quad \phi_{a,b}\mapsto \mu_{a+b}$$
		where $\phi_{a,b}(x,y)=ax+by$ and $\mu_{a+b}(z)=(a+b)z$. Thus $\ker \Hom(\iota)=\{\phi_{a,-a}\mid a\in k\}$ is one-dimensional over $k$. Since $\tau$ acts on $\Hom_k(k\oplus k,k)$ by pre-composition, we see that $\ker\Hom(\iota)$ is the sub-module of $\tau$-anti-invariant sections.
		\item[(ii)] The map $\bar{D}\to D$ is bijective and Galois. Then $\iota\colon k\to L:=\sO_{\bar{D}}$ is a Galois field extension of degree $2$, and $\tau$ is the non-trivial element of $\Gal(L/k)$. As above, the Grothendieck trace map is given by
				$$\Hom(\iota)\colon \Hom_k(L,k)\to \Hom_k(k,k),\quad \phi\mapsto \phi\circ \iota.$$	
	We distinguish two cases:
		\begin{enumerate}
			\item[(ii.1)] If $\Char k\neq 2$, then we can find an element $d\in k$ such that $L=k(\sqrt{d})$. Taking $\{1,\sqrt{d}\}$ as a $k$-basis of $L$, we see that $\ker\Hom(\iota)=\{\phi_{0,a}\mid a\in k\}$. Since $\tau(\sqrt{d})=-\sqrt{d}$, we see that $\ker\Hom(\iota)$ is the sub-module of $\tau$-anti-invariant elements.
			\item[(ii.2)] Assume that $\Char k=2$ and consider the $k$-linear map 
			$$\mathfrak{t}\colon L\to k,\quad \mathfrak{t}(x)=x+\tau(x).$$ 
		Since $L^\tau=k$, it holds $k\cdot \mathfrak{t}\subseteq \ker\Hom(\iota)$. Since $\Hom(\iota)$ is surjective and $\dim_kL=2$, we have $\dim_k\ker\Hom(\iota)=1$ and therefore $k\cdot \mathfrak{t}= \ker\Hom(\iota)$.
			
			Since $\mathfrak{t}$ is $\tau$-invariant, we obtain that $\ker\Hom(\iota)$ is included in the $\tau$-invariant sub-module of $\Hom_k(L,k)$. By counting dimensions, it remains to show that there exists elements of $\Hom_k(L,k)$ that are not $\tau$-invariant.
			
			Let us exhibit such linear maps. We can write $L=k(\alpha)$ for some $\alpha\in L$. If $\tau(\alpha)=c\alpha$ with $c\in k$, then $(1+c)\alpha\in k$, so either $c=1$ and $\tau$ is the identity, either $\alpha\in k$. Neither are possible, so $\{\alpha,\tau(\alpha)\}$ is a $k$-basis of $L$. In this basis, if $a,b$ are distinct elements of $k$, the $k$-linear map $\phi_{a,b}$ is not $\tau$-invariant.
		\end{enumerate}		
				\item[(iii)] The morphism $\bar{D}\to \bar{D}$ is purely inseparable. Then $\iota\colon k\to L:=\sO_{\bar{D}}$ is a purely inseparable field extension of degree $2$. Thus we can write $L=k(\sqrt{d})$ for some $d\in k\setminus k^2$. Taking $\{1,\sqrt{d}\}$ as a $k$-basis of $L$, we see that $\ker\Hom(\iota)=\{\phi_{0,a}\mid a\in k\}$, which is one-dimensional over $k$.
	\end{enumerate}
The proof is complete.
\end{proof}

\begin{proposition}\label{proposition:descend_of_sections_along_normalization}
Let $X$ be an excellent demi-normal scheme with only separable nodes and normalization $(\bar{X},\bar{D}, \tau)$, and $\Delta$ a $\bQ$-divisor on $X$ that has no common component with $\Sing(X)$. Then a section of $\omega_{\bar{X}}^{[m]}(m\bar{D}+m\bar{\Delta})$ descends to $\omega_X^{[m]}(m\Delta)$ if and only if its Poincaré residue at generic points of $\bar{D}$, taking values in $\omega_{\bar{D}}^{[m]}(m\Diff_{\bar{D}}\bar{\Delta})$, is $\tau$-invariant and $m$ is even, or is $\tau$-anti-invariant and $m$ is odd.
\end{proposition}
\begin{proof}
We are dealing with reflexive sheaves, and the morphism $\pi\colon \bar{X}\to X$ is an isomorphism above $X\setminus D$. Thus the only question is above the generic points of $D$. In particular we may assume that $\Delta=0$, that $X$ is Cohen--Macaulay, semi-normal (see \autoref{remark:semi_normality_and_demi_normality}) and that $\bar{X}$ is regular. Then according to \cite[5.9]{Kollar_Singularities_of_the_minimal_model_program}, there is an exact sequence
		$$0\to \omega_X\to \pi_*\omega_{\bar{X}}(\bar{D})\overset{\partial}{\to} \omega_D\to 0$$
	with 
			$$\partial=\left(\omega_{\bar{X}}(\bar{D})\overset{\sR}{\longrightarrow} \omega_{\bar{D}}\overset{\Tr}{\longrightarrow}\omega_D\right)$$
	where $\sR$ is the Poincaré residue map, and $\Tr$ the Grothendieck trace map. Hence for $m=1$ the result follows from \autoref{lemma:trace_map}. Under our current assumptions, both $\omega_X$ and $\omega_{\bar{X}}(\bar{D})$ are invertible, so the cases $m>1$ follow by taking tensor powers.
\end{proof}

\begin{corollary}\label{corollary:descent_in_codim_2}
Let $(X,\Delta)$ be as in \autoref{proposition:descend_of_sections_along_normalization}, with normalization $(\bar{X},\bar{D}+\bar{\Delta},\tau)$. Then the following are equivalent:
	\begin{enumerate}
		\item $\Diff_{\bar{D}^n}\bar{\Delta}$ is $\tau$-invariant, and
		\item $(X,\Delta)$ is slc outside a closed subset of codimension at least $3$.
	\end{enumerate}	 
\end{corollary}
\begin{proof}
The only difference between our statement and the statement of \cite[5.18]{Kollar_Singularities_of_the_minimal_model_program} is that we do not assume that $2\in \sO_X$ is invertible. This assumption is only needed in order to use \cite[5.8]{Kollar_Singularities_of_the_minimal_model_program}, which is generalized to our setting by \autoref{proposition:descend_of_sections_along_normalization}.
\end{proof}

Next we study inseparable nodes.

\begin{example}[Whitney's umbrella in characteristic $2$]\label{example: inseparable nodes}
Let $k$ be any field of characteristic $2$. Consider the ring $A:=k[u,v,w]/(wv^2-u^2)$. We claim that $A$ has only nodal singularities in codimension one. Indeed, the singular locus of $A$ is defined by $(u=v=0)$, so the prime ideal $\fp:=(u,v)$ is the only height one prime such that $A_\fp$ is not regular. Now
		$$A_{\fp}\cong \frac{k[u,v,w]_{(u,v)}}{(wv^2-u^2)}$$
is nodal according to \autoref{definition: node} since $wv^2-u^2$ does not have a square root modulo $\fp^3A_\fp$. 

The normalization of $A$ is given by the regular two-dimensional ring $B:=k\left[t=\frac{u}{v},v\right]$. Hence the normalization of $A_\fp$ is given by $B_{(v,tv)}=B_{\bar{\fp}}$, with $\bar{\fp}=vB$. The conductor of the normalization is generated by the element $v$.

Denote $D=V(\fp)\subset \Spec{A}$ and $\bar{D}=V(\bar{\fp})\subset\Spec{B}$. Then $\pi\colon \Spec{B}\to \Spec{A}$ is an isomorphism above the complement of $D$. Moreover $\bar{D}\to D$ is given by the $k$-algebra morphism
		$$A/\fp=k[w]\longrightarrow k[t]=B/\bar{\fp}, \quad w\mapsto t^2.$$
Since $k$ has characteristic $2$, we obtain that $A_\fp$ is an inseparable node.

In view of the gluing theory we want to develop, we should answer the following question: how to reconstruct the nodal surface $A$ from its normalization $B$? Looking a the commutative diagram
		$$\begin{tikzcd}
		A\arrow[d]\arrow{rr}{v\mapsto v}[swap]{u\mapsto tv,\; w\mapsto t^2} && B\arrow[d] \\
		A_\fp \arrow[d]\arrow[rr]&& B_{\bar{\fp}}\arrow[d] \\
		k(A_\fp)=k(w)\arrow[rr, "w\mapsto t^2"] && k(t)=k(B_{\bar{\fp}})
		\end{tikzcd}$$
	we see that $A$ is the preimage of $k(t^2)\subset k(t)$ under the canonical map $B\to k(B_{\bar{\fp}})$. Hence the data of $A\subset B$ is equivalent to the data of $B$ together with the degree $2$ purely inseparable extension $k(w)\subset k(t)$.
\end{example}

\begin{example}
The calculations of \autoref{example: inseparable nodes} can be generalized to the case of $A':=k[u,v,w]_{(u,v)}/(f(w)v^2-u^2)$, where $f(w)\in k(w)$ is not a square. In this case the data of $A'$ is also recovered from its normalization $B'=k[t=u/v, v]_{(v)}$ and the purely inseparable degree $2$ extension
		$$k(w)\hookrightarrow k(t)\cong k\left(\sqrt{f(w)}\right), \quad w\mapsto t^2=f(w).$$
\end{example}

While separable nodes can be reconstructed from their normalization using the induced involution, the following proposition shows that inseparable nodes can be reconstructed from their normalization using a method generalizing \autoref{example: inseparable nodes}.
\begin{proposition}\label{proposition: reconstructing inseparable nodes from normalizations}
Let $(R,\fm)$ be a Noetherian local ring which is an inseparable node. Let $(T, \fm_T=\fm T)$ be its normalization. Then $R$ is the preimage in $T$ of the subfield $k(\fm)\hookrightarrow k(\fm_T)$.
\end{proposition}
\begin{proof}
We use the notations and results of \autoref{proposition: normalization of nodes}. Then $T=R[\bar{y}/\bar{x}]=R+R\cdot \bar{y}/\bar{x}$, and $k(\fm)\hookrightarrow k(\fm_T)$ is given explicitly by the inclusion of constants
		$$k(\fm)\subset \frac{k(\fm)[Z]}{(\bar{a}+\bar{c}Z^2)}=k(\fm_T),$$
	where $Z$ is the image of $\bar{y}/\bar{x}$. An element $\alpha+\beta \cdot \bar{y}/\bar{x}\in T$, with $\alpha,\beta \in R$, reduces modulo $\fm_T$ to an element of $k(\fm)$ if and only if $\beta\in\fm_T= \fm T$. But $\fm T=\bar{x} T$, so we see that $R$ is equal to the preimage of $k(\fm)$ in $T$.
\end{proof}

The normalization morphism of a demi-normal variety has a particular structure. We record it here, although the proof uses results from the next sections.
\begin{proposition}\label{proposition:structure_of_normalization_of_demi_normal_scheme}
Let $X$ be a demi-normal variety over a field $k$ of positive characteristic, and $\pi\colon \bar{X}\to X$ its normalization. Then we have a factorization
		$$\pi=\left( \bar{X}\overset{\nu}{\longrightarrow}\tilde{X}\overset{F}{\longrightarrow} X\right)$$
where
	\begin{enumerate}
		\item $\nu$ is the geometric quotient of $\bar{X}$ by the finite set-theoretic equivalence relation induced by the separable nodes of $X$;
		\item $F$ is a purely inseparable morphism factorizing the $k$-Frobenius of $X$; it induced by the inseparable nodes of $X$ (see \autoref{section:inseparable_nodes}), and $F=\id$ if $\Char k\neq 2$.
	\end{enumerate}
\end{proposition}
\begin{proof}
Let $D=D_G+D_I\subset X$ be the conductor where $D_G$ is the divisor corresponding to separable nodes and $D_I$ is the divisor corresponding to inseparable nodes. Write accordingly $\bar{D}=\bar{D}_G+\bar{D}_I$ the conductor in $\bar{X}$. By \autoref{lemma:normalization_gives_involution} the morphism $\bar{D}_G^n\to \bar{D}_G$ is Galois with involution $\tau$. Since $\tau$ arises from a finite normalization, the equivalence relation $R(\tau)\rightrightarrows \bar{X}$ is finite, so the quotient $\nu\colon \bar{X}\to \bar{X}/R(\tau)=:\tilde{X}$ exists. The scheme $\tilde{X}$ is a demi-normal variety with normalization $\bar{X}$ (see \autoref{prop:finite_equivalence_relation_give_nodes}). 

The purely inseparable morphism $\bar{D}_I\to D_I$ gives a collection of degree $2$ purely inseparable field extensions $k(\eta_i)\subset k(\bar{\eta}_i)$, where $\eta_i$ runs through the generic points of $D_I$ and $\bar{\eta}_i$ through the corresponding generic points of $\bar{D}_I$. Since $\nu$ is an isomorphism at the generic points of $\bar{D}_I$, we may apply \autoref{construction: construction of inseparable node} to these field extensions to obtain the purely inseparable $F\colon \tilde{X}\to X'$. Notice that $\sO_{X'}\subset\sO_{\tilde{X}}$ is an equality at every codimension one point of $X'$ that is not a generic point of $\nu(\bar{D}_I)$.

We claim that $X'=X$. By the universal property of the quotient the morphism $\pi\colon \bar{X}\to X$ factors through $\tilde{X}$, say 
		$$\begin{tikzcd}
		\bar{X} \arrow[rd, "\nu"] \arrow[rr, "\pi"] && X \\
		& \tilde{X}\arrow[ur, "r"] &
		\end{tikzcd}$$
Both $\pi$ and $\nu$ are finite, so by \cite[01YQ]{Stacks_Project} and \cite[Theorem 1]{Artin_Tate_Note_on_finite_extensions} we see that $r$ is also finite. By the proof of \cite[5.3]{Kollar_Singularities_of_the_minimal_model_program} the morphism $r$ is an isomorphism above every codimension one point of $X$ that is not a generic point of $D_I=r(\nu(\bar{D}_I))$. Combining with \autoref{proposition: reconstructing inseparable nodes from normalizations} we obtain that for every codimension one point $\eta$ of $X$, we have $\sO_{X,\eta}=(r_*\sO_{X'})_\eta$ as subrings of $(r_*\sO_{\tilde{X}})_\eta$. Since $\sO_X$ and $r_*\sO_{X'}$ are $S_2$ \cite[5.4]{Kollar_Mori_Birational_geometry_of_algebraic_varieties} we deduce that $\sO_X=r_*\sO_{X'}$. In other words $(r\colon \tilde{X}\to X)=(F\colon \tilde{X}\to X')$.
\end{proof} 

\subsection{Characterization of nodes}
We will need some equivalent formulations of nodal singularities.

\begin{lemma}\label{lemma:definition_node_via_normalization}
Let $(R,\fm,k)$ be a one-dimensional reduced local ring that is a quotient of a regular ring. Let $T$ be its normalization, and $\fn\subset T$ be its Jacobson radical. Then $R$ is nodal if and only if $\dim_k T/\fn=2$ and $\fn\subset R$.
\end{lemma}
\begin{proof}
The only difference with \cite[1.41.1]{Kollar_Singularities_of_the_minimal_model_program} is that we assume that $R$ is the quotient of a regular ring that is not necessarily local. But the two conditions are equivalent: for if $\pi\colon B\twoheadrightarrow A$ is a surjective map of rings and $A$ is local, then $\fp:=\pi^{-1}(\fm_A)$ is a prime ideal, $\pi$ factors through $B_\fp$ by the universal property of localization, and $B_\fp$ is regular if $B$ is regular.
\end{proof}

Another characterization of nodes can be given in terms of semi-normality, which we define below. See \cite[\S 10.2]{Kollar_Singularities_of_the_minimal_model_program} for more details.

\begin{definition}
Let $X$ be an excellent scheme with normalization $\nu\colon X^\nu\to X$. For a point $x\in X$, let $x^\nu:=\nu^{-1}(x)_\text{red}$. The sections $s\in \nu_*\sO_{X^\nu}$ that satisfy
		$$s|_{x^\nu}\in \im [\nu^*\colon k(x)\to H^0(x^\nu,\sO_{x^\nu})] \quad \forall x\in X$$
form a finite $\sO_X$-algebra $\sO'$. We say that $X$ is \textbf{semi-normal} if $\sO'=\sO_X$.
\end{definition}

\begin{remark}\label{remark:semi_normality_and_demi_normality}
The locus where $\sO_X\to \sO'$ is an equality, is open: thus semi-normality is an open condition on excellent schemes. We mention that demi-normality implies semi-normality: see \cite[\S 2.3.1]{Posva_Gluing_in_mixed_char}. 
\end{remark}

\begin{lemma}\label{lemma: commutative algebra definition of a node}
Let $(R,\fm,k)$ be a one-dimensional excellent reduced semi-normal local ring that is a quotient of a regular ring. Then $R$ is nodal if and only if it is Gorenstein.
\end{lemma}
\begin{proof}
We elaborate \cite[5.9.3]{Kollar_Singularities_of_the_minimal_model_program}. Write $X=\Spec R$ and let $\bar{X}$ be the normalization. Then $\bar{X}$ is a regular scheme and $\bar{X}\to X$ is finite. Let $D\subset X$ and $\bar{D}\subset \bar{X}$ be the conductors subschemes. Then $\Supp D$ is the closed point $x\in X$, and $\Supp\bar{D}$ is its preimage in $\bar{X}$. By \cite[Corollary 1.5.1]{Leahy_Seminormal_rings}, thanks to the semi-normality assumption, the conductor ideal is radical in $\sO_{\bar{X}}$, hence also in $\sO_X$, and therefore $D$ and $\bar{D}$ are reduced.

Since $R$ is $S_1$ and the quotient of a regular ring, $X$ has a canonical sheaf $\omega_X$. The $0$-dimensional reduced schemes $D$ and $\bar{D}$ also have canonical sheaves, which are invertible. The calculations of \cite[5.9.3]{Kollar_Singularities_of_the_minimal_model_program} show that there is an isomorphism of $k(x)$-vector spaces
		$$\omega_X\otimes k(x)\cong \ker [\Tr\colon\omega_{\bar{D}}\to \omega_{D}]$$
where $\Tr$ the Grothendieck trace map.

Suppose that $R$ is nodal. Then by \autoref{lemma:trace_map} we know that $\ker(\Tr)$ is a one-dimensional $k(x)$-vector space. By Nakayama's lemma, it follows that $\omega_X$ is generated by a single element. Since a canonical sheaf has full support, we see that $\omega_X$ is free of rank one, so $R$ is Gorenstein.

Conversely, suppose that $\omega_X$ is invertible. Then $\ker(\Tr)$ is one-dimensional, and as $\bar{D}$ is Gorenstein we obtain that $\dim_{k(x)}\omega_{\bar{D}}=\dim_{k(x)}H^0(\bar{D},\sO_{\bar{D}})=2$. In particular we have the following dichotomy:
	\begin{enumerate}
		\item $\bar{D}$ is a single point. Thus $\bar{X}$ is local, and its maximal ideal is the conductor ideal, which is contained in $R$.
		\item $\bar{D}$ is supported on two distinct closed points, corresponding to two maximal ideals $\fm_1$ and $\fm_2$. Therefore $\fm_1\cap\fm_2=\sqrt{\fm_1\cap \fm_2}$ is the conductor ideal, contained in $R$.
	\end{enumerate}
In both cases we apply \autoref{lemma:definition_node_via_normalization} to see that $R$ is nodal.
\end{proof}


\subsection{Being nodal is an \'{e}tale-local property}\label{section:nodal_is_etale}
For varieties over a field of characteristic different from $2$, one can show that a local ring $\sO_{X,x}$ is nodal if and only if 
		$$\widehat{\sO_{X,x}}\otimes_{k}\bar{k}\cong \bar{k}[[X,Y]]/(XY),$$
where $k=k(x)$ and $\bar{k}$ is an algebraic closure, see for example \cite[1.41.2]{Kollar_Singularities_of_the_minimal_model_program}. In particular, by Artin approximation, the property of being nodal is \'{e}tale-local. In this subsection, we show that being an inseparable node is also an \'{e}tale-local property.

\begin{lemma}\label{lemma: semi-normal and Gorenstein for strict henselization}
Let $(R,\fm)$ be a Noetherian local ring, with strict henselization $(R^\emph{sh},\fm^\emph{sh})$. Then 
		\begin{enumerate}
			\item $R$ is Gorenstein if and only if $R^\emph{sh}$ is Gorenstein, and
		\item $R$ is semi-normal if and only if $R^\emph{sh}$ is semi-normal.
		\item If $R$ is the quotient of a regular ring, then so is $R^\text{sh}$.
		\item The square
					$$\begin{tikzcd}
					\Spec (R^\emph{sh})^\nu\arrow[r]\arrow[d] & \Spec R^\emph{sh}\arrow[d] \\
					\Spec R^\nu\arrow[r] & \Spec R
					\end{tikzcd}$$
				is Cartesian, where $(\cdot)^\nu$ denotes normalization.
		\end{enumerate}
\end{lemma}
\begin{proof}
Recall that $\fm R^\text{sh}=\fm^\text{sh}$ and that $R\to R^\text{sh}$ is faithfully flat. So by \cite[23.4]{Matsumura_Commutative_Ring_Theory} we obtain the first equivalence, and the second one follows from \cite[1.8, 5.2, 5.3]{Greco_Traverso_On_seminormal_schemes}. 

Assume that $R\cong Q/I$, where $Q$ is a regular local ring and $I\subset Q$ an ideal. Then $R^\text{sh}\cong Q^\text{sh}/IQ^\text{sh}$ by \cite[05WS]{Stacks_Project}. Moreover $Q^\text{sh}$ is regular because $Q\to Q^\text{sh}$ is étale. 

The property of the square is proved in \cite[0CBM]{Stacks_Project}.
\end{proof}

\begin{corollary}[Being nodal is \'{e}tale-local]
Let $(R,\fm)$ be an excellent reduced one-dimensional local ring that is a quotient of a regular ring, with strict henselization $(R^\text{sh},\fm^\text{sh})$. Then $(R,\fm)$ is a node if and only if $(R^\text{sh},\fm^\text{sh})$ is a node.
\end{corollary}
\begin{proof}
Combine \autoref{lemma: commutative algebra definition of a node} and \autoref{lemma: semi-normal and Gorenstein for strict henselization}.
\end{proof}

\begin{corollary}[Being an inseparable node is \'{e}tale-local]
Let $(R,\fm)$ be an excellent reduced one-dimensional local ring that is a quotient of a regular ring, with strict henselization $(R^\text{sh},\fm^\text{sh})$. Then $(R,\fm)$ is an inseparable node if and only if $(R^\text{sh},\fm^\text{sh})$ is an inseparable node.
\end{corollary}
\begin{proof}
The local ring $(R,\fm)$ is an inseparable node if and only if it is semi-normal Gorenstein, its normalization $R^\nu$ is local and $R/\fm \to R^\nu/\fm R^\nu$ is a purely inseparable extension of degree $2$. We know by \autoref{lemma: semi-normal and Gorenstein for strict henselization} that $R$ is semi-normal Gorenstein if and only if $R^\text{sh}$ is, so we may assume $R$ is a node and consider the other conditions.

Now consider the Cartesian diagram given in \autoref{lemma: semi-normal and Gorenstein for strict henselization}. Prime ideals of $(R^\text{sh})^\nu$ corresponds to pairs $(\fp\in \Spec{R^\text{sh}},\fq\in\Spec{R^\nu})$ with the property that $\fp\cap R=\fq\cap R$. Now let $\fm'$ be a maximal ideal of $(R^\text{sh})^\nu$. Since $R$ is excellent, the normalization $R\to R^\nu$ is finite, and therefore the pullback $R^\text{sh}\to (R^\text{sh})^\nu$ is also finite. Hence $\fm'\cap R^\text{sh}$ is maximal and therefore equal to $\fm^\text{sh}$. Since $\fm^\text{sh}$ is the unique prime ideal of $R^\text{sh}$ lying above $\fm$, we see that maximal ideals of $(R^\text{sh})^\nu$ corresponds to pairs $(\fm^\text{sh},\fq)$ with $\fq\cap R=\fm$. Thefore $(R^\text{sh})^\nu$ is local if and only if $R^\nu$ is local.

Assume this is the case, and base-change along $R\to R/\fm$. As $\fm R^\text{sh}=\fm^\text{sh}$, we obtain the push-out square
		$$\begin{tikzcd}
		A:=(R^\text{sh})^\nu/\fm^\text{sh}(R^\text{sh})^\nu & k^\text{sh}:=R^\text{sh}/\fm^\text{sh}\arrow[l] \\
		k':=R^\nu/\fm R^\nu\arrow[u] & k:=R/\fm\arrow[u]\arrow[l]
		\end{tikzcd}$$
Since $R^\nu$ is local, by \autoref{proposition: normalization of nodes} we have that $k'$ is a field of the form
		$$k'=k[Z]/p(Z),$$
where $p(Z)\in k[Z]$ has degree $2$. Thus
		$$A\cong k^\text{sh}[Z]/p(Z).$$
Assume that $k'$ is purely inseparable over $k$. Then the roots of $p(Z)$ belongs to a purely inseparable extension of $k$. Since $k^\text{sh}$ a separable closure of $k$, we obtain that $p(Z)\in k^\text{sh}[Z]$ is irreducible and inseparable, and so $k^\text{sh}\subset A$ is a purely inseparable field extension of degree $2$. Conversely, if $A$ is a degree $2$ inseparable field extension of $k^\text{sh}$, then $p(Z)$ is inseparable over $k$ and therefore $k\subset k'$ is purely inseparable of degree $2$.
\end{proof}

\begin{remark}
All the results of this section hold with the Henselization instead of the strict Henselization, and the proofs are the same.
\end{remark}


\subsection{Construction of separable nodes}\label{section:construction_of_separable_nodes}
In this subsection we show how separable nodes appear when quotienting by finite equivalence relations.

\begin{proposition}\label{prop:finite_equivalence_relation_give_nodes}
Let $X$ be a scheme with disjoint and normal irreducible components, $D\subset X$ a reduced divisor with normalization $n\colon D^n\to D$ and a generically fixed point free involution $\tau\colon D^n\to D^n$. Let $R\subset X\times X$ be the equivalence relation induced by $\tau$, and we assume that $R$ is finite. Assume that either:
	\begin{itemize}
		\item $X$ is essentially of finite type over $\bF_p$, or
		\item the geometric quotient $\pi\colon X\to X/R$ exists as a scheme.
	\end{itemize}
Then:
	\begin{enumerate}
		\item the geometric quotient $\pi\colon X\to X/R$ exists (in particular $\pi$ is finite);
		\item $X/R$ is demi-normal with separable nodes, $\pi(D)$ is its conductor subscheme,  $\pi$ is the normalization morphism and is \'{e}tale over the codimension one points of $X/R$;
		\item if $X$ is of finite type (resp. proper) over $k$, then so is $X/R$; 
		\item the involution on $D^n$ induced by $D\to \pi(D)$ is precisely $\tau$.
	\end{enumerate}
\end{proposition}
\begin{proof}
If $X$ is defined over $\bF_p$, it follows from \cite[Theorem 6, Corollary 48]{Kollar_Quotients_by_finite_equivalence_relations} that the geometric quotient $\pi\colon X\to X/R=:Y$ exists as a scheme and the quotient morphism is finite, see \autoref{section:Quotients_by_finite_equivalence_relations}. So from now on we assume that $\pi\colon X\to X/R$ exists, and make no assumption about $X$ being of equicharacteristic.

Since $R$ restricts to the identity on $X\setminus D$, the restriction $\pi|_{X\setminus D}$ is an isomorphism onto its image. Since $\pi$ is finite, it must therefore be the normalization morphism. 

If $X$ is of finite type over $k$, then so is $Y$ by \cite[Theorem 41.2]{Kollar_Quotients_by_finite_equivalence_relations}. If $X$ is proper over $k$, then $Y$ is also proper over $k$ by \cite[09MQ,03GN]{Stacks_Project}.

For the rest of the proof, let us remark that taking the quotient commutes with flat base-change \cite[9.11]{Kollar_Singularities_of_the_minimal_model_program}, thus we are free to localize on $Y$ to perform our arguments.

Next we to show that $Y$ is demi-normal. First of all, since $Y$ is of finite type over $k$, it is excellent and each local ring of $Y$ is a quotient of a regular ring. We claim that $Y$ is reduced. Indeed, since $X$ is reduced then the two compositions
		$$R\rightrightarrows X\overset{\pi}{\longrightarrow} Y$$
factor uniquely through $Y_\text{red}$, and by the universal property of categorical quotients \cite[Definition 4]{Kollar_Quotients_by_finite_equivalence_relations} it follows that $Y=Y_\text{red}$. (A similar argument shows that $\pi$ must be dominant, so $Y$ is irreducible if and only if $X$ is).

We show that $Y$ has at worst nodal singularities at codimension one points.
For this we may localize at such points of $Y$. Then the set-up is the following: $Y=\Spec A$ is local with maximal ideal $\fn$, $X=\Spec B$ is semi-local with maximal ideals $\fm_i$, the map $D\to \pi(D)$ corresponds to $A/\fn\hookrightarrow \bigoplus_i B/\fm_i$, and since $D$ is normal the involution $\tau$ becomes a non-trivial involution of $\bigoplus_i B/\fm_i$. Notice that the natural map $\pi\colon B\to \bigoplus_i B/\fm_i$ is surjective by the Chinese remainder theorem. Now consider the commutative square
		$$\begin{tikzcd}
		A\arrow[d, twoheadrightarrow]\arrow[r, hook] & B\arrow[d, twoheadrightarrow, "\pi"] \\
		K:=A/\fn\arrow[r, hook] & L:=\bigoplus_i B/\fm_i
		\end{tikzcd}$$
According to \cite[9.10]{Kollar_Singularities_of_the_minimal_model_program}, a section $s\in B$ belongs to $A$ if and only if $\pi(s)=\tau(\pi(s))$. So we deduce that the fixed subring $L^{\langle\tau\rangle}$ is equal to $K$, and that the Jacobson radical of $B$ belongs to $A$. We claim that $\dim_KL=2$. Since the fibers of $X\to Y$ are the $R$-equivalence classes, we see that $B$ has at most two maximal ideals, and we consider these two cases separately:
	\begin{enumerate}
		\item $B$ is local. Then $K\subset L$ is a field extension. Since $L^{\langle\tau\rangle}=K$, we deduce that the extension is Galois of degree $2$ \cite[VI,\S 1, Theorem 1.8]{Lang_Algebra}.
		\item $B$ has two maximal ideals. Then $L=L_1\oplus L_2$, in which $K$ embedds diagonally, and $\tau$ exchanges the two components via an isomorphism $L_1\cong L_2$. Identifying $L_2$ with $L_1$ via this isomorphism, we see that $\tau$ acts on $L_1\oplus L_1$ via $(x,y)\mapsto (y,x)$. In particular the diagonal is fixed, so it must be equal to (the diagonal embedding) of $K$. Since $L_1\oplus L_1$ is of dimension $2$ over its diagonal, we have $\dim_KL=2$.
	\end{enumerate}
By \autoref{lemma:definition_node_via_normalization} it follows that $Y$ is nodal. This also proves that the extension of function fields given by $\pi\colon D\to \pi(D)$ is precisely $k(D)^{\langle\tau\rangle}\subset k(D)$, so the nodes are separable and the last statement holds.

To show that $Y$ is demi-normal, by virtue of \autoref{lemma: commutative algebra definition of a node} it remains to show that $Y$ is $S_2$. Since $Y$ is reduced it is $S_1$, and by \cite[1.8]{Hartshorne_Generalized_divisors_and_biliaison} it is sufficient to show the following property: if $U\subset Y$ is open and $Z\subset U$ is closed of codimension $\geq 2$, then the restriction map $\sO_Y(U)\to \sO_Y(U- Z)$ is bijective. Restricting over $U$, we may assume that $Y=U$. So let $s\in \Gamma(Y-Z,\sO_Y)$. Its pullback to $X$ extends to a global section $\pi^*s\in\Gamma(X,\sO_X)$ because $X$ itself is $S_2$ and $\codim_X\pi^{-1}Z\geq 2$. By construction the two compositions
		\begin{equation}\label{eqn:lifting_sections_on_quotient}
		R(\tau)\rightrightarrows D^n\overset{n^*(\pi^*s)}{\xrightarrow{\hspace{1.5cm}}} \bA^1
		\end{equation}
agree on $n^{-1}(X-\pi^{-1}Z)$, which is dense in $D^n$ as $\pi^{-1}Z$ has codimension $\geq 2$ in $X$. Therefore the two compositions in \autoref{eqn:lifting_sections_on_quotient} must agree on $D^n$ \cite[II, Exercise 4.2]{Hartshorne_Algebraic_Geometry}, which implies that $n^*( \pi^*s)$ descends to a section $t\in \Gamma(D/R|_D,\sO)$. By \cite[Proposition 25]{Kollar_Quotients_by_finite_equivalence_relations} the diagram
		$$\begin{tikzcd}
		D\arrow[r, hook]\arrow[d] & X \arrow[d, "\pi"] \\
		D/R|_D\arrow[r, hook] & Y
		\end{tikzcd}$$
is a push-out, thus $\pi^*s$ and $t$ glue together into a section $s'\in \Gamma(Y,\sO_Y)$ which extends $s$. 

Using \autoref{lemma: basic properties of normalization of demi-normal scheme}, it is easy to see that $\pi(D)$ and $D$ are the conductor subschemes of the normalization.
\end{proof}

\subsection{Construction of inseparable nodes}\label{section:inseparable_nodes}
The goal of this section is to characterize completely demi-normal varieties with only inseparable nodes in terms of their normalizations. This is achieved in \autoref{theorem: bijection for inseparable nodes} below. The main ingredient is a construction of inseparable nodes which we present now.

\begin{assumption}
Let $X$ be a Noetherian equidimensional reduced scheme defined over a field $k$ of characteristic $2$.
\end{assumption}
 
\begin{construction}\label{construction: construction of inseparable node} Let $D=\sum_{i=1}^n D_i$ be a reduced Weil divisor on $X$. 
Let $k_i:=k(D_i)$ be the function field $D_i$. Assume that for each $i$, we have an intermediate extension $k\subset k'_i\subsetneq k_i$ such that $k_i'\subset k_i$ is purely inseparable of degree $2$. We construct a subsheaf $\sA=\sA(k_1',\dots,k_n')$ of $\sO_X$ as follows: if $U\subset X$ is open, we let
		$$\sA(U):=\{s\in \sO_X(U)\mid s(\eta_i)\in k_i'\ \forall \eta_i\in U\}$$
where $\eta_i$ is the generic point of $D_i$, and $s(\eta_i)$ denotes the image of $s$ through the canonical map $\sO_X(U)\to k_i$. This defines a presheaf, which is easily seen to be a subsheaf of $\sO_X$.
\end{construction}

\begin{remark}
\autoref{construction: construction of inseparable node} is an example of the \emph{weak gluing} method developed in \cite{Yanagihara_Weakly_normal_ring_extensions}. Some of the results that follow are proved in \emph{op. cit.}. We give a geometric treatment that is better adapted to our purpose.
\end{remark}

\begin{proposition}\label{proposition: Properties of subsheaf I}
Let $X,D, k_i'\subset k_i$ and $\sA$ be as in \autoref{construction: construction of inseparable node}. Then:
	\begin{enumerate}
		\item $\sO_X^2\subset \sA\subset \sO_X$;
		\item $\sA$ is a sheaf of $k$-algebras and $\Spec{\sA}:=(|X|,\sA)$ is a reduced equidimensional scheme; 
		\item $\Spec\sA$ is Noetherian if $X$ is excellent or $F$-finite;
		\item if $X$ is $S_2$ and $\Spec\sA$ is Noetherian, then $\Spec\sA$ is $S_2$;
		\item if $X$ is excellent (resp. $F$-finite, resp. locally of finite type over $k$, resp. proper over $k$), then so is $\Spec{\sA}$;
		\item the morphism $\pi\colon X\to \Spec{\sA}$ is an affine integral birational universal homeomorphism, and it is finite if $X$ is excellent;
		\item $X$ and $\Spec{\sA}$ have the same normalization.
	\end{enumerate}
\end{proposition}
\begin{proof}
It is clear that $\sA$ is a sheaf of $k$-sub-algebras of $\sO_X$. Since $k_i/k_i'$ is inseparable of degree $2$, we have $(k_i)^2\subseteq k_i'$ for each $i$ and it follows that $\sO_X^2\subseteq \sA$. To show that $\Spec\sA$ is a scheme, we may assume that $X=\Spec{R}$ is affine Noetherian. Then $|\Spec \sA|=|X|$, and it is sufficient to show that $\sA(X_f)=\sA(X)_{f^2}$ for $f\in R$. The containment $\supseteq$ is clear. For the converse one, let $s/f^n\in \sA(X_f)\subset R_f$ with $s\in R$. Then $\bar{s}/\bar{f}^n\in k_i'$, so $\bar{f^n}\bar{s}\in k_i'$ as $\bar{f}^{2n}\in k_i'$. Hence $f^n s\in \sA(X)$, so 
				$$\frac{s}{f^n}=\frac{f^ns}{f^{2n}}\in \sA(X)_{f^2}.$$
	Hence $\Spec\sA$ is an integral scheme and the structural morphism $\pi\colon \Spec X\to \Spec\sA$ factors the Frobenius morphism of $X$. In particular $\pi$ is an affine integral universal homeomorphism.

It is clear that $\pi$ is an isomorphism away from the support of $D$. Since $\pi$ is integral, the normalization of $X$ and of $\Spec\sA$ are the same.

Assume that $X$ is locally of finite type over $k$. Then $\Spec\sA$ is locally of finite type over $k^2$ by \cite[Theorem 41.2]{Kollar_Quotients_by_finite_equivalence_relations}. Similarly, if $X$ is proper over $k$ then by \cite[09MQ,03GN]{Stacks_Project} it follows that $\Spec\sA$ is proper over $k$.

Assume that $X$ is $S_2$. Since $\Spec\sA$ is $S_1$ and $X\to \Spec{\sA}$ is an homeomorphism, it follows easily from the criterion given in \cite[1.8]{Hartshorne_Generalized_divisors_and_biliaison} that $\Spec\sA$ is $S_2$, provided it is Noetherian.

Now assume that $R$ is excellent. Since $R$ is reduced, we have an abstract ring isomorphism $R\cong R^2$, and so $R^2$ is also excellent. An excellent ring is Nagata \cite[07QV]{Stacks_Project}, so $A:=\sA(X)\subset \text{Frac}(R)$ is a finite $R^2$-module. Therefore $A$ is Noetherian and excellent \cite[07QU]{Stacks_Project}. Applying the same argument for $A\subset R$, we obtain that $R$ is finite over $A$. In particular $\pi$ is finite.

If $R$ is $F$-finite, then it is a finite $R^2$-module. Since $R^2$ is Noetherian, we get that $R$ is a Noetherian $R$-module. As $A\subset R$ is an $R^2$-submodule, we obtain that $A$ is a finite $R^2$-module, and hence a Noetherian ring. Moreover $R^4\to R^2$ is a finite extension and factors through $A^2$, so $A^2\to R^2$ is a finite extension. Hence $A^2\to R^2\to A$ is a composition of finite extensions, so it is finite, which means that $A$ is $F$-finite.
\end{proof}

\begin{proposition}\label{prop: Properties of subsheaf II}
Notations as in \autoref{proposition: Properties of subsheaf I}. Assume that $X$ is excellent, normal at the generic points of $D$ and demi-normal elsewhere. Then $\Spec{\sA}$ is demi-normal with inseparable nodes at the generic points of $D$.
\end{proposition}
\begin{proof}
Let $\eta\in X$ be a generic point of $D$, and $k'_\eta\subset k_\eta:=k(\eta)$ the inseparable degree $2$ sub-extension fixed at the beginning. By assumption $\sO_{X,\eta}$ is a DVR. We know that
		$$\sO_{X,\eta}^2\subset \sA_\eta\subset\sO_{X,\eta}$$
and since $X$ is excellent, the extension $\sO_{X,\eta}^2\subset\sA_\eta$ is finite. In particular there is a sujective map of rings
		$$\sO_{X,\eta}^2[x_1,\dots,x_n]\twoheadrightarrow \sA_\eta$$
and $\sO_{X,\eta}^2[x_1,\dots,x_n]$ is a regular ring. This shows that $\sA_\eta$ is a quotient of a regular ring. Moreover $\sA_\eta\to \sO_{X,\eta}$ is the normalization. By construction $\fm_{X,\eta}=\fm_{\sA,\eta}\subset \sA_\eta$ and $\sA_\eta/\fm_{\sA,\eta}=k'_\eta$, so
		$$\dim_{\sA_\eta/\fm_{\sA,\eta}}\sO_{X,\eta}/\fm_{X,\eta}=\dim_{k'_\eta}k_\eta=2.$$
	Therefore $\eta\in\Spec\sA$ is nodal by \autoref{lemma:definition_node_via_normalization}, and it is an inseparable node by definition. On the other hand, $\Spec\sA$ is demi-normal at codimension one points that are not in the image of $D$, since $\pi$ is an isomorphism in a neighborhood of such points. Since the $S_2$ property of $X$ descends to $\Spec\sA$, we obtain that $\Spec\sA$ is demi-normal. 
\end{proof}

\begin{proposition}\label{proposition: Properties of subsheaf III}
Notations as in \autoref{proposition: Properties of subsheaf I}. Assume that $X$ is excellent, normal at the codimension one points of $D$ and demi-normal elsewhere. Let $\Delta$ be a $\bQ$-Weil divisor on $X$ that shares no common component with $D$. Assume that $K_X+D+\Delta$ is $\bQ$-Cartier. Then:
	\begin{enumerate}	
		\item if $\Delta_{\sA}:=\pi_*\Delta$, then $K_{\Spec\sA}+\Delta_{\sA}$ is $\bQ$-Cartier;
		\item $K_X+D+\Delta =\pi^*(K_{\Spec\sA}+\Delta_\sA)$.
	\end{enumerate}
\end{proposition}
\begin{proof}
The Gorenstein locus $j\colon U\subset \Spec{\sA}$ is open \cite[1.5]{Greco_Marinari_Nagata_criterion}, and by \autoref{prop: Properties of subsheaf II} and \autoref{lemma: commutative algebra definition of a node} it contains every codimension one point. By \cite[5.7]{Kollar_Singularities_of_the_minimal_model_program}, we have $\pi^*K_{\Spec\sA}= K_X+D$ above every codimension one point of $\Spec\sA$. Since $\Delta$ and $D$ have no common components, we may shrink $U$ (but keeping it a big open subset) and assume that $\pi^{-1}(U)\cap D\cap\Supp\Delta$ is empty. As $\pi$ is an isomorphism outside $D$, we obtain that $K_{\Spec\sA}+\Delta_\sA$ is $\bQ$-Cartier over $U$ and
			$$\pi_U^*(K_{\Spec\sA}+\Delta_\sA)|_U\sim_\bQ (K_X+D+\Delta)|_{\pi^{-1}U}.$$
			By \cite[Lemma 1.4.3]{Keel_basepoint_freeness_in_positive_char}, the pullback map $\pi^*\colon \Pic(\Spec\sA)[1/2]\to\Pic(X)[1/2]$ is an isomorphism, and this statement localizes on $X$. Since $K_X+D+\Delta$ is $\bQ$-Cartier, for $m$ divisible enough there is a Cartier divisor $L$ on $\Spec\sA$, unique up to isomorphism, such that $\pi^*L=m(K_X+D+\Delta)$. By uniqueness over $U$ we have
				$$L_U\sim_\bQ m(K_{\Spec\sA}+\Delta_\sA)|_U.$$
		Since $\sO(L)$ and $\sO(m(K_{\Spec\sA}+\Delta_\sA))$ are reflexive and $U$ is big, we obtain that
			$$L\sim_\bQ m(K_{\Spec\sA}+\Delta_\sA)$$
	which proves that $K_{\Spec\sA}+\Delta_\sA$ is $\bQ$-Cartier and pullbacks (as $\bQ$-divisor) to $K_X+D+\Delta$.
\end{proof}

\begin{theorem}\label{theorem: bijection for inseparable nodes}
Let $k$ be a field of characteristic $2$. Let $\bold{P}$ be any (or none) of the following properties: $F$-finite, locally of finite type over $k$, proper over $k$. 

Then normalization gives a one-to-one correspondence
	\begin{equation*}
		(\Char k = 2) \quad
		\begin{pmatrix}
		\text{Demi-normal excellent reduced}\\ 
		\text{equidimensional schemes } Y \text{ over }k\\
		\text{with only inseparable nodes}\\
		\text{satisfying }\bold{P}
		\end{pmatrix}
		\overset{1:1}{\longrightarrow}
		\begin{pmatrix}
		\text{Triples }\big(X,\sum_i D_i, k\subset k_i'\subset k(D_i)\big)\\
		\text{where }X\text{ is a normal excellent}\\
		\text{scheme over }k\text{ satisfying }\bold{P},\\
		 \sum_iD_i \text{ is a reduced Weil divisor},\\ 
		 k_i'\subset k(D_i)\text{ are degree }2\\
		 \text{inseparable extensions over }k
		\end{pmatrix}
	\end{equation*}
whose inverse is given by \autoref{construction: construction of inseparable node}. This correspondence specializes to
	\begin{equation*}
		(\Char k = 2) \quad\quad\quad
		\begin{pmatrix}
		\text{Slc surface} \\
		\text{pairs }(S,\Delta)\text{ over }k \\		
		\text{with only inseparable nodes}
		\end{pmatrix}
		\overset{1:1}{\longrightarrow}
		\begin{pmatrix}
		\text{Lc surface}\\
		\text{pairs }(\tilde{S},\tilde{D}+\tilde{\Delta})\text{ over }k
		\end{pmatrix}
	\end{equation*}
where $S$ is proper if and only if $\tilde{S}$ is proper, and $K_S+\Delta$ is ample if and only if $K_{\tilde{S}}+\tilde{D}+\tilde{\Delta}$ is ample.
\end{theorem}
\begin{proof}
Given $\big(X,\sum_i D_i, k\subset k_i'\subset k(D_i)\big)$ as in the right-hand side, \autoref{construction: construction of inseparable node} gives a $k$-scheme $Y$. By \autoref{proposition: Properties of subsheaf I} and \autoref{prop: Properties of subsheaf II}, $Y$ satisfies the claimed properties. This gives a map $\Phi\colon \{(X,\sum_i D_i, k_i'\subset k(D_i))\}\to \{Y\}$. 

Conversely, let $Y$ be as in the left-hand side with conductor $D\subset Y$. Let $Y^\nu\to Y$ be its normalization with conductor $D^\nu\subset Y^\nu$. Since the normalization is finite, the finiteness properties of $Y$ ascend to $Y^\nu$. Then $(Y^\nu, D^\nu, k(D)\subset k(D^\nu))$ is a triplet as in the right-hand side. This gives a map $\nu\colon\{Y\}\to \{(X,\sum_i D_i, k_i'\subset k(D_i))\}$.

The composition $\nu\circ\Phi$ is the identity by \autoref{proposition: Properties of subsheaf I} and \autoref{prop: Properties of subsheaf II}. To check that $\Phi\circ\nu(Y)=Y$, observe that both sides can be described by their structural sheaves on the topological space $|Y|$, and that both $\sO_{\Phi\circ \nu(Y)}$ and $\sO_Y$ are $S_2$ subsheaves of $\sO_{Y^\nu}$. Thus it suffices to show that $\sO_{\Phi\circ \nu(Y)}=\sO_Y$ at codimension one points. Equality is obvious on the normal locus. On the demi-normal locus, it follows from \autoref{proposition: reconstructing inseparable nodes from normalizations}.

Now let us consider the surface case. We claim that if $\tilde{S}$ is a normal surface over $k$ and $\tilde{D}\subset\tilde{S}$ is a prime Weil divisor, then there is a unique purely inseparable degree $2$ sub-extension $k'\subset k(\tilde{D})$ containing $k$. Since $\trdeg_kk(\tilde{D})=1$, there is indeed a unique such sub-extension, given by the relative Frobenius of $k(D)$ over $k$: see the proof of \cite[0CCY]{Stacks_Project}. By \autoref{proposition: Properties of subsheaf III}, $K_S+\Delta$ is $\bQ$-Cartier if and only if $K_{\tilde{S}}+\tilde{D}+\tilde{\Delta}$ is $\bQ$-Cartier, and the first one pullbacks to the second one. In particular $(S,\Delta)$ is slc if and only if $(\tilde{S},\tilde{D}+\tilde{\Delta})$ is lc. Since the normalization $\tilde{S}\to S$ is finite, the ampleness statement is immediate.
\end{proof}

\subsection{Semi-resolutions of demi-normal surfaces}
It might be difficult to study some aspects of a demi-normal scheme $X$ in terms of its normalization $\nu\colon \bar{X}\to X$, since $\nu_*\sO_{\bar{X}}\neq \sO_X$ and $\nu$ is not an isomorphism in codimension one. Instead, one may try to resolve singularities in codimension two only. In characteristic $0$, there is a good notion of such partial resolutions for demi-normal surfaces \cite{van_Straten_weakly_normal_surfaces, Kollar_Shepherd_Barron_3folds_and_deformations_of_surfaces_singularities}.
We work out the case of demi-normal surfaces in positive characteristic.

\begin{definition}
A germ of surface $(s\in S)$ is called a \textbf{normal crossing point}, respectively a \textbf{pinch point}, if there is a finite étale morphism $\widehat{\sO_{S,s}}\to \sO'$ such that $\sO'\cong k[[x,y,z]]/(xy)$, respectively $\sO'\cong k[[x,y,z]]/(x^2-zy^2)$, where $k$ is some field.

A surface (essentially of finite type over a field) is called \textbf{semi-smooth} if every closed point is either regular, normal crossing or a pinch point.
\end{definition}

\begin{remark}
The étale base-change $\sO_{S,s}\to \sO'$ may be non-elementary (that is, it may not induce an isomorphism of residue fields). But if $k(s)$ is algebraically closed, $\widehat{\sO_{S,s}}\to \sO'$ is necessarily elementary. Since $\widehat{\sO_{S,s}}$ is Henselian and $\sO'$ is assumed to be local, we deduce that $\sO'\cong\widehat{\sO_{S,s}}$. In other words, we can omit the base-change in the definition if we work over an algebraically closed field.
\end{remark}

\begin{remark}[Singularities of semi-smooth surfaces]\label{remark:singularities_of_semi_smooth_surfaces}
Let $S$ be a semi-smooth surface. Then we claim that $S$ is a local complete intersection scheme -- that is, every completed local ring $\widehat{\sO_{S,s}}$ is the quotient of a regular ring by a regular sequence. In particular $S$ is Cohen--Macaulay and Gorenstein. Since the local complete intersection locus is open by \cite[3.3]{Greco_Marinari_Nagata_criterion}, it is sufficient to show the local rings of $S$ at closed points are local complete intersections. This property descends and ascends étale morphisms \cite[09Q7]{Stacks_Project}. Thus it suffices to show that the models $k[x,y,z]/(xy)$ and $k[x,y,z]/(x^2-zy^2)$ are local complete intersections (near the origin), which is clear.

We claim that $S$ is also semi-normal. Since the semi-normal locus is open (\autoref{remark:semi_normality_and_demi_normality}), it is sufficient to check that the closed points belong to it. Combining \cite[Theorem 1.6]{Greco_Traverso_On_seminormal_schemes} and \cite[7.8.3.vii]{EGA_IV.2}, it is sufficient to check that condition on the local models $k[x,y,z]/(xy)$ and $k[x,y,z]/(x^2-zy^2)$. This is easily seen, e.g. using \cite[Corollary 2.7.vii]{Greco_Traverso_On_seminormal_schemes}. It follows by \autoref{lemma: commutative algebra definition of a node} that $S$ is demi-normal.

We also claim that the singular locus of $S$ is regular. If $f\colon X\to Y$ is an \'{e}tale or a regular morphism of Noetherian schemes, then $\sO_{X,x}$ is regular if and only if $\sO_{Y,f(x)}$ is regular \cite[23.7]{Matsumura_Commutative_Ring_Theory}, and the completion of an excellent local ring is a regular morphism \cite[7.8.3.v]{EGA_IV.2}. Thus it is sufficient to prove that the singular loci of $k[x,y,z]/(xy)$ and $k[x,y,z]/(x^2-zy^2)$ are regular (near the origin). In both cases the singular locus is the line $(x=y=0)$, which is regular.

A similar argument shows that the normalization of $S$ is regular and that the preimage of the conductor is also regular. 
\end{remark}

\begin{definition}\label{definition:semi_resolution}
A proper birational morphism of demi-normal surfaces $f\colon T\to S$ is called a \textbf{semi-resolution} if:
	\begin{enumerate}
		\item $T$ is semi-smooth;
		\item no component of $D_T$, the conductor divisor of $T$, is $f$-exceptional;
		\item $f$ is an isomorphism over a big open subset of $S$.
	\end{enumerate}
We say that $f$ is a \textbf{good semi-resolution} if in addition
	\begin{enumerate}\setcounter{enumi}{3}
		\item $\Exc(f)\cup D_T$ has regular components and transverse intersections (so $E$ has at most double points, and $E\cup D_T$ at most triple points).
	\end{enumerate}
\end{definition}

In characteristic $0$, it is well-known that demi-normal surfaces admit good semi-resolutions. We prove that it is also true for demi-normal surfaces in characteristic $\neq 2$. Our proof is similar in spirit to \cite[\S 1.4]{van_Straten_weakly_normal_surfaces}.

\begin{lemma}\label{lemma:involution_of_power_series}
Let $K$ be a field of characteristic $\neq 2$ and $\tau$ a non-trivial non-necessarily $K$-linear involution of $K[[t]]$ such that $\tau(K)=K$. Then there exists a uniformizer $s\in (t)\setminus (t^2)$ such that $\tau(s)=-s$.
\end{lemma}
\begin{proof}
We have $\tau(t)=\lambda t$ with $\lambda\in K[[t]]^\times$. If the constant term of $\lambda$ is not equal to $1$, then $1-\lambda\in K[[t]]^\times$ and so the $K$-linear ring map defined by $t\mapsto t-\tau(t)$ is an automorphism, thus we may take $s=t-\tau(t)$.

From now on assume that the constant term of $\lambda$ is $1$. If $\tau$ does not act as the identity on $K$, choose $\alpha\in K\setminus K^\tau$. If $\alpha=\tau(\alpha)+1$ then $\alpha=\tau(\tau(\alpha))=\tau(\alpha-1)=\tau(\alpha)-1$, a contradiction since the characteristic is different from $2$. Therefore $\alpha -\tau(\alpha)\lambda$ is invertible with constant term different from $1$. Thus we apply the argument of the previous paragraph with $\alpha t$ in place of $t$.

Finally assume that $\tau$ is $K$-linear. By the current assumption on $\lambda$ we have $\tau(t)=t+O(t^2)$, and so $\tau(t^k)=t^k+O(t^{k+1})$. We define a Cauchy sequence $(t_n)$ such that $t_n-\tau(t_n)\in O(t^{n+1})$. Take $t_1:=t$. If $t_n-\tau(t_n)=at^{n+1}+O(t^{n+2})$ with $a\in K$, the element $t_{n+1}:=t_n-\frac{a}{2}t^{n+1}$ is a valid choice. Then $t_\infty:=\lim_n t_n$ is a uniformizer that satisfies $t_\infty=\tau(t_\infty)$. Hence $\tau$ is the identity, and we excluded this case in the hypothesis. Thus the proof is complete.
\end{proof}

\begin{lemma}\label{lemma:quotient_of_log_smooth_gives_semi_smooth}
Let $S$ be a regular surface of finite type over an arbitrary field $k$ of positive characteristic $\neq 2$, $D\subset S$ a reduced divisor with regular support and $\tau\colon D\cong D$ a non-trivial involution. Then $S/R(\tau)$ exists and is a semi-smooth surface whose conductor subscheme is the image of $D$.
\end{lemma}
\begin{proof}
We elaborate the last paragraph of \cite[1.43]{Kollar_Singularities_of_the_minimal_model_program}. First of all, it is clear that $R(\tau)$ is finite, thus $U:=S/R(\tau)$ exists as a scheme of finite type over $k$ and $\pi\colon S\to U$ is the normalization. By \autoref{prop:finite_equivalence_relation_give_nodes}, $U$ is demi-normal with conductor $\pi(D)$. By \cite[9.13, 9.30]{Kollar_Singularities_of_the_minimal_model_program} the square
		$$\begin{tikzcd}
		D\arrow[r, hook]\arrow[d, "q"] & S\arrow[d, "\pi"] \\
		D/\langle\tau\rangle\arrow[r, hook] & U
		\end{tikzcd}$$
is a push-out. Geometric quotients by finite group actions preserve normality, so $\pi(D)$ is normal, therefore a regular curve.

To study the singularities of $U$, we may localize at a closed point $u\in U$ that belongs to $\pi(D)$. Then we may assume that $U=\Spec \sO$ is local, $S=\Spec A$ and that $f\in A$ is a local equation for $D$. Since $D\to\pi(D)$ is a $\bZ/2\bZ$-quotient, only two cases can happen.
	\begin{enumerate}
		\item $A$ has exactly two maximal ideals. Since $\hat{\sO}\cong \widehat{\sO^h}$ \cite[06LJ]{Stacks_Project}, we may base-change along an elementary \'{e}tale morphism $\Spec\sO'\to \Spec\sO$ and assume that $A=A_1\oplus A_2$, where both $A_i$ are local. Let $f_i\in A_i$ be the local equation of $D$, and $\tau\colon A_1/(f_1)\cong A_2/(f_2)$ be the involution. Since $A_1/(f_1)$ is regular local of dimension one, there exists $g_1\in A_1$ such that $(f_1,g_1)=\fm_{A_1}$. Let $g_2\in A_2$ be any lift of $\tau(g_1+(f_1))$, it also holds that $(f_2,g_2)=\fm_{A_2}$.
		
		The push-out description of $\sO\subset A_1\oplus A_2$ shows that $\fm_{\sO}=\sO\cap (\fm_{A_1}\oplus\fm_{A_2})$. Hence we see that $x:=(f_1,0),y:=(0,f_2),z:=(g_1,g_2)$ generate $\fm_\sO$, with relation $xy=0$. Thus $\hat{\sO}\cong k[[x,y,z]]/(xy)$, where $k$ is the residue field of $\sO$.
		
		\item $A$ is local, in other words $u$ is the image of a $\tau$-fixed point. Let $f\in A$ be the local equation of $D$, and $\tau\colon A/(f)\cong A/(f)$ be the involution. We may work with the completions, as $\hat{\sO}$ is the preimage in $\hat{A}$ of the $\tau$-invariant elements of $\hat{A}/(f)$.
		
		The action of $\tau$ descends to the residue field $K$ of $A$, and extends to the completion of $A/(f)$. Notice that the restriction of $\tau$ on the field $K\subset \hat{A}/(f)$ is precisely the action of $\tau$ on the residue field. The residue field of $\sO$ is the fixed subfield $K':=K^\tau$. If $K'$ is algebraically closed, notice that necessarily $K=K'$.
		
		The completion $\hat{A}/(f)$ is isomorphic to $K[[t]]$. By \autoref{lemma:involution_of_power_series} we may assume that $\tau(t)=-t$.

		If $K'=K$, then $\hat{\sO}=K[[f,fg,g^2]]\cong K[[x,y,z]]/(x^2z-y^2)$, where $g\in \hat{A}$ is any lift of $t\in \hat{A}/(f)$. This is a pinch point.
		
		If $K'\subsetneq K$ then $K=K'(\gamma)$ where $\gamma^2=c\in K$ and $\tau(\gamma)=-\gamma$. A monomial $(a+b\gamma)t^i$, with $a,b\in K'$, is $\tau$-invariant if and only if $i$ is even and $b=0$, or $i$ is odd and $a=0$. Thus if $g\in \hat{A}$ is any lift of $t$, we have
				$$\hat{\sO}=K'[[f,fg,g^2]]+\gamma g\cdot K'[[g^2]]+\gamma\cdot K'[[f,fg,fg^2,fg^3,\dots ]].$$
		Let
				$$x:=g^2,\ y:=\gamma g,\ z:=f,\ v:=fg,\ w:=\gamma f,$$
		then we have the presentation
				$$\hat{\sO}\cong \frac{K'[[x,y,z,v,w]]}{(y^2+cx,yw+cv, cz^2+w^2,xz^2-v^2)}\cong\frac{K'[[y,z,w]]}{(cz^2+w^2)}.$$
		The finite \'{e}tale extension $\hat{\sO}\subset \hat{\sO}[T]/(T^2+c)$ shows that $\sO$ is a normal crossing point.
	\end{enumerate}
This completes the proof.
\end{proof}

\begin{remark}
In the situation of \autoref{lemma:quotient_of_log_smooth_gives_semi_smooth}, notice that if the base field $k$ is algebraically closed, then the normal crossing points of $U$ are the image of the non-$\tau$-fixed points of $D$, and the pinch points of $U$ are the images of the $\tau$-fixed points of $D$.
\end{remark}

\begin{theorem}\label{theorem:semi_resolution_for_surfaces}
Let $S$ be a demi-normal surface that is essentially of finite type over an arbitrary field $k$ of positive characteristic $\neq 2$. Then $S$ has a good semi-resolution with slc singularities.
\end{theorem}
\begin{proof}
It follows from \autoref{lemma: commutative algebra definition of a node}, \autoref{remark:semi_normality_and_demi_normality}, \cite[7.8.3.iv]{EGA_IV.2} and \cite[1.5]{Greco_Marinari_Nagata_criterion} that the demi-normal locus is open on excellent schemes, thus we can realize $S$ as the localization of a demi-normal surface of finite type over $k$. Hence we may assume that $S$ is of finite type over $k$ to begin with.

Let $\nu\colon (\bar{S},D_{\bar{S}})\to S$ be the normalization morphism. By \cite[2.25]{Kollar_Singularities_of_the_minimal_model_program} there exists a proper birational morphism $\bar{f}\colon \bar{T}\to \bar{S}$, where $\bar{T}$ is regular, $D_{\bar{T}}:=\bar{f}_*^{-1}D_{\bar{S}}$ is regular, the components of $E=\Exc(\bar{f})$ are regular and $E\cup D_{\bar{T}}$ has simple normal crossings.

The normalization $\bar{S}\to S$ induces an involution $\tau$ on $D_{\bar{S}}^n=D_{\bar{T}}$. So we get a finite equivalence relation $R(\tau)\rightrightarrows \bar{T}$. Now let $q\colon \bar{T}\to T$ be the quotient morphism. By the universal property of the quotient, there is a unique morphism $f\colon T\to S$ such that the square
		\begin{equation}\label{eqn:semi_resolution_diagram}
		\begin{tikzcd}
		\bar{T}\arrow[d, "q"]\arrow[r, "\bar{f}"] & \bar{S}\arrow[d, "\nu"] \\
		T \arrow[r, "f"] & S
		\end{tikzcd}
		\end{equation}
commutes. Combining \autoref{prop:finite_equivalence_relation_give_nodes} and \cite[09MQ, 03GN]{Stacks_Project} we see that $f$ is proper, and clearly it is birational.

By \autoref{lemma:quotient_of_log_smooth_gives_semi_smooth} the surface $T$ is semi-smooth, with normalization $\bar{T}$ and conductors $D_T\subset T$ and $D_{\bar{T}}\subset \bar{T}$. By semi-smoothness $\omega_T$ is invertible (see \autoref{remark:singularities_of_semi_smooth_surfaces}), and as it pull backs to $\omega_{\bar{T}}(D_{\bar{T}})$ we see that $T$ is slc. Since we have glued along an involution on $D_{\bar{T}}$, the divisor $q(E)\cup D_T$ are at worst triple points and $q(E)$ has at most double points. There is a big open subset $U\subset S$ such that both $\nu^{-1}(U)$ and $D_{\bar{S}}\cap\nu^{-1}U$ are regular, and $f(\Exc(f))\cap U$ is empty. Then $f$ is an isomorphism over $U$.

Notice however that $q(E)$ might not have regular components. If a component $E_i$ of $E$ intersects $D_{\bar{T}}$ in two points that are $\tau$-conjugate, then $q(E_i)$ will be a nodal curve. In this case $T$ is only a semi-resolution. To obtain a good semi-resolution, we may blow-up these intersection points on $\bar{T}$ before gluing along $\tau$: this ensures that the components of $q(E)$ are regular.
\end{proof}

\begin{example}
Let us illustrate this semi-resolution procedure with the triple point $S=(xyz=0)\subset \bA^3$. The normalization $\bar{S}$ is a disjoint union of three planes, and the conductor is the union of the coordinate axis on each plane. Let $\bar{T}=\bigsqcup_{i=0}^2\bar{T}_i$ be the blow-up of $\bar{S}$ at the three origins, and $L_i^1,L_i^2\subset \bar{T}_i$ be the transforms of the coordinate axis. Then the semi-resolution $T\to S$ is obtained by gluing $L_i^1$ along $L_{i+1}^2$, where the index is taken modulo $3$.
\end{example}

\begin{proposition}[see {\cite[4.10,4.13]{Kollar_Shepherd_Barron_3folds_and_deformations_of_surfaces_singularities}}]\label{proposition:GR_for_demi_normal_surface}
Let $S$ be as in \autoref{theorem:semi_resolution_for_surfaces}. 
	\begin{enumerate}
		\item If $f\colon T\to S$ is a semi-resolution then $f_*\sO_T\cong \sO_S$,
		\item Grauert--Riemenschneider vanishing holds: $R^1f_*\omega_T=0$.
		\item There exists a minimal semi-resolution of $S$.
	\end{enumerate}
\end{proposition}
\begin{proof}
By definition $f$ is an isomorphism over a big open subset of $S$, thus the inclusion $\sO_S\subseteq f_*\sO_T$ is an equality in codimension one. Since $\sO_S$ is $S_2$, equality holds everywhere.

Let $\pi\colon \bar{T}\to T$ be the normalization. The trace map $\pi_*\omega_{\bar{T}}\to\omega_T$ is given by the evaluation at $1$
		$$\pi_*\omega_{\bar{T}}=\Hom_T(\pi_*\sO_{\bar{T}},\omega_T)\overset{\ev_1}{\longrightarrow} \omega_T.$$
Since the characteristic is different from $2$, the normalization is \'{e}tale over codimension one points by \autoref{lemma: basic properties of normalization of demi-normal scheme}. Thus the trace map is injective in codimension one. The pushforward $\pi_*\omega_{\bar{T}}$ is $S_2$ and $\omega_T$ is torsion-free, so the trace map is injective. Thus we obtain an exact sequence 
		$$0\to \pi_*\omega_{\bar{T}}\to \omega_T\to \sQ\to 0,$$
where $\sQ$ is supported on the divisor $D_T=\Sing(T)$. Pushing forward along $f\colon T\to S$ and using that $\pi$ is finite, we obtain the exact sequence
		$$R^1(f\circ\pi)_*\omega_{\bar{T}}\cong R^1f_*(\pi_*\omega_{\bar{T}})\to R^1f_*\omega_T\to R^1f_*\sQ\cong R^1(f|_{D_T})_*\sQ.$$
We have $R^1(f\circ\pi)_*\omega_{\bar{T}}=0$ by \cite[10.4]{Kollar_Singularities_of_the_minimal_model_program} and $R^1(f|_{D_T})_*\sQ=0$ since $f|_{D_T}$ is finite on its image by assumption on $f$. Thus $R^1f_*\omega_T=0$.

Finally, let $\bar{S}\to S$ be the normalization, $\bar{T}_m\to \bar{S}$ the minimal log resolution of $(\bar{S},D_{\bar{S}})$ and $f_m\colon T_m\to S$ be the semi-resolution obtained by gluing $\bar{T}_m$ as in \autoref{theorem:semi_resolution_for_surfaces}. If $f\colon T\to S$ is a semi-resolution, then using \autoref{remark:singularities_of_semi_smooth_surfaces} we see that $\bar{T}\to (\bar{S},D_{\bar{S}})$ is a log resolution. By minimality of $\bar{T}_m$ the morphism $\bar{T}\to \bar{S}$ factors through $\bar{T}_m$. Since $\pi\colon \bar{T}\to T$ is a quotient, it is easy to see that we obtain a commutative diagram
		$$\begin{tikzcd}
		\bar{T}\arrow[r]\arrow[d, "\pi"] & \bar{T}_m\arrow[r]\arrow[d] & \bar{S}\arrow[d] \\
		T \arrow[r, "h"] & T_m\arrow[r, "f_m"] & S
		\end{tikzcd}$$ 
So both $f_m\circ h$ and $f$ give a factorization of $\bar{T}\to \bar{S}\to S$ through $T$. Using again that $\bar{T}\to T$ is a quotient, we deduce that $f_m\circ h=f$. This shows minimality of $T_m$ amongst semi-resolutions.
\end{proof}

\begin{remark}[Minimal good semi-resolutions for slc singularities]
Let $(s\in \bar{S},D_{\bar{S}})$ be a germ of lc surface with non-empty reduced boundary $D_{\bar{S}}$, over an algebraically closed field of characteristic $\neq 2$. Let $(\bar{T},D_{\bar{T}}+E)$ be the minimal log resolution. The possible dual graphs of the resolution are listed in \cite[Reminder 9]{Kollar_Log_plurigenera_for_stable_surface_families}. The only case where a component of $E$ has two intersection points with $D_{\bar{T}}$ is: when $D_{\bar{S}}$ has a node at $s$ and $E$ is irreducible \cite[9.2]{Kollar_Log_plurigenera_for_stable_surface_families}. If $D_{\bar{S}}^n$ is endowed with an involution $\tau$ and if the two intersection points are conjugate, then the image of $E$ is $\bar{T}/R(\tau)$ will be a nodal curve. This is the only case where gluing the minimal log resolution gives a semi-resolution that fails to be a good semi-resolution.

Notice that in this case the boundary $D_{\bar{S}}$ has a node at $s$ and that $E$ is an lc place of $(\bar{S},D_{\bar{S}})$ \cite[2.31]{Kollar_Singularities_of_the_minimal_model_program}. A log resolution is a also a resolution of the underlying surface, so $(s\in\bar{S})$ is at worse an $A_1$-singularity.

We can get a good semi-resolution by blowing-up either one of the two intersection points of $E$ and $D_{\bar{T}}$, and then glue along the involution. This gives two good semi-resolutions, neither of which dominates the other.

A concrete example can be constructed as follows. Take the quadratic cone $X=V(xz-y^2)\subset\bA^3$, it has an $A_1$-singularity at the vertex that is resolved by blowing-up once. Consider the two lines $L_1=(x=y=0)$ and $L_2=(z=y=0)$. Then $(X,L_1+L_2)$ is an lc surface and blowing-up the vertex gives a log resolution. Since $L_1\cong\bA^1\cong L_2$, we can find an isomorphism $\tau\colon L_1\cong L_2$ sending the vertex to the vertex, and by \autoref{theorem:gluing_for_surfaces_p_different_from_2} (or \autoref{remark:gluing_works_in_char_2_for_separable_nodes} in characteristic $2$) we can glue $X$ along $\tau$ to obtain an slc surface which does not have a minimal good semi-resolution.

Therefore an slc surface singularity $(s\in S)$ has a minimal good semi-resolution if and only if there is no preimage $\bar{s}$ in the normalization $\bar{S}$ such that: $(\bar{s}\in\bar{S})$ is at worse an $A_1$-singularity, an lc center of $(\bar{S},D_{\bar{S}})$, and the two preimages of $\bar{s}$ in $D_{\bar{S}}$ are $\tau$-conjugate.
\end{remark}

\begin{remark}\label{remark:semi_resolution_in_char_2}
There are two obstacles to extend \autoref{theorem:semi_resolution_for_surfaces} in characteristic $2$. First of all we must choose whether the semi-resolution $T\to S$ is an isomorphism over the inseparable nodes of $S$. If we want a local isomorphism, we must apply \autoref{construction: construction of inseparable node} to a regular curve on $\bar{T}$. Then it seems difficult to provide a formal description of the local rings at the closed points of the inseparable-nodal locus of $T$. 

The second difficulty stems from the failure of \autoref{lemma:involution_of_power_series} in characteristic $2$, as demonstrated by the $\mathbb{F}_2$-linear involution of $\bF_2[[t]]$ given by $t\mapsto (1+t^2+t^3+\dots )t$. Thus the treatment of $\tau$-fixed points in the proof of \autoref{lemma:quotient_of_log_smooth_gives_semi_smooth} becomes problematic. In principle one could use Artin--Schreier theory to classify involutions on one-dimensional power series, see e.g. the example at the end of \cite{Artin_Wildly_ramified_Z/2_action_in_dim_2} for the linear cases. This leads to surface singularities that are not normal crossing or pinch points, for example $(x^2+zy^2+xyz^r=0)\subset \bA^3$ for $r\geq 1$.
\end{remark}

We can nonetheless prove a weaker semi-resolution statement for surfaces with only separable nodes:

\begin{proposition}\label{proposition:semi_resolution_in_char_2}
Let $S$ be a demi-normal surface with only separable nodes over an arbitrary field. Then there exists a proper birational morphism $f\colon T\to S$ such that
	\begin{enumerate}
		\item $T$ is slc $2$-Gorenstein with regular conductor $D_T$;
		\item $f$ is an isomorphism over a big open subset of $f$;
		\item no component of $D_T$ is $f$-exceptional;
		\item each component of $\Exc(f)$ is regular and intersects $D_T$ transversally, and $\Exc(f)$ has only normal crossings.
	\end{enumerate}
\end{proposition}
\begin{proof}
We repeat the construction of \autoref{theorem:semi_resolution_for_surfaces}. The only part of the proof of that does not extend in our generality is the study of the ring-theoretic singularities of $T$. By \autoref{prop:finite_equivalence_relation_give_nodes}, $T$ is a demi-normal surface. Write $D_{\bar{T}}:=\bar{f}^{-1}_*\bar{D}$. Then $(\bar{T},D_{\bar{T}})$ is lc (because it is log smooth) and $\omega_{\bar{T}}(D_{\bar{T}})|_{D_{\bar{T}}}=\omega_{D_{\bar{T}}}$ is $\tau$-invariant. Hence we may apply the proof of \cite[5.38]{Kollar_Singularities_of_the_minimal_model_program}, which works over any field as explained in \autoref{corollary:descent_in_codim_2}, and get that $\omega_T^{[2]}$ is invertible. Thus $T$ is $2$-Gorenstein and slc.
\end{proof}

\subsection{Embeddings of demi-normal schemes}

In this subsection, we prove the following theorem, which is a positive characteristic analog of the main result of \cite{Berquist_Embedding_of_demi_normal_varieties}.
\begin{theorem}\label{thm:embedding_in_demi_normal}
Let $U$ be a demi-normal scheme that is separated and of finite type over a scheme $B$ that is of finite type and separated over an arbitrary field of positive characteristic. Then there exists a demi-normal $\overline{U}$ that is proper over $B$, and a commutative diagram
		$$\begin{tikzcd}
		U\arrow[dr]\arrow[rr, hook, "j"] && \overline{U}\arrow[dl] \\
		& B &
		\end{tikzcd}$$
where $j$ is a dominant open embedding. In addition:
	\begin{enumerate}
		\item We can arrange so that the singular codimension one points of $\overline{U}$ are all contained in $U$.
		\item If $U$ is quasi-projective over $B$, then we can arrange so that $\overline{U}$ is projective over $B$.
	\end{enumerate}
\end{theorem}
\begin{proof}
By Nagata compactification \cite[0F41]{Stacks_Project} we can find a proper $V$ over $B$ such that $U$ embeds as a dense open $B$-subscheme of $V$. Let $n\colon V^n\to V$ be the normalization. Let $\bar{D}$ be the closure of the conductor $D_{U^n}\subset U^n$. Since $U$ is demi-normal we can write $\bar{D}=\bar{D}_G+\bar{D}_I$, where $\bar{D}_G$ is the preimage of the separable nodes and $\bar{D}_I$ is the preimage of the inseparable ones. We have an involution $\tau$ of $\bar{D}_G^n$ over $V$, and for every generic point $\eta\in \bar{D}_I$ the data of a purely inseparable degree $2$ sub-extension of $k(\eta)$. Notice that $R(\tau)\rightrightarrows V^n$ is finite, since it arises from a normalization. Thus we may apply \autoref{prop:finite_equivalence_relation_give_nodes} and \autoref{theorem: bijection for inseparable nodes} to obtain a finite morphism $V^n\to \overline{U}$ where $\overline{U}$ is demi-normal. The image of $U^n$ is open dense and isomorphic to $U$, and the codimension one singular points of $\overline{U}$ are contained in $U$ by construction.

We show that $\overline{U}$ is proper over $B$. The morphism $V^n\to\overline{U}$ can be decomposed as
		$$V^n\overset{q_1}{\longrightarrow} \tilde{U}\overset{q_2}{\longrightarrow} \overline{U}$$
where $q_1$ is the quotient by $R(\tau)$, and $q_2$ factors the Frobenius. Since $\tau$ commutes with the projection to $V$ on the dense open subset $D_{U^n}^n$, it does globally, and so $V^n\to V$ factors as $V^n\to \tilde{U}\overset{h}{\to} V$. The construction of $\sO_{\overline{U}}\subset \sO_{\tilde{U}}$ is given by \autoref{construction: construction of inseparable node} and it is easy to verify that $\sO_V\subset h_*\sO_{\tilde{U}}$. Thus $\tilde{U}\to V$ factors through $\overline{U}$. Since $V^n\to V$ and $V^n\to \overline{U}$ are finite, by \cite[01YQ]{Stacks_Project} and \cite[Theorem 1]{Artin_Tate_Note_on_finite_extensions} we see that $\overline{U}\to V$ is finite. In particular $\overline{U}$ is proper over $B$.

Finally, assume that $U$ is quasi-projective over $B$. Then instead of an arbitrary proper compactification, we can simply take $V$ to be a projective closure of $U$ over $B$. Let $H$ be a Cartier divisor on $V$ that is ample over $B$. Its pullback $H_{\overline{U}}$ is also ample over $B$, thus $\overline{U}$ is projective over $B$.
\end{proof}

\section{Gluing theory for surfaces}\label{section:gluing_for_surfaces}

\subsection{Gluing theory for surfaces in characteristic $> 2$}
In this section we prove the analog of \cite[5.13]{Kollar_Singularities_of_the_minimal_model_program} for surfaces in positive characteristic different from $2$. 

\begin{theorem}\label{theorem:gluing_for_surfaces_p_different_from_2}
Let $k$ be a field of characteristic $> 2$. Then normalization gives a one-to-one correspondence
	\begin{equation*}
		(\Char k > 2) \quad
		\begin{pmatrix}
		\text{Slc surface pairs } (S,\Delta)\\
		\text{of finite type over }k
		\end{pmatrix}
		\overset{1:1}{\longrightarrow}
		\begin{pmatrix}
		\text{Lc surface pairs } (\bar{S},\bar{D}+\bar{\Delta})  \\
		\text{of finite type over }k \\
		\text{plus an involution }\tau\text{ of }(\bar{D}^n,\Diff_{\bar{D}^n}\bar{\Delta})\text{ that is}\\
		\text{generically fixed point free on every component.}
		\end{pmatrix}
		\end{equation*}
	Moreover, $S$ is proper if and only if $\bar{S}$ is proper, and $K_S+\Delta$ is ample if and only if $K_{\bar{S}}+\bar{D}+\bar{\Delta}$ is ample.
\end{theorem}
\begin{proof}
Given $(\bar{S},\bar{D}+\bar{\Delta},\tau)$, we claim that the equivalence relation $R(\tau)\subset \bar{S}\times_k \bar{S}$ is finite. It suffices to show that the equivalence classes of the points on $\bar{D}$ are finite. By \cite[2.35]{Kollar_Singularities_of_the_minimal_model_program}, the closed subset $\Sigma=\Supp\Diff_{\bar{D}^n}\bar{\Delta}$ is equal to the preimage through $n\colon\bar{D}^n\to\bar{D}$ of the set of points $s\in \bar{D}$ such that: $s\in D\cap \Supp(\bar{\Delta})$, or $\bar{D}$ is singular at $s$, or $\bar{S}$ is singular at $s$. Since $\tau$ is an involution of the pair $(\bar{D}^n,\Diff_{\bar{D}^n}\bar{D})$, we have $\tau(\Sigma)=\Sigma$. Moreover $n^{-1}(n(\Sigma))=\Sigma$ because $\Sigma$ is a preimage through $n$, and $n\colon \bar{D}^n-\Sigma\longrightarrow \bar{D}-n(\Sigma)$ is an isomorphism. Therefore:
	\begin{enumerate}
		\item if $s\in \bar{D}-n(\Sigma)$, then the equivalence class of $s$ is equal to $\{s,(n\circ\tau)(s')\}$ where $s'\in\bar{D}^n$ is the unique point above $s$;
		\item if $s\in n(\Sigma)$ then the equivalence class of $s$ is contained in $n(\Sigma)$, which is finite.
	\end{enumerate}
This proves our claim. By \autoref{prop:finite_equivalence_relation_give_nodes}, the geometric quotient $\pi\colon \bar{S}\to \bar{S}/R(\tau)=:S$ exists. It is a demi-normal surface over $k$, and is proper if and only if $\bar{S}$ is proper.

Set $\Delta:=\pi_*\bar{\Delta}$. \autoref{lemma:trace_map} shows that the proof of \cite[5.18]{Kollar_Quotients_by_finite_equivalence_relations} is also valid over imperfect fields. Thus $K_S+\Delta$ is $\bQ$-Cartier. Since $\pi$ is an isomorphism from $\bar{S}-\bar{D}$ onto $S-\pi(\bar{D})$ and $\Delta$ has no components along $\pi(\bar{D})$, it follows from \autoref{lemma: basic properties of normalization of demi-normal scheme} that $\pi^*(K_S+\Delta)\sim_\bQ K_{\bar{S}}+\bar{D}+\bar{\Delta}$.

This gives a map $\Phi\colon \{(\bar{S},\bar{D}+\bar{\Delta},\tau)\}\to \{(S,\Delta)\}$, while the normalization gives a map $\nu$ in the other direction. We have $\nu\circ\Phi=\Id$ by \autoref{prop:finite_equivalence_relation_give_nodes}, and $\Phi\circ \nu=\Id$ by \autoref{proposition:structure_of_normalization_of_demi_normal_scheme}. Therefore the theorem is proved.
\end{proof}

\begin{remark}\label{remark:proof_in_dim_2}
The key point of the proof of \autoref{theorem:gluing_for_surfaces_p_different_from_2} is to find the $\tau$-invariant proper subset $\Sigma\subset \bar{D}^n$ with the property that $D^n-\Sigma\longrightarrow \bar{D}-n(\Sigma)$ is an isomorphism. Its existence is given by the fact that $\tau$ preserves the lc stratification of $(\bar{D}^n,\Diff_{\bar{D}^n}\bar{\Delta})$. The fact that $\dim \bar{S}=2$ then implies that $n(\Sigma)$ is finite, from which we deduce finiteness of the equivalence relation. In particular we do not need any positivity assumption on $K_{\bar{S}}+\bar{D}+\bar{\Delta}$ to conclude. 

If $\dim \bar{S}>2$ then the proof hints at an inductive process on the lc strata of $(\bar{D}^n,\Diff_{\bar{D}^n}\bar{\Delta})$. This seems dubious at first, since we have no control on the singularities of the higher codimension strata. In characteristic $0$, the theory of source and springs of lc centers \cite[\S 4.5]{Kollar_Singularities_of_the_minimal_model_program} provides a well-behaved replacement for the strata. In \autoref{section: springs and sources}, we will develop a similar theory for threefolds in positive characteristic. However, it seems that positivity assumptions are needed to ensure that finiteness holds in higher codimension.
\end{remark}

\subsection{Gluing theory for surfaces in characteristic $2$}
We prove the analog of \cite[5.13]{Kollar_Singularities_of_the_minimal_model_program} for surfaces in characteristic $2$. It is given by a combination of \autoref{theorem: bijection for inseparable nodes} and the method of \autoref{theorem:gluing_for_surfaces_p_different_from_2}.

\begin{lemma}\label{remark:gluing_works_in_char_2_for_separable_nodes}
Let $(\bar{S},\bar{D}+\bar{\Delta})$ be an lc surface pair over of finite type a field $k$ of characteristic $2$, and $\tau$ a log involution of $(\bar{D}^n,\Diff_{\bar{D}^n}(\bar{\Delta})$. Then:
	\begin{enumerate}
		\item $R(\tau)\rightrightarrows \bar{S}$ is finite, and induces a finite quotient morphism $\pi\colon \bar{S}\to S$;
		\item $(S,\pi_*\bar{\Delta})$ is slc of finite type over $k$, with normalization $(\bar{S},\bar{D}+\bar{\Delta},\tau)$.
	\end{enumerate}
\end{lemma}
\begin{proof}
We may apply the same gluing process as in \autoref{theorem:gluing_for_surfaces_p_different_from_2}. Indeed, \cite[2.35]{Kollar_Singularities_of_the_minimal_model_program} and \autoref{prop:finite_equivalence_relation_give_nodes} hold in every positive characteristic, and descent of the log canonical divisor holds in dimension two by \autoref{corollary:descent_in_codim_2}.
\end{proof}

\begin{theorem}\label{theorem:gluing_for_surfaces_in_char_2}
Let $k$ be a field of characteristic $2$. Then normalization gives a one-to-one correspondence
		\begin{equation*}
		(\Char k = 2) \quad
		\begin{pmatrix}
		\text{Slc surface pairs }(S,\Delta)\\
		\text{of finite type over }k
		\end{pmatrix}
		\overset{1:1}{\longrightarrow}
		\begin{pmatrix}
		\text{Lc surface pairs } (\bar{S},\bar{D}_\emph{Gal}+\bar{D}_\emph{ins}+\bar{\Delta})  \\
		\text{of finite type over }k \\
		\text{plus an involution }\tau\\
		\text{of }(\bar{D}_\emph{Gal}^n,\Diff_{\bar{D}_\emph{Gal}^n}(\bar{\Delta}+\bar{D}_\emph{ins})) \text{ that is}\\
		\text{generically fixed point free on every component.}
		\end{pmatrix}
		\end{equation*}
Moreover, $S$ is proper if and only if $\bar{S}$ is proper, and $K_S+\Delta$ is ample if and only if $K_{\bar{S}}+\bar{D}_\emph{Gal}+\bar{D}_\emph{ins}+\bar{\Delta}$ is ample.
\end{theorem}
\begin{proof}
We construct an inverse map in two steps. First of all, thanks to \autoref{remark:gluing_works_in_char_2_for_separable_nodes} we can apply the method of \autoref{theorem:gluing_for_surfaces_p_different_from_2} to the triplet $(\bar{S},\bar{D}_\text{Gal}+\bar{\Gamma},\tau)$ where $\bar{\Gamma}:=\bar{D}_\text{ins}+\bar{\Delta}$. We obtain an slc pair $(\tilde{S}, \tilde{\Gamma}=\tilde{D}_\text{ins}+\tilde{\Delta})$. Then, we can apply \autoref{theorem: bijection for inseparable nodes} to it and obtain an slc pair $(S,\Delta)$. This gives a map
		$$(\bar{S},\bar{D}_\text{Gal}+\bar{D}_\text{ins}+\bar{\Delta},\tau)\overset{\Phi_\text{Gal}}{\xrightarrow{\hspace{1.5cm}}} (\tilde{S},\tilde{D}_\text{ins}+\tilde{\Delta})\overset{\Phi_\text{ins}}{\xrightarrow{\hspace{1.5cm}}} (S,\Delta).$$
	It follows from \autoref{proposition:structure_of_normalization_of_demi_normal_scheme} that $\Phi_\text{ins}\circ\Phi_\text{Gal}$ is an inverse map to the normalization. 
	The statements about properness and ampleness follow from the corresponding statement in \autoref{theorem: bijection for inseparable nodes} and \autoref{remark:gluing_works_in_char_2_for_separable_nodes}.
\end{proof}

\subsection{Gluing theory for germs of surfaces}
Our gluing theorems for surfaces, \autoref{theorem:gluing_for_surfaces_p_different_from_2} and \autoref{theorem:gluing_for_surfaces_in_char_2}, are formulated for surfaces that are quasi-projective. It is natural to ask whether similar statements hold for germs of surfaces. If $(\bar{S},\bar{D}+\bar{\Delta})$ is a germ of lc surface with a log involution $\tau$ on $(\bar{D}^n, \Diff\bar{\Delta})$, then the argument in the proof of \autoref{theorem:gluing_for_surfaces_p_different_from_2} applies verbatim to show that $R(\tau)\rightrightarrows \bar{S}$ is a finite equivalence relation, and \cite[Theorem 6]{Kollar_Quotients_by_finite_equivalence_relations} applies to schemes that are essentially of finite type over $k$. So we obtain a quotient $\bar{S}/R(\tau)$. However I am not aware of an Eakin--Nagata type theorem for rings that are essentially of finite type over $k$, so it is unclear what type of algebraic space the quotient $\bar{S}/R(\tau)$ is. It turns out that it is a germ of variety, but our proof is somewhat roundabout. The details are given in the next theorem.

\begin{theorem}\label{theorem:gluing_for_germs_of_surfaces}
Let $k$ be a field of positive characteristic. Then:
	\begin{enumerate}
		\item If $\Char k>2$, then normalization gives a one-to-one correspondence
		\begin{equation*}
		\begin{pmatrix}
		\text{Slc semi-local affine }\\
		\text{surface pairs } (S,\Delta)\\
		\text{essentially of finite type over }k
		\end{pmatrix}
		\overset{1:1}{\longrightarrow}
		\begin{pmatrix}
		\text{Lc semi-local affine }\\
		\text{surface pairs } (\bar{S},\bar{D}+\bar{\Delta})  \\
		\text{essentially of finite type over }k \\
		\text{plus an involution }\tau\text{ of }(\bar{D}^n,\Diff_{\bar{D}^n}\bar{\Delta})\text{ that is}\\
		\text{generically fixed point free on every component.}
		\end{pmatrix}
		\end{equation*}
		
		\item If $\Char k=2$, then normalization gives a one-to-one correspondence
		\begin{equation*}
		\begin{pmatrix}
		\text{Slc semi-local affine } \\
		\text{surface pairs }(S,\Delta)\\
		\text{essentially of finite type over }k
		\end{pmatrix}
		\overset{1:1}{\longrightarrow}
		\begin{pmatrix}
		\text{Lc semi-local affine }\\
		\text{surface pairs } (\bar{S},\bar{D}_\emph{Gal}+\bar{D}_\emph{ins}+\bar{\Delta})  \\
		\text{essentially of finite type over }k\text{, plus an involution} \\
		\tau\text{ of }(\bar{D}_\emph{Gal}^n,\Diff_{\bar{D}_\emph{Gal}^n}(\bar{\Delta}+\bar{D}_\emph{ins}))\text{ that is}\\
		\text{generically fixed point free on every component.}
		\end{pmatrix}
		\end{equation*}
	\end{enumerate}
Moreover, in both cases $S$ is local if and only if all closed points of $\bar{S}$ belong to the same $R(\tau)$-equivalence class.
\end{theorem}
\begin{proof}
We show that, starting with an lc semi-local affine pair $(\Spec\sO,D_\sO+\Delta_\sO)$ essentially of finite type over $k$ with an involution $\tau$, we can produce an slc semi-local affine pair $(\Spec\sO',\Delta_{\sO'})$ essentially of finite type over $k$ and a finite morphism $\Spec\sO\to\Spec\sO'$ that is the quotient by $R(\tau)$. Then we show that \autoref{construction: construction of inseparable node} preserves the property of being essentially of finite type over $k$. The proof that the combination of these operations is the inverse of the normalization is exactly the same as in \autoref{theorem:gluing_for_surfaces_p_different_from_2} and \autoref{theorem:gluing_for_surfaces_in_char_2}.

Hence let $\sO$ be a semi-local ring of equi-dimension two that is essentially of finite type over $k$, $D_\sO$ a reduced divisor and $\Delta_\sO$ an effective $\bQ$-divisor on $\Spec\sO$, with no component in common. Since $\sO$ is essentially of finite type over $k$, there exists a scheme $\bar{S}$ of finite type over $k$, such that $\Spec\sO$ is a subscheme of $\bar{S}$. Replacing $\bar{S}$ by the closure of $\Spec\sO$, we may assume that $\bar{S}$ is two-dimensional.

Assume furthermore that $(\Spec\sO,D_\sO+\Delta_\sO)$ is lc. Then $\Spec\sO$ belongs to the normal locus of $\bar{S}$, so taking the normalization we may assume that $\bar{S}$ is normal. Taking an open subset, we may actually assume that $\bar{S}$ is regular away from the closed points of $\Spec\sO$.

Let $\bar{D}$ and $\bar{\Delta}$ be the closures (as $\bQ$-divisors) of $D_\sO$ and $\Delta_\sO$ in $\bar{S}$. We may shrink $\bar{S}$ until $\bar{D}$ is regular away from the closed points of $\Spec\sO$, and $\bar{D}\cap\bar{\Delta}$ is supported on the closed points of $\Spec\sO$. Then $K_{\bar{S}}+\bar{D}+\bar{\Delta}$ is $\bQ$-Cartier, $(\bar{S}, \bar{D}+\bar{\Delta})$ is lc and $\Diff_{\bar{D}^n}\bar{\Delta}=\Diff_{D_\sO^n}\Delta_\sO$ by \cite[2.35]{Kollar_Singularities_of_the_minimal_model_program}.

Let $\tau$ be a generically fixed point free involution of $(D^n_\sO,\Diff_{D^n_\sO}\Delta_\sO)$. Then $\tau$ extends to an involution $\bar{\tau}$ on the projective regular model $\sD$ of $D^n_\sO$. Since $\bar{D}^n$ is a dense open subset of $\sD$, the set $\sD-\bar{D}^n$ is finite. As $\bar{\tau}$ is an involution and $D^n_\sO$ is $\tau$-stable, we can find finitely many closed points $x_1,\dots,x_n\in \bar{D}^n-D^n_\sO$ such that $\bar{D}^n-\{x_1,\dots,x_n\}$ is $\bar{\tau}$-stable. Since $\bar{D}$ is regular away from $D_\sO$, the $x_i$ correspond to unique closed points of $\bar{S}$. Therefore after shrinking $\bar{S}$, we may assume that $\bar{\tau}$ gives an involution of $\bar{D}^n$. Moreover, since $\Diff_{D^n_\sO}\Delta_\sO=\Diff_{\bar{D}^n}\bar{\Delta}$, we see that $\bar{\tau}$ is a generically fixed point free involution of  the pair $(\bar{D}^n,\Diff_{\bar{D}^n}\bar{\Delta})$.

Thus we may apply \autoref{theorem:gluing_for_surfaces_p_different_from_2} and \autoref{remark:gluing_works_in_char_2_for_separable_nodes} to obtain a finite quotient $q\colon \bar{S}\to \bar{S}/R(\bar{\tau})=:S$. Moreover the pair $(S,\Delta:=q_*\bar{\Delta})$ is an slc surface pair of finite type over $k$ and $q$ is the normalization morphism.

To simplify the rest of the discussion, we reduce to the case where $S$ and $\bar{S}$ are affine. Let $\Sigma\subset S$ be the finite reduced subscheme supported on the image of the closed points of $\Spec\sO$. There exists an affine open subscheme $S'\subset S$ containing $\Sigma$ (because $S$ satisfies the Chevalley--Kleiman property, see \cite[9.28]{Kollar_Singularities_of_the_minimal_model_program}). Then $q^{-1}S'$ is affine open and contains $\Spec\sO$, and $q^{-1}S'\to S'$ is the quotient by the induced equivalence relation on $q^{-1}S'$ \cite[9.11]{Kollar_Singularities_of_the_minimal_model_program}.

So say $S=\Spec A$, $\bar{S}=\Spec \bar{A}$ and $\Sigma=V(\fm_1)\cup\dots\cup V(\fm_n)$. Then $T:=A-\bigcup_i\fm_i$ is a multiplicatively closed set, $T^{-1}A$ is a semi-local ring essentially of finite type over $k$ with maximal ideals $\fm_i$ and $T^{-1}\bar{A}$ is the normalization of $T^{-1}\bar{A}$ inside $\Frac A=\Frac\bar{A}$. On the other hand, $\sO$ is a fraction ring of $\bar{A}$, say $\sO=U^{-1}\bar{A}$. Since the maximal ideals of $U^{-1}\bar{A}$ are precisely those above the $\fm_i$'s, we see that $U^{-1}\bar{A}$ is a further localization of $T^{-1}\bar{A}$ and that these two rings have the same maximal ideals. But this implies that $T^{-1}\bar{A}=U^{-1}\bar{A}$ already. Applying \cite[9.11]{Kollar_Singularities_of_the_minimal_model_program} once again we obtain that $\Spec\sO\to \Spec T^{-1}A$ is the quotient by the finite equivalence relation $R(\tau)$, and by construction $(\Spec T^{-1}A, \Delta_{T^{-1}A})$ is a semi-local slc surface pair essentially of finite type over $k$.

It remains to show that \autoref{construction: construction of inseparable node} applied to a semi-local ring $\sO$ essentially of finite type over $k$, produces a semi-local ring essentially of finite type over $k$. The construction output is a ring $\sA$ and successive finite extensions
		$\sO^2\subset \sA\subset \sO$
which shows that $\sA$ is semi-local. Since $\sO$ is essentially of finite type over $k$, so is $\sO^2$, and as $\sO^2\subset \sA$ is a finite extension we obtain that $\sA$ is essentially of finite type over $k$. This completes the proof.
\end{proof}

\section{Gluing theory for threefolds}

\subsection{Preliminary results}
In this first section some results that will be needed for the proof of the gluing theorem for threefolds. While most of them are well-known, we include proofs to conveniently reference them.

\subsubsection{Some results about surfaces and threefolds}
We gather some facts about the geometry of surface and threefold pairs.

\begin{fact}[Birational Geometry of excellent surfaces]\label{fact:birational_geometry_of_surfaces}
Resolutions of singularities (resp. log resolutions) exist for excellent surfaces (resp. excellent surface pairs) \cite[2.25]{Kollar_Singularities_of_the_minimal_model_program}. Moreover, the MMP works for quasi-projective lc surface pairs over an arbitrary field \cite{Tanaka_MMP_for_excellent_surfaces}.
\end{fact}

\begin{fact}[Resolutions for threefolds]\label{fact:resolution_of_singularities_for_threefolds}
Resolutions of singularities exist for quasi-projective threefolds over a perfect field \cite[2.1]{Cossart_Piltant_Resolution_of_singularities_for_3_folds_in_pos_char_I}. Moreover, if $(X,\Delta)$ is a quasi-projective threefold pair, then there exists a snc log resolution $(X',\Delta')\to (X,\Delta)$ that is an isomorphism above the snc locus of $(X,\Delta)$ \cite[2.1,4.1]{Cossart_Piltant_Resolution_of_singularities_for_3_folds_in_pos_char_I}.
\end{fact}

\begin{corollary}
If $(S,\Delta)$ is a numerically dlt (resp. numerically terminal) surface pair over an arbitrary field, then $S$ is $\bQ$-factorial (resp. regular).
\end{corollary}
\begin{proof}
The numerically terminal case is shown in \cite[2.29]{Kollar_Singularities_of_the_minimal_model_program}. To prove the numerically dlt case, we can apply \cite[4.11]{Kollar_Mori_Birational_geometry_of_algebraic_varieties} once we know that \cite[4.10]{Kollar_Mori_Birational_geometry_of_algebraic_varieties} applies as well. It is the case, since log resolutions are available, negative-definiteness of contracted curves holds by \cite[10.1]{Kollar_Singularities_of_the_minimal_model_program} and the Base-point freeness theorem is established with a sufficient level of generality in \cite[4.2]{Tanaka_MMP_for_excellent_surfaces}.
\end{proof}

\begin{fact}[Inversion of adjunction for threefolds]
Log canonical inversion of adjunction holds for quasi-projective threefold pairs over a perfect field of characteristic $>5$ \cite[Lemma 3.3]{Patakfalvi_Projectivity_moduli_space_of_surfaces_in_pos_char} (the proof there is given for an algebraically closed field, but extends easily to the case of a perfect field).
\end{fact}

Next we gather some results about fibrations of surfaces.

\begin{lemma}\label{lemma:Square_zero_vertical_is_fiber}
Let $f\colon S\to T$ be a proper flat morphism from a normal surface onto a normal curve with connected fibers, over a arbitrary field. Let $E$ be a $\bQ$-divisor on $S$ that is vertical over $T$. Then $E^2\leq 0$, with equality if and only if $E$ is a weighted sum of fibers of $f$.
\end{lemma}
\begin{proof}
By \cite[2.6]{Badescu_Agebraic_surfaces}, the result holds if $S$ is regular. It is assumed there that the base field $k$ is algebraically closed, but this is not necessary: the Néron--Severi theorem holds over any field \cite[Exp. XIII, Théorème 5.1]{SGA6}, the Grothendieck--Riemann--Roch formula \cite[Exp. VIII, Théorème 3.6]{SGA6} holds for $S$ since $S\to k$ is a complete intersection morphism, and it reduces to the usual formula by the arguments of \cite[Appendix A, 4.1.2]{Hartshorne_Algebraic_Geometry}. The rest of the proof in \cite{Badescu_Agebraic_surfaces} is linear algebra.

In general, let $\pi\colon S'\to S$ be a (minimal) resolution of singularities. Then $f\circ\pi\colon S'\to T$ has connected fibers and $(\pi^*E)^2=E^2$. Moreover $\pi^*E$ is a weighted sum of fibers if and only if $E$ is, so the lemma is proved.
\end{proof}

\begin{lemma}\label{lemma:Fiber_of_MFS_are_normal}
Let $f\colon S\to T$ be a proper flat morphism from a normal surface onto a normal curve, such that $-K_S$ is $f$-ample and $f_*\sO_S=\sO_T$, over an arbitrary field of characteristic $p>2$. Then the general fiber of $f$ is geometrically normal.
\end{lemma}
\begin{proof}
To begin with, we show that the generic fiber of $f$ is geometrically normal, assuming it is geometrically integral. Let $\eta\in T$ be the generic point, then $S_{\eta}$ is a normal scheme of finite type over $k(\eta)$. Since $f_*\sO_S=\sO_T$, the field $k(\eta)$ is algebraically closed in $k(S_{\eta})$. Then if $Y$ is the normalization of the reduced structure of $S_{\overline{\eta}}$, with canonical morphism $\pi\colon Y\to S_{\eta}$, by \cite[Theorem 1.1]{Patakfalvi_Waldron_Singularities_of_general_fibers} there is an effective $\bZ$-Weil divisor $C$ on $Y$ such that
		\begin{equation}\label{eqn:inseparable_base_change}
		K_{S_{\overline{\eta}}}+(p-1)C\sim \pi^*K_{S_{\eta}}.
		\end{equation}
Since $S_{\overline{\eta}}$ is reduced by the assumption we made, by \cite[Theorem 1.2]{Patakfalvi_Waldron_Singularities_of_general_fibers} we may choose $C$ so that $(p-1)C$ is equal to the divisorial part of the conductor of the normalization $Y\to S_{\overline{\eta}}$. Notice that $S_{\overline{\eta}}$ satisfies the property $S_2$, since it is preserved by field-extension \cite[6.4.2]{EGA_IV.2} and $S_{\eta}$ satisfies it. Thus $S_{\overline{\eta}}$ is normal if and only if it is $R_1$, so we can choose $C=0$ if and only if $S_{\overline{\eta}}$ is normal.

Suppose $S_{\overline{\eta}}$ is not normal. Then we may assume that $C>0$. Since $-K_{S_{\eta}}$ is ample over $k(\eta)$, its pullback $-\pi^*K_{S_{\eta}}$ is ample over $\overline{k(\eta)}$. By \autoref{eqn:inseparable_base_change} we deduce that $-K_{S_{\overline{\eta}}}-(p-1)C$ is ample. Since $C>0$, we deduce that $-K_{S_{\overline{\eta}}}$ is already ample, and therefore $S_{\overline{\eta}}\cong \bP_{\overline{k(\eta)}}$. Looking again at the equation \autoref{eqn:inseparable_base_change} and taking degrees, we see that $\deg (p-1)C< 2$. Since $C>0$, this condition is satisfied only if $p=2$.

Now we show that $S_{\eta}$ is geometrically integral. Since $f_*\sO_S=\sO_T$, the field $k(\eta)$ is algebraically closed in $k(S_{\eta})$. Therefore the function field of $T$ is separable over the function field of $S$ \cite[7.2]{Badescu_Agebraic_surfaces}. It follows from \cite[4.5.9, 4.6.1]{EGA_IV.2} that $S_{\eta}$ is geometrically integral.

Finally, by \cite[12.2.4]{EGA_IV.3} the set of points $t\in T$ such that the fiber $S_{t}$ is geometrically normal, is open.
\end{proof}

\begin{lemma}\label{lemma:sections_of_fibered_curves}
Let $f\colon S\to T$ be a proper flat morphism from an integral surface to an integral curve. Let $D$ be a prime $\bQ$-Cartier divisor such that $\deg_{k(t)}D|_{S_t}=1$ for every $t\in T$. Assume that $S$ is normal, or more generally that the Cartier locus of $D$ dominates $T$. Then $f|_D\colon D\to T$ is a birational morphism, and an isomorphism if $T$ is normal.
\end{lemma}
\begin{proof}
Notice that $S_t$ is proper for every $t\in T$, so the intersection product $D.S_t$ is well-defined (see \cite[Appendix A.2]{Kollar_Rational_curves}). Let us assume first that $D$ is Cartier. Then by \cite[A.2.8]{Kollar_Rational_curves} (applied with $X=S_t$ and $F=\sO_{S_t}$) we have
		$$1=\deg_{k(t)}D|_{S_t}=\sum_{x\in D\cap S_t}\length_{k(t)}\left(\sO_{S_t,x}/(\xi_x)\right)$$
	where $\xi_x$ is the restriction to $S_t$ of the local equation of $D$ at $x$. This implies that $D$ meets $S_t$ at a single point, and that the canonical map $\sO_T\to (f|_D)_*\sO_D$ is surjective after tensoring by $k(t)$. So $f|_D$ is proper and quasi-finite, therefore finite, and $(f|_D)_*\sO_D$ is a finite $\sO_T$-module. By Nakayama's lemma, we obtain that $\sO_T\to (f|_D)_*\sO_D$ is surjective. But it is also injective, so it is an isomorphism. So $f|_D$ is an isomorphism.
	
	If $D$ is only Cartier over a dense open subset of $T$, we obtain that $f|_D\colon D\to T$ is birational. Assume that $T$ is normal: then the composition $D^n\to D\to T$ is a birational morphism of normal curves, hence an isomorphism, so actually $D^n=D\cong T$.
\end{proof}

\begin{lemma}\label{lemma: dlt surface over a surface}
Let $(S,\Delta)$ be a dlt surface pair and $g\colon S\to T$ be a birational proper morphism onto a (non-necessarily normal) surface, over a perfect field. Assume that $K_S+\Delta\sim_{\bQ,g}0$. Then $R^1g_*\sO_S(-\Delta^{=1})=0$.
\end{lemma}
\begin{proof}
By \'{e}tale base-change we may assume that the base-field is algebraically closed. Since $S$ is normal, there is a factorization $g^\nu\colon S\to T^\nu$ through the normalization $\nu\colon T^\nu\to T$. Since $\nu$ is finite, for any coherent sheaf $\sF$ on $S$ we have $R^1(\nu\circ g)_*\sF\cong  \nu_*R^1g^\nu_*\sF$. Therefore we may assume that $T$ is normal.

Let $\varphi\colon S'\to (S,\Delta)$ is a minimal log resolution. Since $(S,\Delta^{<1})$ is klt, we can write $\varphi^*(K_S+\Delta^{<1})=K_{S'}+\Delta'$ where $\Delta'\geq 0$ and $\lfloor \Delta'\rfloor=0$. 

We notice that since $S$ is smooth in a neighborhood of $\Delta^{=1}$, the projection formula yields $\varphi_*\sO_{S'}(-\varphi^*\Delta^{=1})=\sO_S(-\Delta^{=1})$ (such projection formula holds in greater generality, see for example \cite[2.1, 2.4]{Sakai_Divisors_normal_surfaces}).


By assumption, there is a $\bQ$-Cartier divisor $N$ on $T$ such that $-\Delta^{=1}\sim_{\bQ} K_S+\Delta^{<1}+g^*N$. It follows that
		$$-\varphi^*\Delta^{=1}\sim_\bQ K_{S'}+\Delta'+\varphi^*(g^*N).$$
We can write $\Delta'=V_{\varphi}+H_{\varphi}$, where $V_{\varphi}$ is $\varphi$-exceptional and $H_{\varphi}$ has no $\varphi$-exceptional components. Then $H_{\varphi}$ is $\varphi$-nef and $\lfloor V_{\varphi}\rfloor =0$. So by  \cite[2.2.5]{Kollar_Kovacs_Birational_Geometry_of_log_surfaces} (see also \cite[10.4]{Kollar_Singularities_of_the_minimal_model_program}), we obtain that $R^i\varphi_*\sO_{S'}(-\varphi^*\Delta^{=1})=0$ for $i>0$. A similar argument with the composition $g'=g\circ\varphi\colon S'\to T$ in place of $\varphi$ shows that $R^1g'_*\sO_{S'}(-\varphi^*\Delta^{=1})=0$. As $\varphi_*\sO_{S'}(-\varphi^{*}\Delta^{=1})=\sO_S(-\Delta^{=1})$, the Leray spectral sequence for $g'=g\circ \varphi$ gives that
		$$0=R^1g'_*\sO_{S'}(-\varphi^*\Delta^{=1})\cong R^1g_*\sO_S(-\Delta^{=1})$$
and the proof is complete.
\end{proof}

\subsubsection{Crepant birational maps}
\begin{definition}
Let $f\colon (X',\Delta')\to (X,\Delta)$ be a birational proper morphism of pairs. We say that $f$ is \textbf{crepant} if $K_{X'}+\Delta'= f^*(K_X+\Delta)$.
\end{definition}

\begin{definition}
Let $(X,\Delta)$ be a pair. A \textbf{crepant dlt blow-up} of $(X,\Delta)$ is a projective birational morphism $f\colon Y\to X$ such that
	\begin{enumerate}
		\item $Y$ is $\bQ$-factorial, and
		\item $(Y, f^{-1}_*\Delta+E)$ is a dlt pair, where $E$ is the sum of all $f$-exceptional divisors with coefficients $1$, and
		\item $K_Y+f^{-1}_*\Delta+E\sim_\bQ f^*(K_X+\Delta)$.
	\end{enumerate}
\end{definition}

The following fact will be crucial for our study of lc centers on threefolds.

\begin{fact}\label{fact:crepant_dlt_blow_up_ind_dim_2_3}
Crepant dlt blow-ups exist for lc surface pairs over an arbitrary field (see \autoref{fact:birational_geometry_of_surfaces}), and for quasi-projective lc threefold pairs over a perfect field of characteristic $>5$ \cite[3.6]{Hashizume_Nakamura_Tanaka_MMP_for_lc_threefolds_in_char_p}.
\end{fact}

\begin{definition}
More generally, a birational map $f\colon (X',\Delta')\dashrightarrow (X,\Delta)$ of pairs is \textbf{crepant} if there is a normal variety $Y$, a (non-necessarily effective) $\bQ$-divisor $\Delta_Y$ and a commutative diagram
		$$\begin{tikzcd}
		& (Y,\Delta_Y)\arrow[dl, "v'" above left]\arrow[dr, "v"] & \\
		(X',\Delta') \arrow[rr, dotted, "f"] && (X,\Delta)
		\end{tikzcd}$$
where $v,v'$ are proper, such that
		$$(v')^*(K_{X'}+\Delta')= K_Y+\Delta_Y = v^*(K_X+\Delta).$$
	The set of crepant birational self-map of $(X,\Delta)$ with the composition forms a group, denoted $\Bir^c(X,\Delta)$.
	
	If $X$ is endowed with a morphism $X\to Z$, we let $\Bir^c_Z(X,\Delta)$ be the subgroup of crepant birational self-maps over $Z$.
\end{definition}

\begin{lemma}\label{lemma:crepant_pairs_and_non_klt_locus}
Let $\phi\colon (S,\Delta)\dashrightarrow (S',\Delta')$ be a crepant birational map between two excellent surface pairs over an arbitrary field. Then:
	\begin{enumerate}
		\item $(S,\Delta)$ is klt if and only if $(S',\Delta')$ is klt;
		\item more generally, there is a bijection between the connected components of $\Nklt(S,\Delta)$ and those of $\Nklt(S',\Delta')$.
	\end{enumerate}
\end{lemma}
\begin{proof}
An equivalent definition of being crepant is that $a(E;S,\Delta)=a(E;S',\Delta')$ for every prime divisor $E$ of $k(S)=k(S')$ \cite[2.32.2]{Kollar_Singularities_of_the_minimal_model_program} so the first statement holds. To prove the other one, let $(Y,\Delta_Y)$ be a crepant snc resolution of $\phi$. Then $\Nklt(S,\Delta)$, respectively $\Nklt(S',\Delta')$, is the image of $\lfloor\Delta_Y^{>0}\rfloor$ through $Y\to S$, respectively through $Y\to S'$. By \cite[2.36]{Kollar_Singularities_of_the_minimal_model_program}, the fibers of
		$$\lfloor\Delta_Y^{>0}\rfloor\to \Nklt(S,\Delta),\quad \lfloor\Delta_Y^{>0}\rfloor\to \Nklt(S',\Delta')$$
	are connected, thus each morphism induces a bijection between the connected components of the target and those of the source. The result follows.
\end{proof}

\begin{remark}
The first part of \autoref{lemma:crepant_pairs_and_non_klt_locus} is true in every dimension. Moreover, it follows from the connectedness principle proved in \cite[Theorem 1.2]{Nakamura_Tanaka_Witt_Nadel_vanishing_for_threefolds} that the second part is true for threefolds over a perfect field of characteristic $>5$.
\end{remark}

We will frequently encounter pairs $(X,\Delta)$ together with a proper morphism onto a normal variety $X\to Z$, satisfying the condition $K_X+\Delta\sim_{\bQ,Z}0$. It will be useful to understand the crepant birational $Z$-maps of such pairs.

\begin{lemma}\label{lemma:birational_crepant_pairs}
Let $(X,\Delta)$ be a pair, $f\colon X\to Z$ a proper surjective morphism onto a normal variety such that $K_X+\Delta\sim_{\bQ,f}0$. Let $X'$ be a normal variety with a proper surjective morphism $f'\colon X'\to Z$, and $\phi\colon X\dashrightarrow X'$ a birational $Z$-map. Then:
	\begin{enumerate}
		\item There exists a unique $\bQ$-divisor $\Delta'$ on $X'$ such that $(X',\Delta')$ is a sub-pair and $\phi\colon (X,\Delta)\dashrightarrow (X',\Delta')$ is crepant;
		\item $K_{X'}+\Delta'\sim_{\bQ,f'}0$;
		\item If $\phi^{-1}$ does not extract divisors, then $\Delta'=\phi_*\Delta$. So $(X',\Delta')$ is a pair in this case.
	\end{enumerate}
\end{lemma}
\begin{proof}
Uniqueness of $\Delta'$ is clear, we prove its existence. We can find a commutative diagram
		$$\begin{tikzcd}
		& Y\arrow[dl, "u" above left]\arrow[dr, "v"] &\\
		X\arrow[rr, dotted, "\phi"]\arrow[dr, "f" below left] & & X'\arrow[dl, "f'"] \\
		& Z &
		\end{tikzcd}$$
where $Y$ is normal, $u$ and $v$ are birational, $u_*\sO_Y=\sO_X$ and $v_*\sO_Y=\sO_{X'}$. Write $K_Y+\Gamma=u^*(K_X+\Delta)$. We claim that $\Delta':=v_*\Gamma$ is a valid choice. By commutativity of the diagram, there is a $\bQ$-Cartier divisor $N$ on $Z$ such that $(f'\circ v)^*N\sim_\bQ K_Y+\Gamma$. Fix a canonical divisor $K_{X'}$ on $X'$ such that $v_*K_Y=K_{X'}$. We have:
		$$(f')^*N\sim_\bQ v_*(f'\circ v)^*N\sim_\bQ v_*(K_Y+\Gamma)=K_{X'}+\Delta'$$
which shows that $K_{X'}+\Delta'$ is $\bQ$-Cartier. Now $\pm ((K_Y+\Gamma)-v^*(K_{X'}+\Delta'))$ are $v$-exceptional and $\bQ$-linearly trivial over $Z$. In particular both are $v$-nef, and by the negativity lemma \cite[3.39]{Kollar_Mori_Birational_geometry_of_algebraic_varieties} we deduce that $K_Y+\Gamma=v^*(K_{X'}+\Delta')$. This shows that $\phi\colon (X,\Delta)\dashrightarrow (X',\Delta')$ is crepant.

If $\phi^{-1}$ does not extract divisors, then any $u$-exceptional divisor is also $v$-exceptional, so $v_*\Gamma=\phi_*\Delta$.
\end{proof}

\subsubsection{Pluricanonical representations in low dimensions}\label{section:pluricanonical_representations}
The key to the gluing theorems in \autoref{section: proof of main result} is a finiteness result in the theory of pluricanonical representations. See \cite[Theorem 7]{Kollar_Sources_of_lc_centers} and \cite[\S 10.5]{Kollar_Singularities_of_the_minimal_model_program} for what is known in characteristic $0$. I am not aware of similar general results in positive characteristic. Fortunately, for the gluing of threefolds there are only a few easy cases to consider, which we discuss below. The results are probably well-known, we give proofs for convenience.

\begin{proposition}\label{proposition:pluricanonical_representations_for_curves_I}
Let $(C,E)$ be a proper dlt curve over a perfect field $k$ such that $K_C+E\sim_\bQ 0$. Then $\im\left[ \Bir^c_k(C,E)\to \Aut_k H^0(C,\omega_C^m(mE))\right]$ is finite, for $m$ divisible enough.
\end{proposition}
\begin{proof}
We may extend the scalars along an algebraic closure of $k$, and assume it is algebraically closed. By assumption on $E$, the curve $C$ is smooth proper of genus $0$ or $1$. Moreover $\Bir^c_k(C,E)=\Aut_k(C,E)$.
	\begin{enumerate}
		\item If $C$ has genus $1$, then $\omega_C\sim 0$ and so  $E=0$. Fix a closed point $o\in C$, and consider the elliptic curve $(C,o)$. The $1$-dimensional $k$-vector space $H^0(C,\omega_C)$ is generated by a differential $\delta$ with the property that $t_c^*\delta=\delta$ for any $c\in C$, where $t_c\colon C\to C$ is the translation by $C$ \cite[III.5.1]{Silverman_Arithmetic_of_elliptic_curves}. Thus for any $\tau\in \Aut(C)$ we have
				$$\tau^*\delta=\tau^*t_{-\tau(o)}^*\delta$$
	and $t_{-\tau(o)}\circ \tau\in\Aut(C,o)$. Thus we only need to show that $\Aut(C,o)$ acts as a finite group on $H^0(C,\omega_C)$. But $\Aut(C,o)$ is already a finite group \cite[III.10.1]{Silverman_Arithmetic_of_elliptic_curves}.
		\item If $C$ has genus $0$, then $C\cong \bP^1_k$. If $\Supp(E)$ contains at least three points, then $\Aut(\bP^1_k,E)$ is finite. If $E$ is the sum of two distinct reduced points, we may choose coordinates $x,y$ such that $E=[0;1]+[1;0]$. Then $\Aut_k(\bP^1,E)$ sits in an exact sequence
		$$1\to \Aut_k(\bA^1_k,0)\to \Aut_k(\bP^1,E)\to \Bij(E)\to 1$$
	so it suffices to show that $\Aut(\bA^1_k,0)=\{x\mapsto ax\mid a\in k^*\}$ acts finitely on the $1$-dimensional $k$-vector space $H^0(\bP^1_k,\omega_{\bP^1_k}(E))$. This vector space is generated by $dx/x$, which is invariant through $x\mapsto ax$, thus the action of $\Aut_k(\bA^1_k,0)$ is actually trivial.
	\end{enumerate}
\end{proof}

\begin{proposition}\label{proposition:pluricanonical_representations_for_curves_II}
Let $(C,E)$ be a dlt curve over a perfect field $k$. Assume that $K_C+E$ is ample. Then $\Aut_k(C,E)$ is finite.
\end{proposition}
\begin{proof}
As above, we may assume that $k$ is algebraically closed. If $C\cong \bP^1_k$ then $\Supp(E)$ contains at least three points, so the log automorphism group is finite. If $g(C)=1$ then $\Supp(E)$ contains at least one point, so the log automorphism group is finite. Finally if $g(C)\geq 2$ then $\Aut_k(C)$ is already finite.
\end{proof}

\begin{lemma}\label{lemma:Finiteness_automorphisms_general_type_variety}
Let $(S,\Delta)$ be a projective lc surface pair over an algebraically closed field $k$. Assume that $K_S+\Delta$ is big. Then $\Aut_k(S,\sO_S(K_S+\Delta))$ is finite.
\end{lemma}
\begin{proof}
This actually holds in any dimension. The argument follows the proof of \cite[10.1]{Patakfalvi_Zdanowicz_Beauville_Bogomolov_in_pos_char}, with the following modifications: the inequality in (10.1.b) becomes strict (by our bigness assumption); the first sentence after (10.1.c) reads \emph{as $(X,\Delta)$ is lc, all coefficients of $\Gamma$ are smaller or equal to $1$}; and the inequality in the last displayed equation is not strict.
\end{proof}

\begin{proposition}\label{proposition:pluricanonical_representation_for_surfaces}
Let $f\colon (S,\Delta)\to T$ be a proper morphism with $f_*\sO_S=\sO_T$ between normal projective surfaces over a perfect field $k$. Assume that $(S,\Delta)$ is a dlt pair and that $K_S+\Delta\sim_\bQ f^*L$, where $L$ is ample on $T$. Then for $m$ divisible enough,
		$$\im\left[\Bir^c_k(S,\Delta)\to \Aut_k H^0(T, L^{\otimes m})\right]$$
is finite.
\end{proposition}
\begin{proof}
We may extend to the algebraic closure of $k$, and assume it is algebraically closed. Since $L$ is ample, we may assume that $L^{\otimes m}$ is very ample and thus $\Aut_k H^0(T, L^{\otimes m})\subseteq \Aut_k(T)$. 

It follows from the assumptions that $f\colon S\to T$ is a birational morphism. So $K_S+\Delta$ is big and nef. Therefore the canonical model $(S_\text{can},\Delta_\text{can})$ of $(S,\Delta)$ exists \cite{Tanaka_MMP_and_abundance_for_positive_characteristic_log_surfaces}, and it is given by
		$$\psi\colon S\to S_\text{can}:=\Proj\sum_{r\geq 0}H^0(S,rm(K_S+\Delta)),\quad \Delta_\text{can}:=\psi_*\Delta$$
where $m$ is sufficiently divisible. Since $f_*\sO_S=\sO_T$, we have $H^0(S,rm(K_S+\Delta))=H^0(T,L^{\otimes rm})$. As $L$ is ample, we deduce that $T\cong S_\text{can}$ and that we can identify $\psi$ with $f$. Writing $\Delta_T=f_*\Delta$, we have that $K_T+\Delta_T$ is $\bQ$-Cartier and $(T,\Delta_T)$ is lc.

We claim that $f\colon (S,\Delta)\to (T,\Delta_T)$ is crepant. Indeed, we have $K_T+\Delta_T\sim_\bQ f_*f^*L=L$, so $(K_S+\Delta)-f^*(K_T+\Delta_T)$ is $\bQ$-linearly equivalent to $0$ and exceptional over $T$. The negativity lemma then implies that $K_S+\Delta=f^*(K_T+\Delta_T)$.
		
Take $\tau\in \Bir^c(S,\Delta)$. It follows from the definition of a crepant map that $\tau$ induces an automorphism of every group $H^0(S,rm(K_S+\Delta))$ and consequently of their direct sum. Therefore we obtain an induced automorphism $\beta(\tau)\in \Aut(T)$, making the diagram
		$$\begin{tikzcd}
		S\arrow[rr, dotted, "\tau"]\arrow[d, "f"left] && S\arrow[d, "f"] \\
		T\arrow[rr, "\beta(\tau)"] && T
		\end{tikzcd}$$
commutative. It follows $\beta(\tau)^*\sO_T(K_T+\Delta_T)\cong \sO_T(K_T+\Delta_T)$. Thus $\Bir^c_k(S,\Delta)\to\Aut_k(T)$ factorizes through $\Aut_k(T, \sO_T(K_T+\Delta_T))$, which is finite by \autoref{lemma:Finiteness_automorphisms_general_type_variety}.
\end{proof}

\subsubsection{Some remarks on fields of definition}
Let $k$ be a (not necessarily perfect) field, $k^s$ a separable closure and $X$ a $k$-scheme. We say that a $k^s$-sub-scheme $W\subset X_{k^s}$ is defined over $k$ if there exists a sub-$k$-scheme $\sW\subset X$ such that $W=\sW_{k^s}$ (as sub-schemes of $X_{k^s}$).

\begin{lemma}\label{lemma:global_sections_is_a_field}
If $\sY$ is a proper reduced connected $k$-scheme, then $H^0(\sY,\sO_\sY)$ is a finite field extension of $k$.
\end{lemma}
\begin{proof}
The structure morphism $f\colon \sY\to k$ is proper, so $f_*\sO_\sY$ is a coherent $k$-module. Hence $H^0(\sY,\sO_\sY)$ is a finite $k$-algebra. In particular it is Artinian. It is reduced by hypothesis, thus it is a finite direct product of field extensions of $k$. 
Since $\sY$ is connected, there can be only one direct summand. Therefore $H^0(\sY,\sO_\sY)$ is a field.
\end{proof}

\begin{lemma}\label{lemma:minimal_global_sections_implies_geom_connected}
Let $\sY$ be a proper connected $k$-scheme, and assume that $H^0(\sY,\sO_\sY)=k$. Then $\sY$ is geometrically connected.
\end{lemma}
\begin{proof}
It is sufficient to show that $\sY_{k^s}$ is connected. By flat base-change we have $H^0(\sY_{k^s},\sO_{\sY_{k^s}})=H^0(\sY,\sO_\sY)\otimes_k k^s=k^s$. If $\sY_{k^s}=\bigsqcup_{i=1}^d Y^i$ is the decomposition into connected components, then
		$$H^0(\sY_{k^s},\sO_{\sY_{k^s}})=\bigoplus_{i=1}^d H^0(Y^i,\sO_{Y^i})$$
where each $H^0(Y^i,\sO_{Y^i})$ is a finite $k^s$-vector space. Considering the dimensions as $k^s$-vector spaces, we see that $d=1$.
\end{proof}

\begin{lemma}\label{lemma:Galois_closure_is_a_field_of_definition}
Let $X$ be a $k$-scheme and $\sY\subset X$ an connected reduced proper sub-$k$-scheme.
	\begin{enumerate}
		\item The field of definition of any connected component of $\sY_{k^s}$ contains $H^0(\sY,\sO_{\sY})$.
		\item Assume that $H^0(\sY,\sO_\sY)$ is separable over $k$. Then every connected component of $\sY_{k^s}$ is defined over the Galois closure of $H^0(\sY,\sO_\sY)/k$.
	\end{enumerate}
\end{lemma}
\begin{proof}
Let us write $l:=H^0(\sY,\sO_\sY)$. Let $\sY^{(i)}$ be a connected component of $\sY_{k^s}$, and let $k_i$ be its field of definition. We can see $\sY^{(i)}$ as a $k_i$-scheme, and we have a commutative diagram
		$$\begin{tikzcd}
		\sY^{(i)}\arrow[d]\arrow[r] & \sY\arrow[d] \\
		\Spec k_i\arrow[r] & \Spec k.
		\end{tikzcd}$$
The morphism $\sY^{(i)}\to\sY$ is surjective, and induces an inclusion $l\hookrightarrow H^0(\sY,\sO_{\sY^{(i)}})$. By \autoref{lemma:minimal_global_sections_implies_geom_connected} we have $k_i=H^0(\sY^{(i)},\sO_{\sY^{(i)}})$, so the first point follow.

Now assume that $l$ is separable over $k$. Then we can consider $l$ as a subfield of $k^s$. Let $L\subset k^s$ be the Galois closure of $l/k$ and let $\sY_L=\bigsqcup_i\sY_L^i$ be the decomposition into connected components. By separability $H^0(\sY_L,\sO_{\sY_L})=l\otimes_k L$ is a reduced Artinian $L$-algebra, thus a direct product of finitely many field extensions of $L$. By considering the Stein factorization of $\sY_L\to L$, we see that in fact these field extensions are the $H^0(\sY_L^i,\sO_{\sY_L^i})$. On the other hand, we have an inclusion $l\otimes_kL\hookrightarrow L\otimes_kL$, and since $L$ is Galois over $k$ we have an $L$-algebra isomorphism $L\otimes_k L\cong \bigoplus_{\Gal(L/k)}L$. This induces inclusions $H^0(\sY_L^i,\sO_{\sY_L^i})\subseteq L$. Hence $H^0(\sY_L^i,\sO_{\sY_L^i})=L$ for every $i$. By \autoref{lemma:minimal_global_sections_implies_geom_connected} the proof is complete.
\end{proof}

\begin{corollary}
Let $\sY\subset X$ be a connected reduced proper $k$-sub-scheme. Assume that $H^0(\sY,\sO_\sY)$ is separable over $k$. Then the number of connected components of $\sY_{k^s}$ is equal to $\dim_k H^0(\sY,\sO_\sY)$.
\end{corollary}
\begin{proof}
By \autoref{lemma:Galois_closure_is_a_field_of_definition} and its proof, there is a finite Galois extension $L/k$ with the following property: if $\sY_L=\bigsqcup_{i=1}^d Y^i$ is the decomposition into connected components, then $H^0(\sY_L,\sO_{\sY_L})=\bigoplus_{i=1}^dH^0(Y^i,\sO_{Y^i})$ and $H^0(Y^i,\sO_{Y^i})=L$ for each $i$. It follows from \autoref{lemma:minimal_global_sections_implies_geom_connected} that the number of connected components of $\sY_{k^s}$ is equal to $d$. On the other hand
		$$d=\dim_LH^0(\sY_L,\sO_{\sY_L})=\dim_LH^0(\sY,\sO_\sY)\otimes_k L=\dim_k H^0(\sY,\sO_\sY)$$
so the result follows.
\end{proof}

\begin{lemma}\label{lemma:Galois_descent_of_schemes}
Let $X$ be quasi-projective $k$-scheme, $K/k$ a Galois extension and $W\subset X_K$ a closed reduced sub-$k$-scheme. Assume that $W$ is stable under the action of a non-trivial subgroup $H$ of $G:=\Gal(K/k)\circlearrowright X_K$. Then $W$ is defined over the sub-field $K^H$.
\end{lemma}
\begin{proof}
By Artin's lemma, the extension $K^H\subset K$ is Galois with Galois group $H$. Replacing $k$ by $K^H$, we may assume that $W$ is stable under the action of $G$. The canonical morphism $\pi\colon X_K\to X$ is the quotient by $G$. Let $Y:=\pi(W)$ be the reduced closed image of $W$. Since $W$ is $G$-invariant and the fibers of $\pi$ are the $G$-orbits, we see that $\Supp(W)=\Supp(Y_K)$. Since $W$ and $Y_K$ are reduced, we deduce that $W=Y_K$.
\end{proof}

\subsection{Sources of lc centers}
In this section we develop the theory of sources for lc centers of lc threefold pairs over a perfect field of characteristic $>5$. Our approach follows closely Koll\'{a}r's original one \cite[\S 4]{Kollar_Singularities_of_the_minimal_model_program}.

\begin{notation}\label{notation:crepant_dlt_blow_up_of_threefold}
Let $(X,\Delta)$ be a quasi-projective lc threefold pair defined over a perfect field $k$ of characteristic $>5$, and $(Y,\Delta_Y)\to (X,\Delta)$ be a crepant dlt $\bQ$-factorial blow-up (which exists by \autoref{fact:crepant_dlt_blow_up_ind_dim_2_3}).
\end{notation}
\noindent Our program is the following:
	\begin{enumerate}
		\item In \autoref{section: Poincaré residues}, we observe that the lc centers of dlt $\bQ$-factorial pairs $(Y,\Delta_Y)$ are the strata of the reduced boundary, analogously to the characteristic $0$ case. This allows us to define higher codimension adjunction for dlt pairs.
		\item In \autoref{section: geometry of lc centers}, we compare the fibrations between lc centers obtained from $(Y,\Delta_Y)\to (X,\Delta)$.
		\item In \autoref{section: springs and sources}, we define springs and sources for lc centers on the reduced boundary of threefold pairs.
	\end{enumerate}

\subsubsection{Higher Poincaré residues}\label{section: Poincaré residues}
In characteristic $0$, the use of crepant blow-ups is motivated by the very simple structure of lc centers on a dlt pair \cite[4.16]{Kollar_Singularities_of_the_minimal_model_program}. Using the recent results of \cite{Arvidsson_Bernasconi_Lacini_KVV_for_log_dP_surfaces_in_pos_char}, we can extend this result to positive characteristic. The same result was obtained in \cite[2.2]{Das_Hacon_Adjunction_for_3_folds_in_char_>5} with other methods.

\begin{proposition}\label{proposition: lc centers on dlt 3-folds}
Let $(Y,\Delta)$ be a $\bQ$-factorial dlt pair of dimension $\leq 3$ over a perfect field of characteristic $p>5$. Write $\Delta=\Delta^{<1}+\sum_i D_i$ with each $D_i$ prime. Then
	\begin{enumerate}
		\item The lc centers of $(Y,\Delta)$ are exactly the irreducible components of the intersections of the $D_i$'s.
		\item Every irreducible component of such an intersection is normal of the expected codimension.
		\item Let $Z\subset Y$ be an lc center of $(Y,\Delta)$. If $D_i$ is $\bQ$-Cartier and does not contain $Z$, then every irreducible component of $D_i|_Z$ is $\bQ$-Cartier.
		\item Each $(D_i,\Diff_{D_i}(\Delta-D_i))$ is dlt.
	\end{enumerate}
\end{proposition}
\begin{proof}
Since $Y$ is $\bQ$-factorial, the pair $(Y,(1-\epsilon)\Delta^{=1}+\Delta^{<1})$ is klt. Hence by \cite[Corollary 1.3]{Arvidsson_Bernasconi_Lacini_KVV_for_log_dP_surfaces_in_pos_char}, if $D$ is any Weil divisor on $Y$ then $\sO_Y(D)$ is CM. Thus the proof of \cite[4.16]{Kollar_Singularities_of_the_minimal_model_program} applies verbatim.
\end{proof}

This implies the existence of higher-dimensional Poincaré residue maps as in \cite[4.18-19]{Kollar_Singularities_of_the_minimal_model_program}.

\begin{definition-proposition}\label{cor:Higher_Poincaré_residues}
Let $(Y,\Delta)$ be as above and $Z$ an lc center of $(Y,\Delta)$. Then \emph{there exists a canonically-defined $\bQ$-Cartier $\bQ$-Weil divisor $\Diff^*_Z\Delta$ on $Z$} such that:
	\begin{enumerate}
		\item $(Z,\Diff^*_Z\Delta)$ is dlt;
		\item If $m$ is even and $m(K_Y+\Delta)$ is Cartier, there is a canonical isomorphism
				$$\sR^m_{Y,Z}\colon\omega_Y^{[m]}(m\Delta)|_Z\cong \omega_Z^{[m]}(m\Diff^*_Z\Delta).$$
		\item If $W\subset Z$ is a lc center of $(Y,\Delta)$, then $W$ is also a lc center of $(Z,\Diff^*_Z\Delta)$ and
				$$\Diff^*_W\Delta=\Diff^*_W(\Diff^*_Z\Delta).$$
	\end{enumerate}
\end{definition-proposition}

\begin{remark}
If $Z\subset \Delta^{=1}$ is a prime divisor then we have $\Diff^*_Z\Delta=\Diff_Z(\Delta-Z)$, which is arguably a conflict of notations. When working on dlt pairs we will only use the $\Diff^*$-notation, hence no confusion should arise.
\end{remark}

\begin{corollary}\label{corollary:Crepant_map_induces_crepant_map_on_lc_centers}
Let $(Y_1,\Delta_1), (Y_2,\Delta_2)$ be surface pairs over an arbitrary field, or threefold pairs over a perfect field of characteristic $>5$. Assume that $(Y_1,\Delta_1)$ is $\bQ$-factorial dlt.

Let $\phi\colon (Y_1,\Delta_1)\dashrightarrow (Y_2,\Delta_2)$ be a crepant birational map. Assume that $S_1\subset Y_1$ is an lc center of $(Y_1,\Delta_1)$ and that $\phi$ is a local log isomorphism at the generic point of $S_1$. Then $S_2:=\phi_*S_1$ is an lc center of $(Y_2,\Delta_2)$, and $\phi$ restricts to a crepant birational map
		$\phi|_{S_1}\colon (S_1,\Diff^*_{S_1}\Delta_1)\dashrightarrow (S_2^n,\Diff^*_{S^n_2}\Delta_2)$.
\end{corollary}
\begin{proof}
By induction on the strata of $\Delta_1^{=1}$, we may assume that $S_1$ is a divisor. Since $\phi$ is a log isomorphism at the generic point of $S_1$, we obtain that $S_2$ is a component of $\Delta_2^{=1}$ and a lc center of $(Y_2,\Delta_2)$. 

Consider a resolution
		$$\begin{tikzcd}
		& W\arrow[dl, "u" above left]\arrow[dr, "v"] &\\
		Y_1\arrow[rr, dotted, "\phi"] & &Y_2
		\end{tikzcd}$$
Let $S_W\subset W$ be the strict transform of $S_1$, with normalization $n\colon S^n_W\to S_W$. Since $\phi$ is a local isomorphism at the generic point of $S_1$, $v$ maps $S_W^n$ to $S_2^n$ and we have a commutative diagram of birational maps
		$$\begin{tikzcd}
		& S^n_W\arrow[dr, "v|_{S_W^n}" ]\arrow[dl, "u|_{S_W^n}" above left] & \\
		S_1 \arrow[rr, dotted, "\phi|_{S_1}"] && S_2^n
		\end{tikzcd}$$
Since $\phi$ is crepant and a local log isomorphism at $S_1$, we can write
		$$u^*(K_{Y_1}+\Delta_1)=K_W+S_W+\Gamma=v^*(K_{Y_2}+\Delta_2)$$
where $\coeff_{S_W}\Gamma=0$. Thus by adjunction
		\begin{eqnarray*}
			(u|_{S^n_W})^*(K_{S_1}+\Diff^*_{S_1}\Delta_1) & = & n^*u^*(K_{Y_1}+\Delta_1) \\
			& = & n^*(K_W+S_W+\Gamma) \\
			& = & K_{S_W^n}+\Diff_{S_W^n}\Gamma 
		\end{eqnarray*}
and similarly $(v|_{S^n_W})^*(K_{S_2^n}+\Diff^*_{S_2^n}\Delta_2)=K_{S_W^n}+\Diff_{S_W^n}\Gamma$. This implies that the birational map $\phi|_{S_1}\colon (S_1,\Diff^*_{S_1}\Delta_1)\dashrightarrow (S_2^n,\Diff^*_{S_2^n}\Delta_2)$ is crepant.
\end{proof}

\subsubsection{Geometry of lc lcenters}\label{section: geometry of lc centers}

Let $(Y,\Delta_Y)\to (X,\Delta)$ be as in \autoref{notation:crepant_dlt_blow_up_of_threefold}, $Z\subset X$ an lc center of $(X,\Delta)$ and $S,S'\subset Y$ lc centers of $(Y,\Delta_Y)$ that are minimal for the property of dominating $Z$. In characteristic $0$, a crucial observation due to Koll\'{a}r \cite{Kollar_Sources_of_lc_centers} is that $S$ and $S'$ are birational, and that the birational map arises from a special structure which he calls a \emph{standard $\bP^1$-link}. In this section, we prove some analogous statements in positive characteristic.

We begin by defining the $\bP^1$-links.

\begin{definition}[Weak standard $\bP^1$-link]\label{definition:weak-P1_link}
A \textbf{weak standard $\bP^1$-link} is a $\bQ$-factorial pair $(T, W_1+W_2+E)$ together with a proper morphism $\pi\colon T\to W$ such that
	\begin{enumerate}
		\item $\pi_*\sO_T=\sO_W$,
		\item $K_T+W_1+W_2+E\sim_{\bQ,\pi} 0$,
		\item $W_1$ and $W_2$ are disjoint and normal,
		\item both $\pi\colon W_i\to W$ are isomorphisms, and
		\item $\red(T_w)\cong \bP^1_{k(w)}$ for every $w\in W$.
	\end{enumerate}
\end{definition}

\begin{remark}
In Koll\'{a}r's definition of standard $\bP^1$-link \cite[4.36]{Kollar_Singularities_of_the_minimal_model_program}, it is assumed that $(T,W_1+W_2+E)$ is plt. Let us justify our deviation. The use of (weak) $\bP^1$-links is motivated by \autoref{prop:P1_link_from_crepant_structure} below, which is analogous to \cite[4.37]{Kollar_Singularities_of_the_minimal_model_program}. In both cases, during the proof we find a weak $\bP^1$-link $(T,W_1+W_2+E)\to W$ such that $(T,E)$ is klt. Using cyclic covers, we can see that $(T,W_1+W_2+E)$ is locally the quotient of a product $(\tilde{W}\times\bP^1, \tilde{W}\times [0;1]+\tilde{W}\times[1;0]+E_{\tilde{W}}\times\bP^1)$. By inversion of adjunction, this product is plt. In characteristic $0$ this implies that the quotient is plt, but this is not necessarily the case in positive characteristic.

However, our weaker assumption on standard links is sufficient for our purpose, as we shall see.
\end{remark}

\begin{lemma}\label{lemma:std_P1_link_gives_crepant_isomorphism}
Let $\pi\colon (T,W_1+W_2+E)\to W$ be a weak standard $\bP^1$-link over a field of characteristic $>2$, where $T$ is a surface. Then $\Supp(E)$ is a union of fibers and there is a log isomorphism $(W_1,\Diff_{W_1}E)\cong (W_2,\Diff_{W_2}E)$ that commutes with the projections onto $W$.
\end{lemma}
\begin{proof}
Let $\varphi\colon T'\to T$ be any birational proper morphism, and write $\varphi^*(K_T+W_1+W_2+E)=K_{T'}+W_1'+W_2'+\sE$ where $W_i'$ is the strict transform of $W_i$. Then $\sE$ might not be effective, but since $W_i'\cong W_i$ by normality and dimension reasons, we have $(W_i',\Diff_{W_i'}\sE)\cong (W_i,\Diff_{W_i}E)$. Thus we are free to replace $T$ by any birational model. In particular, we can successively blow-up the singular locus and assume that $T$ is regular. For dimensional reasons, the singular locus of $T$ does not dominate $W$, thus the property on the fibers becomes: $\red(T_w)\cong\bP^1_{k(w)}$ for a general $w\in W$. 

Now the generic fiber $F=T_w$ is geometrically normal by \autoref{lemma:Fiber_of_MFS_are_normal}, and by assumption $K_F$ is anti-ample. In particular $2=-K_T\cdot_{k(w)} F=(W_1+W_2+E)\cdot_{k(w)} F$. Since $W_i\cdot_{k(w)} F=1$, we deduce that the support of $E$ is a union of fibers, say $E=\sum_j a_j\red(T_{w_j})$.

Since the $\pi_i\colon W_i\to W$ are isomorphisms, we obtain an isomorphism $\phi=\pi_2^{-1}\circ\pi_1\colon W_1\cong W_2$. By \cite[4.5.(4)]{Kollar_Singularities_of_the_minimal_model_program}, since $T$ is regular, the different divisor $\Diff_{W_i}E$ is just the set of points of $\sum_j \pi_i^{-1}w_j$. Thus $\phi_*\Diff_{W_1}E=\Diff_{W_2}E$, and by \autoref{lemma:birational_crepant_pairs} we obtain that $\phi\colon (W_1,\Diff_{W_1}E)\to (W_2,\Diff_{W_2}E)$ is crepant.
\end{proof}

\begin{proposition}\label{prop:P1_link_from_crepant_structure}
Let $(S, \Delta)$ be a quasi-projective klt surface pair defined over a field of characteristic $>2$, with a projective morphism $f\colon S\to Z$. Let $D$ be an effective $\bZ$-divisor that dominates $Z$, such that $K_S+\Delta+D\sim_{\bQ,f}0$ and $(S,D+\Delta)$ is dlt. Pick a point $z\in Z$ such that $f^{-1}(z)$ is connected but $f^{-1}(z)\cap D$ is disconnected. 

Then there exists an elementary \'{e}tale morphism $(z'\in Z')\to (z\in Z)$ and a proper morphism $W\to Z'$ such that $(S,\Delta+D)\times_Z Z'=(S', \Delta'+D')$ is crepant birational to a weak standard $\bP^1$-link over $W$.
\end{proposition}
\begin{proof}
Given an elementary \'{e}tale neighborhood $(z'\in Z')\to (z\in Z)$, we form the Cartesian diagram
		$$\begin{tikzcd}
		(S',D'+\Delta')\arrow[r]\arrow[d, "f'"] & (S,D+\Delta)\arrow[d, "f"] \\
		Z'\arrow[r] & Z
		\end{tikzcd}$$
We can choose $Z'\to Z$ such that the different connected components of $(f')^{-1}(z')\cap D'$ are contained in different connected components of $D'$. Notice that $(S',\Delta')$ is klt, $(S',D'+\Delta')$ is dlt, $K_{S'}+D'+\Delta'\sim_{\bQ,f'} 0$ and $(f')^{-1}(z')$ is connected.

By \cite[Theorem 1.1]{Tanaka_MMP_for_excellent_surfaces}, we may run an MMP for $(S',\Delta')$ over $Z'$. Since $K_{S'}+\Delta'\sim_{\bQ,f'}-D'$ is not pseudo-effective on the generic fiber of $f'$, the MMP terminates with a Mori fiber space $p\colon (S^*,\Delta^*)\to W$ over $Z'$. We picture our construction as follows:
		$$\begin{tikzcd}
		(S',D'+\Delta') \arrow[rr,"\varphi", dotted]\arrow[d, "f'"] && (S^*,D^*+\Delta^*)\arrow[d, "p"]\arrow[dll, "g"] \\		
		Z' && W\arrow[ll]
		\end{tikzcd}$$
where $D^*=\varphi_*D$ and $\Delta^*=\varphi_*\Delta$. (Since $S'$ is a surface, $\varphi$ is actually a morphism.)

By construction we have $K_{S^*}+D^*+\Delta^*\sim_{\bQ,g}0$, so in particular $K_{S^*}+D^*+\Delta^*\sim_{\bQ,p}0$. Moreover $\varphi\colon (S',D'+\Delta')\dashrightarrow (S^*,D^*+\Delta^*)$ is crepant, so it follows from \autoref{lemma:crepant_pairs_and_non_klt_locus} that $g^{-1}(z')\cap D^*$ is still disconnected.

Since $p$ is a Mori fiber space, there is a component $D_1^*\subset D^*$ that has positive intersection with the contracted ray inducing $p\colon S^*\to W$. Thus $D_1^*$ is $p$-ample. Now there is another component $D_2^*\subset D^*$ that is disjoint from $D^*_1$. Take a curve $C$ which intersects $D_2^*$ and is contained in a fiber of $p$. The intersection $C\cdot D_2^*$ cannot be negative, otherwise the fiber would be included in $D_2^*$, and $D_1^*$ and $D_2^*$ would not be disjoint. So $D_2^*$ also has positive intersection with the contracted ray, and is also $p$-ample. 

So $D_1^*$ and $D_2^*$ are disjoint and have positive intersection with every curve contracted by $p$. Assume that a fiber $F$ of $p$ has dimension at least $2$: then $F\cap D_1^*$ must be at least $1$-dimensional, hence it intersects $D^*_2$, contradiction. Thus the fibers of $p$ have dimension $1$. Since $(S^*,\Delta^*)$ is a klt surface, it is $\bQ$-factorial and $K_{S^*}$ is $\bQ$-Cartier. As $K_{S^*}+\Delta^*$ is $p$-anti-ample, it follows that $K_{S^*}$ is $p$-anti-ample. By \autoref{lemma:Fiber_of_MFS_are_normal}, a general fiber $F$ of $p$ is smooth rational. As
		$$F\cdot (\Delta^*+D^*)=-K_{S^*}\cdot F=2,$$
we deduce via \autoref{lemma:sections_of_fibered_curves} that $D_1^*$ and $D_2^*$ are sections of $p$ and that the other components of $\Delta^*+D^*$ are vertical over $W$. If there was another component of $D^*$ whose image under $g$ contained $z'$, then that component would also contain $g^{-1}(z)$, and thus $g^{-1}(z)\cap D^*$ would not be disconnected. Thus we can shrink $Z'$ and assume that $D^*=D^*_1+D_2^*$. 

By \cite[Theorem 1.3]{Tanaka_MMP_for_excellent_surfaces}, we have $R^1p_*\sO_{S^*}=0$, and thus every reduced fiber of $p$ is a tree of $\bP^1$'s. If a fiber is reducible, then by $p$-ampleness both $D_1^*$ and $D_2^*$ must intersect every irreducible component. But both $D_i^*$ are sections of $p$, thus the reducible fiber has two components with a single intersection point, and both $D_i^*$'s pass through that point, which contradicts $D_1^*\cap D_2^*=\emptyset$. Thus every fiber is irreducible.

Thus $p\colon (S^*,D_1^*+D_2^*+\Delta^*)\to W$ is a weak standard $\bP^1$-link. This proves the proposition.

\end{proof}

\begin{remark}
In the case of threefolds in characteristic $>5$, a generalization of \autoref{prop:P1_link_from_crepant_structure} is proved in \cite[Theorem 1.2]{Filipazzi_Waldron_Connectedness_principle_for_3folds_pos_char}.
\end{remark}

\begin{definition}[Weak $\bP^1$-links]
Let $f\colon (Y,\Delta_Y)\to (X,\Delta)$ be as in \autoref{notation:crepant_dlt_blow_up_of_threefold}, and let $Z_1,Z_2\subset Y$ be two lc centers. We say that $Z_1$ and $Z_2$ are \textbf{directly weakly $\bP^1$-linked} if there exists an lc center $W\subseteq Y$ (we allow $W=Y$) such that $f(Z_1)=f(Z_2)=f(W)$ and $(W,\Diff^*_W\Delta_Y)$ is crepant birational to a $\bP^1$-link with the $Z_i$ as the two sections, and the base of the $\bP^1$-link factorizes the morphism to $X$.

We say that $Z_1$ and $Z_2$ are \textbf{weakly $\bP^1$-linked} if there is a sequence of lc centers $Z_1',\dots,Z_m'$ with $Z_1=Z_1'$ and $Z_2=Z_m'$, such that $Z_i'$ and $Z_{i+1}'$ are directly weakly $\bP^1$-linked.
\end{definition}

\begin{lemma}\label{lemma:P1_link_preserves_minimality}
Let $f\colon (Y,\Delta_Y)\to (X,\Delta)$ be as in \autoref{notation:crepant_dlt_blow_up_of_threefold}, and let $S_1,S_2\subset Y$ be two weakly $\bP^1$-linked lc centers. Then:
	\begin{enumerate}
		\item $(S_1,\Diff^*_{S_1}\Delta_Y)$ and $(S_2,\Diff^*_{S_2}\Delta_Y)$ are crepant birational over $X$, and
		\item $(S_1,\Diff^*_{S_1}\Delta_Y)$ is klt if and only if $(S_2,\Diff^*_{S_2}\Delta_Y)$ is klt.
	\end{enumerate}
\end{lemma}
\begin{proof}
We may assume that $S_1$ and $S_2$ are directly weakly $\bP^1$-linked. Then there is a lc center $T\subset Y$ containing both $S_i$'s, a weak standard $\bP^1$-link $(T', W_1+W_2+E)\to W$ and a crepant birational map $\phi\colon (T,\Diff^*_T\Delta_Y)\dashrightarrow (T',W_1+W_2+E)$ such that the map $T\dashrightarrow T'\to W$ factors the morphism $T\to X$, and such that $\phi$ maps birationally $S_i$ to $W_i$.

By \autoref{corollary:Crepant_map_induces_crepant_map_on_lc_centers}, $\phi$ induces crepant birational $X$-maps $(S_i,\Diff^*_{S_i}\Delta_Y)\dashrightarrow (W_i,\Diff_{W_i}E)$. Composing these maps with the log isomorphism $(W_1,\Diff_{W_1}E)\cong (W_2,\Diff_{W_2}E)$ over $W$ given by \autoref{lemma:std_P1_link_gives_crepant_isomorphism}, we obtain that the $(S_i,\Diff_{S_i}^*\Delta_Y)$ are crepant birational over $X$.

The second assertion follows immediately from the first one.
\end{proof}

\begin{proposition}\label{prop:minimal_centers_are_linked}
Let $f\colon (Y,\Delta_Y)\to (X,\Delta)$ be as in \autoref{notation:crepant_dlt_blow_up_of_threefold}. Fix $T\subset X$ an lc center of $(X,\Delta)$. Then the lc centers of $(Y,\Delta_Y)$ that are minimal for dominating $T$, are weakly $\bP^1$-linked to each other.
\end{proposition}
\begin{proof}
We prove the following more precise result:

\begin{claim}\label{claim:linked_lc_centers_above_a_fixed_point}
Given a point $x\in X$, an lc center $Z\subset Y$ of $(Y,\Delta_Y)$ that is minimal for the property $x\in f(Z)$, an lc center $W\subset Y$ such that $x\in f(W)$, then there exists an lc center $Z_W\subset W$ such that $Z$ and $Z_W$ are weakly $\bP^1$-linked.
\end{claim}

By \autoref{lemma:P1_link_preserves_minimality} this claim implies the result, taking $x$ to be the generic point of $T$.

First we prove our claim holds after passing to an elementary \'{e}tale neighborhood of $(x\in X)$. Then we will show that the weak $\bP^1$-link descends along the \'{e}tale base-change.

By \cite[1.2]{Nakamura_Tanaka_Witt_Nadel_vanishing_for_threefolds}, the fibers of 
		$$\Supp\Delta_Y^{=1}=\Nklt(Y,\Delta_Y)\longrightarrow X$$
are geometrically connected. After passing to an elementary \'{e}tale neighborhood of $(x\in X)$, we may assume that $\Delta_Y^{=1}=\sum_{i=1}^r \Delta_i$, that each $\Delta_i$ has a connected fiber above $x$ and that every lc center of $(Y,\Delta_Y)$ intersects $f^{-1}(x)$. Relabelling the $\Delta_i$ is necessary, we assume that $Z\subset \Delta_1,W\subset \Delta_r$ and $\Delta_i\cap \Delta_{i+1}\cap f^{-1}(x)\neq \emptyset$ for all $i=1,\dots,r-1$. 

Now we prove that $Z\subset \Delta_1$ is $\bP^1$-linked (after the \'{e}tale base-change) to an lc center contained in $\Delta_1\cap \Delta_2$. For ease of notation, we let $S:=\Delta_1$, $D:=\Diff^*_{S}\Delta_Y$, $E:=\Delta_1\cap \Delta_2$. Then $(S,D)$ is a dlt surface pair, admits a projective morphism $g=f|_S\colon S\to X$ with connected fibers such that $K_S+D\sim_{\bQ,g}0$. Notice that $ E\subset \Supp D^{=1}$ and that $Z\subset \Supp D^{=1}\cap g^{-1}(x)$. We distinguish a few cases according to the dimension of $g(S)$.
	\begin{enumerate}
		\item Suppose that $\dim g(S)=2$. Then $R^1g_*\sO_S(-D^{=1})=0$ by \autoref{lemma: dlt surface over a surface}, so the natural map $g_*\sO_S\to g_*\sO_{D^{=1}}$ is surjective. If $D^{=1}\cap g^{-1}(x)$ was disconnected, then we would get a section of $g_*\sO_S$ that vanishes on one connected component of $g^{-1}(x)\cap D^{=1}$, and is identically $1$ on another. But $g^{-1}(x)$ is connected and sections of $\sO_{g^{-1}(x)_\text{red}}$ are supported on entire components, so we get a contradiction. Thus $D^{=1}\cap g^{-1}(x)$ is connected. (If we know that $(x\in g(S))$ was a normal surface singularity, then we can simply apply the connectedness principle for surfaces \cite[2.36]{Kollar_Singularities_of_the_minimal_model_program}.)
		
		The scheme $D^{=1}\cap g^{-1}(x)$ is either $0$-dimensional or $1$-dimensional. If it is $0$-dimensional then $D^{=1}\cap g^{-1}(x)=Z$ and we are done. 
		
		Assume that $D^{=1}\cap g^{-1}(x)$ is $1$-dimensional. Then $Z$ is an intersection point of two components of $D^{=1}$ above $x$. If $Z$ is the only such point, then we are done. If not, then some component of $D^{=1}$ are contained in $g^{-1}(x)$ and by adjunction we easily see that each one of them must contain two such intersection points. In particular these points are $k(x)$-points, and the components of $D^{=1}$ contained in $g^{-1}(x)$ are isomorphic to $\bP^1_{k(x)}$. Thus the components of $D^{=1}$ contained in $g^{-1}(x)$ give the weak $\bP^1$-link between these intersection points.
		
		\item Suppose that $\dim g(S)=1$. First we consider the case where $x$ is the generic point of $C:=g(S)$. Then $D^{=1}\cap g^{-1}(x)$ is the set of generic points of those components of $D^{=1}$ that dominate $C$. If there is only one such component, then it is $Z$ and we are done. If at least two components of $D^{=1}$ dominate $C$, then we may use \autoref{prop:P1_link_from_crepant_structure} to obtain that, after possibly a further elementary \'{e}tale base-change, exactly two components of $D^{=1}$ dominate $C$ and they are $\bP^1$-linked to each other.
		
		Now assume that $x\in C$ is a closed point. If $D^{=1}\cap g^{-1}(x)$ is connected, we can apply the same analysis as in the previous point.
		 If $D^{=1}\cap g^{-1}(x)$ is disconnected and $D^{=1}$ dominates $C$, then we can use \autoref{prop:P1_link_from_crepant_structure}.
		
		
		Finally, if $D^{=1}$ does not dominate $C$, we show that $D^{=1}\cap g^{-1}(x)$ is connected, hence all cases have been covered. Indeed, if $D^{=1}$ does not dominate $C$ then it is supported on fibers, and $K_S+D^{<1}\sim_{\bQ,g}-D^{=1}$ is pseudo-effective on the generic one. By \cite{Tanaka_MMP_for_excellent_surfaces} we can run a $(K_S+D^{<1})$-MMP over $C$, which terminates with a birational model $S'$ of $S$, say
			$$\begin{tikzcd}
			(S,D^{<1},D^{=1})\arrow[dr, "g" below left]\arrow[rr, "\varphi"] && (S',\Delta':=\varphi_*D^{<1},\Gamma':=\varphi_*D^{=1})\arrow[dl, "g'"] \\
			& C &
			\end{tikzcd}$$
		such that $K_{S'}+\Delta'\sim_{\bQ,g'}-\Gamma'$ is $g'$-nef. Notice that:
		\begin{itemize}
			\item By \autoref{lemma:birational_crepant_pairs} the morphism $\varphi\colon (S,D)\dashrightarrow (S',\Delta'+\Gamma')$ is crepant;
			\item Since $K_{S'}+\Delta'+\Gamma'\sim_{\bQ,g'}0$ and $\Gamma'$ is vertical over $C$, we see that $(K_{S'}+\Delta')\cdot (g')^*M=0$ for any $\bQ$-Cartier divisor $M$ on $C$;
			\item $\Gamma'\neq 0$. For if $\varphi$ would contract $D^{=1}$, then looking at in intermediate step of the MMP we may assume that $D^{=1}$ is irreducible. But then $(K_S+D^{<1})\cdot D^{=1}=-(D^{=1})^2$ is non-negative by \autoref{lemma:Square_zero_vertical_is_fiber}, which is a contradiction.
		\end{itemize}			
		   Now we claim that $\Gamma'\cap (g')^{-1}(x)$ is connected for any $x\in C'$. By \autoref{lemma:Square_zero_vertical_is_fiber} we have $(\Gamma')^2\leq 0$, with equality if and only if $\Gamma'$ is a weighted sum of fibers. In any case we can write 
				$$-\Gamma'\sim_{\bQ} (g')^*N+E$$
		where $N$ is $\bQ$-Cartier on $C$ and $E$ is an effective $\bQ$-divisor supported on the fibers. Hence we have
				$$0\geq (\Gamma')^2=(K_{S'}+\Delta')\cdot (-\Gamma')=(K_{S'}+\Delta')\cdot ((g')^*N+E)\geq 0,$$
		thus $(\Gamma')^2=0$. Since $g'\colon S'\to C$ has connected fibers, we deduce that every $\Gamma'\cap (g')^{-1}(x)$ is connected. 
		
		Since $(S,D^{<1})$ is klt and $\varphi$ is a $(K_S+D^{<1})$-MMP, the pair $(S',\Delta')$ is also klt. In particular $\Nklt(S',\Delta'+\Gamma')=\Gamma'$. On the other hand, since $(S,D)$ is dlt we have $\Nklt(S,D)=D^{=1}$. Therefore, since $\varphi\colon (S,D)\to (S',\Delta'+\Gamma')$ is a crepant $C$-morphism, by \autoref{lemma:crepant_pairs_and_non_klt_locus} we have a bijection between the connected components of $D^{=1}\cap g^{-1}(x)$ and the connected components of $\Gamma'\cap (g')^{-1}(x)$. In particular $D^{=1}\cap g^{-1}(x)$ is connected.
		
		\item Finally, suppose that $\dim g(S)=0$. If $D^{=1}\cap g^{-1}(x)$ is connected, we apply the same analysis as in the first point. If $D^{=1}\cap g^{-1}(x)$ is disconnected, we can use \autoref{prop:P1_link_from_crepant_structure}.
	\end{enumerate}
	
It remains to show that the weak $\bP^1$-link descends the \'{e}tale base-change. More precisely, let $Z_i\subset Y$ be the lc centers that are minimal for the property $x\in f(Z_i)$. We have proved that after an elementary \'{e}tale base-change $(x'\in X')\to (x\in X)$, the $Z_i':=Z_i\times_XX'$ are weakly $\bP^1$-linked over $X'$. In particular they have the same image in $X'$, and thus the $Z_i$ have the same image in $X$, say $V$. Let $v\in X$ be the generic point of $V$. Notice that the lc centers of $(Y,\Delta_Y)$ that are minimal above $v$ are exactly the $Z_i$'s. Thus we can refine the earlier statement: there is a elementary \'{e}tale neighborhood $(\tilde{v}\in \tilde{X})\to (v\in X)$ such that the $\tilde{Z}_i:=Z_i\times_X\tilde{X}$ are $\bP^1$-linked. Since $k(\tilde{v})=k(v)$, the morphisms $\tilde{Z}_i\to Z_i$ are generic isomorphisms. Thus the weak $\bP^1$-links between the $\tilde{Z}_i$'s descend to weak $\bP^1$-links between the $Z_i$'s.
\end{proof}

\begin{remark}\label{remark:minimal_centers_are_linked_in_smaller_dimensions}
The same argument shows that \autoref{prop:minimal_centers_are_linked} also holds for an arbitrary quasi-projective lc surface pair $(S,\Delta)$ over a field and a crepant dlt blow-up $(S',\Delta')\to (S,\Delta)$. This can also be shown by considering the minimal resolution of $(S,\Delta)$.
\end{remark}

\begin{corollary}\label{cor:adjunction_for_lc_centers}
Let $f\colon (Y,\Delta_Y)\to (X,\Delta)$ be as in \autoref{notation:crepant_dlt_blow_up_of_threefold}, and $S\subset Y$ a lc center of $(Y,\Delta_Y)$. Denote by 
		$$S\overset{f_S}{\longrightarrow} X_S\overset{\pi}{\longrightarrow} X$$
the Stein factorization of $f|_S$. If $Z\subset X$ is an lc center of $(X,\Delta)$ contained in $f(S)$, then every irreducible component of $\pi^{-1}Z\subset X_S$ is the image of an lc center of $(S,\Delta_S:=\Diff^*_S\Delta_Y)$.
\end{corollary}
\begin{proof}
Choose a minimal lc center $V\subset Y$ above $Z$. By \autoref{claim:linked_lc_centers_above_a_fixed_point}, there is a minimal lc center $V'\subseteq Y$ that dominates $Z$ and that is contained in $S$. By adjunction for dlt pairs, $(V',\Diff^*_{V'}\Delta_Y)$ is an lc center of $(S,\Delta_S)$. Then $f_S(V')\subset X_S$ is one of the irreducible components of $\pi^{-1}Z$.

Now let $\eta\in X$ be the generic point of $Z$. After passing to an elementary \'{e}tale neighborhood of $(\eta\in X)$, we can assume that $S$ is the union of irreducible components $S_j$ where each $f^{-1}(\eta)\cap S_j$ is connected. The previous argument show that the irreducible components of $\pi^{-1}Z$ are images of lc centers of the $S_j$, and these lc centers descend to $S$ by \cite[2.15]{Kollar_Singularities_of_the_minimal_model_program}.
\end{proof}

\begin{corollary}\label{corollary:intersection_of_lc_centers}
Let $(X,\Delta)$ be a quasi-projective slc threefold pair over a perfect field of characteristic $>5$. Then any intersection of lc centers is a union of lc centers.
\end{corollary}
\begin{proof}
The normalization $(\bar{X},\bar{D}+\bar{\Delta})\to (X,\Delta)$ is crepant and $(\bar{X},\bar{D}+\bar{\Delta})$ is lc. Thus the lc centers of $(X,\Delta)$ are the images of the lc centers of $(\bar{X},\bar{D}+\bar{\Delta})$. Hence we may assume that $(X,\Delta)$ is lc. Let $Z,Z'\subset X$ be lc centers, and pick a point $x\in Z\cap Z'$. If $f\colon (Y,\Delta_Y)\to (X,\Delta)$ is a crepant dlt blow-up, then by \autoref{claim:linked_lc_centers_above_a_fixed_point} we can find a minimal lc center $W$ of $(Y,\Delta_Y)$ above $x$ whose image $f(W)$ is contained in $Z\cap Z'$. Since $f(W)$ is an lc center of $(X,\Delta)$, the result is proved.
\end{proof}

\begin{corollary}\label{proposition:minimal_lc_centers_are_normal_up_to_homeo}
Let $(X,\Delta)$ be a quasi-projective slc threefold pair over a perfect field of characteristic $>5$. Then the minimal lc centers of $(X,\Delta)$ are normal up to universal homeomorphism.
\end{corollary}
\begin{proof}
Let $Z\subset X$ be a minimal lc center of $(X,\Delta)$, and $z\in Z$ be any point. Choose an étale neighborhood $(z'\in Z')\to (z\in Z)$. After shrinking it if necessary, we may assume that it is a standard étale neighborhood \cite[02GI, 02GU]{Stacks_Project} and thus there is an étale neighborhood $(z'\in X')\to (z\in X)$ such that $Z'=Z\times_XX'$. Hence $Z'$ is an lc center of $(X',\Delta')$ and it is connected. If it was reducible, then the intersection of its components would be a union of lc centers by \autoref{corollary:intersection_of_lc_centers}, and their images in $X$ would also be lc centers \cite[2.15]{Kollar_Singularities_of_the_minimal_model_program}. This contradicts the minimality of $Z$. Thus $(z'\in Z')$ is irreducible. 

By \cite[0BQ4]{Stacks_Project} we obtain that $Z$ is geometrically unibranch. So the normalization morphism $Z^n\to Z$ is universally bijective \cite[0C1S]{Stacks_Project}. Since it is surjective and universally closed, we obtain that that $Z^n\to Z$ is a universal homeomorphism.
\end{proof}

\subsubsection{Springs and sources for the reduced boundary}\label{section: springs and sources}

We are now able to define sources of lc centers, analogously to \cite[\S 4.5]{Kollar_Singularities_of_the_minimal_model_program}. The theory of sources should be thought of as higher codimension version of adjunction for divisors. However we only define sources for lc centers that are contained in the reduced boundary: see \autoref{remark:sources_for_general_lc_centers} below.

\begin{theorem}\label{theorem:spring_and_sources}
Let $f\colon (Y,\Delta_Y)\to (X,\Delta)$ be as in \autoref{notation:crepant_dlt_blow_up_of_threefold}. Let $Z\subset X$ be a lc center contained in $\Delta^{=1}$ with normalization $n\colon Z^n\to Z$.

Let $(S,\Delta_S:=\Diff^*_{S}\Delta_Y)\subset Y$ be a minimal lc center above $Z$, with Stein factorization $f^n_S\colon S\to Z_S\to Z^n$. 
Then:
	\begin{enumerate}
		\item \textsc{Uniqueness of sources.} The crepant birational class of $(S,\Delta_S)$ over $Z$ does not depend on the choice of $S$. We call it the \textbf{source} of $Z$, and denote it by $\src(Z, Y,\Delta_Y)$.
		\item \textsc{Uniqueness of springs.} The isomorphism class of $Z_S$ over $Z$ does not depend on the choice of $S$. We call it the \textbf{spring} of $Z$, and denote it by $\spr(Z,Y,\Delta_Y)$.
		\item \textsc{Crepant log structure.} $(S,\Delta_S)$ is dlt, $K_S+\Delta_S\sim_{\bQ,Z}0$ and $(S,\Delta_S)$ is klt on the general fiber above $Z$.
		\item \textsc{Galois property.} The field extension $k(Z)\subset k(Z_S)$ is Galois and the morphism $\Bir^c_Z(S,\Delta_S)\to\Gal(Z_S/Z)$ is surjective.
		\item \textsc{Poincaré residue map.} For $m>0$ divisible enough, there are well-defined isomorphisms
				$$\omega_Y^{[m]}(m\Delta_Y)|_S\cong \omega_S^{[m]}(m\Delta_S)$$
		and
				$$n^*(f_*(\omega_Y^{[m]}(m\Delta_Y)))\cong ((f^n_S)_*\omega_S^{[m]}(m\Delta_S))^{\Bir^c_Z(S,\Delta_S)}.$$
		
		\item \textsc{Birational invariance.} Let $(Y',\Delta_{Y'})$ be another crepant dlt blow-up of $(X,\Delta)$. Then
				\begin{eqnarray*}
				\src(Z,Y,\Delta_Y)&\overset{\emph{cbir}}{\cong}& \src(Z,Y',\Delta_{Y'}), \\
				\spr(Z,Y,\Delta_Y)&\cong & \spr(Z,Y',\Delta_{Y'})
				\end{eqnarray*}
			over $Z$. In particular, we may write $\src(Z,X,\Delta)$ for the source, and $\spr(Z,X,\Delta)$ for the spring of $Z\subset X$.
	\end{enumerate}
\end{theorem}
\begin{proof}
For clarity, we divide the proof into four steps and several claims. Parts $(a),(b),(c)$ and $(f)$ are proved in Step 1. Part $(e)$ is proved in Step 2. Steps 3 and 4 are devoted to the proof of $(d)$.

\bigskip
\textsc{Step 1: Sources, springs and invariance.} If $S'\subset Y$ is another minimal lc center above $Z$, then by \autoref{prop:P1_link_from_crepant_structure} and \autoref{lemma:P1_link_preserves_minimality} there is a crepant birational map $(S,\Delta_S)\dashrightarrow (S',\Delta_{S'})$ over $Z$ given by a composition of weak direct $\bP^1$-links. Since $Z_S=\Spec_{Z^n}f_*\sO_S$, we obtain a $Z$-isomorphism $Z_S\cong Z_{S'}$. This shows the uniqueness of $(S,\Delta_S)$ up to crepant birational maps over $Z$, and the uniqueness of $Z_S$ up to isomorphisms over $Z$.

The fact that $(S,\Delta_S)$ is dlt follows from \autoref{cor:Higher_Poincaré_residues}, and that fact that it is klt on the generic fiber above $Z$ holds by \autoref{proposition: lc centers on dlt 3-folds} and the fact that $S$ is minimal.

To obtain birational invariance, we apply \cite[4.44]{Kollar_Singularities_of_the_minimal_model_program} to a common log resolution of $(Y,\Delta_Y)$ and $(Y',\Delta_{Y'})$. (The proof of \cite[4.44]{Kollar_Singularities_of_the_minimal_model_program} uses \cite[10.45.2]{Kollar_Singularities_of_the_minimal_model_program}, which we replace by \autoref{fact:resolution_of_singularities_for_threefolds}). 

The following observation will be crucial for the rest of the proof:
	\begin{claim}\label{claim:dimension_of_sources}
		If $\dim Z=2$ then $\dim \src(Z,Y,\Delta_Y)=2$. If $\dim Z\leq 1$ then $\dim \src(Z,Y,\Delta_Y)\leq 1$.
	\end{claim}
\begin{proof}
\renewcommand{\qedsymbol}{$\lozenge$}
Indeed, if $Z$ is a divisor, then $S$ can be chosen to be the strict transform $Z_Y$ of $Z$. If $Z$ is a curve or a single point, contained in a component $\Gamma$ of $\Delta^{=1}$, then by \autoref{cor:adjunction_for_lc_centers} we can choose $S$ to be contained in $\Gamma_Y$. Since $\Gamma_Y$ is $2$-dimensional and not $f$-exceptional, $S$ is at most $1$-dimensional. Since the dimension of $S$ depends only on $Z$, our claim is proved.
\end{proof}

\bigskip
\textsc{Step 2: Poincaré residue.} To obtain the Poincaré residue map, let $m>0$ be even and sufficiently divisible such that $\omega_Y^{[m]}(m\Delta_Y)\sim f^*L$ for some line bundle $L$ on $X$. By \autoref{cor:Higher_Poincaré_residues}, for every $S$ we have a canonical isomorphism
		$$\sR^m_{Y,S}\colon f^*L|_S\cong \omega_Y^{[m]}(m\Delta_Y)\overset{\sim}{\longrightarrow} \omega_S^{[m]}(m\Delta_S).$$
	However the choice of $S$ is in general not unique. The following claim shows how the maps $\sR^m_{Y,S}$ relate for different choices of $S$.
	
	\begin{claim}\label{claim:Poincare_residue_map}
	Let $(S,\Delta_S)$ and $(S',\Delta_{S'})$ be minimal lc centers above $Z$. Then there is a crepant birational map $\phi\colon (S,\Delta_S)\dashrightarrow (S',\Delta_{S'})$ such the diagram
		\begin{equation}\label{diagram:Poincare_residue_diagram}			
			\begin{tikzcd}
			\omega_Y^{[m]}(m\Delta_Y) 
			\arrow[r, "\sim"] \arrow[d, "\sR_{Y,S}^m"]
			& f^*L
			\arrow[r, "\sim"]
			& \omega_Y^{[m]}(m\Delta_Y)
			\arrow[d, "\sR_{Y,S'}^m"]
			\\
			\omega_S^{[m]}(m\Delta_S)
			\arrow[rr, "\phi^*"]
			&& \omega_{S'}^{[m]}(m\Delta_{S'})
			\end{tikzcd}
		\end{equation}
	is commutative.
	\end{claim}
	\begin{proof}
	\renewcommand{\qedsymbol}{$\lozenge$}
	Indeed, we may assume that $S$ and $S'$ are directly weakly $\bP^1$-linked. Then there is a lc center $T\subset Y$ containing both $S$ and $S'$, which is birational to the total space of a weak standard $\bP^1$-link $T'\to W$ whose sections map birationally to $S$ and $S'$. Moreover, the map $T\dashrightarrow W$ factors the morphism $T\to f(T)$. The induced projections $S\dashrightarrow W \dashleftarrow S'$ are birational, and induce a birational map $\phi\colon S\dashrightarrow S'$, which we claim is the map we are looking for.

	Since $\phi$ is obtained from a $\bP^1$-link, by \autoref{lemma:P1_link_preserves_minimality} it is crepant.
	
	To prove the commutativity of the diagram, since $\sR^m_{T,S}\circ\sR^m_{Y,T}=\sR^m_{Y,S}$, we may assume that $Y=T$. In this case, note that $S,S'\subset \Delta_Y^{=1}$. Moreover we are dealing with torsion-free sheaves, so it is enough to check commutativity generically. Thus we may assume that we have a standard weak $\bP^1$-link $X\to W$ factorizing $f$, with sections $S$ and $S'$. Localizing at the generic point of $W$, we may furthermore assume that $W$ is the spectrum of a field $L$ and that $X=\bP^1_L$ \cite[25.3]{Hartshorne_Deformation_Theory}. In this case $\Delta_Y=S+S'$ and we may choose coordinates $x,y$ on $X$ such that $S=[0;1]$ and $S'=[1;0]$. Then a generator of $H^0(\bP^1_L,\omega_{\bP^1_L}(S+S'))$ is $dx/x$, and
			$$\sR^1_{Y,S}(dx/x)=1,\quad \sR^1_{Y,S'}(dx/x)=-1$$
while $\phi^*$ is the identity map on $L$. Thus \autoref{diagram:Poincare_residue_diagram} indeed commutes for $m$ even. 
	\end{proof}
	
	\begin{claim}\label{claim:roots_of_unity}
	The group $\Bir^c_Z(S,\Delta_S)$ acts on $\omega_S^{[m]}(m\Delta_S)$ as a finite group of $r^\text{th}$-roots of unity.
	\end{claim}
	\begin{proof}
	\renewcommand{\qedsymbol}{$\lozenge$}
	We are dealing with torsion-free sheaves and with a group action that commutes with the projection to $Z$. So to understand the action, we can localize over the generic point of $Z$. Then we obtain a proper $k(Z)$-pair $(S_{k(Z)},\Delta_{S_{k(Z)}})$ such that $\omega_{S_{k(Z)}}^{[m]}(m\Delta_{S_{k(Z)}})$ is trivial, and we must show that the action of $\Bir^c_{k(Z)}(S_{k(Z)},\Delta_{S_{k(Z)}})$ on the $1$-dimensional $k(Z)$-vector space $H^0(S_{k(Z)},\omega_{S_{k(Z)}}^{[m]}(m\Delta_{S_{k(Z)}}))$is finite. If $\dim S=\dim Z$ then $S_{k(Z)}$ is the spectrum of a finite field extension of $k(Z)$, and so $\Aut_{k(Z)}(S_{k(Z)})$ is finite. By \autoref{claim:dimension_of_sources}, the only case left is when $Z=\{x\}$ is a closed point of $X$ and $(S,\Delta_S)$ is a proper $1$-dimensional klt pair  over $k(x)$ such that $\omega_S^{[m]}(m\Delta_S)$ is trivial. Since $k(x)$ is a perfect field, it follows from \autoref{proposition:pluricanonical_representations_for_curves_I} that $\Bir^c_{k(x)}(S,\Delta_S)$ acts finitely on the $1$-dimensional vector space $H^0(S,\omega_S^{[m]}(m\Delta_S))$. 
	
	So in every case $\Bir^c_Z(S,\Delta_S)$ acts finitely on the generic stalk of $\omega_S^{[m]}(m\Delta_S)$, hence through the multiplication with some $r^\text{th}$-root of unity in $k(Z)$.
	\end{proof}
	
	If we think about $(S,\Delta_S)$ as a crepant birational class, then \autoref{claim:Poincare_residue_map} shows that we can define a Poincaré residue map $\sR^m_{Y,S}\colon \omega_Y^{[m]}(m\Delta_Y)\to \omega_S^{[m]}(m\Delta_S)$, up to the action of the group $\Bir^c_Z(S,\Delta_S)$ on the target. We can remedy to this ambiguity by replacing $m$ with $mr$, where $r$ is given by \autoref{claim:roots_of_unity}.

\bigskip
\textsc{Step 3: the Galois property.} We wish to prove that the field extension $k(Z_S)/k(Z)$ is Galois. 

The case where $Z$ is a divisor is the easiest, we treat it first. By \cite[1.2]{Nakamura_Tanaka_Witt_Nadel_vanishing_for_threefolds} there is a unique lc center $S$ above $Z$ (namely, its strict transform) and the morphism $S\to Z$ has connected fibers in a neighborhood of the generic point of $Z$. Thus $S\to Z$ is birational, $\Bir^c_Z(S,\Delta_S)$ is trivial and $k(Z_S)=k(Z)$.

From now on, we assume that $\dim Z\leq 1$. To prove that the finite morphism $Z_S\to Z$ induces a Galois extension on the function fields, we may localize at the generic point of $Z$. Then the situation is the following: $(x\in X,\Delta)$ is a local lc pair of dimension $2$ or $3$, $(Y,\Delta_Y)\to (X,\Delta)$ is a crepant dlt blow-up and $(S,\Delta_S)\subset Y$ is an lc center of dimension $\leq 1$ such that the morphism $S\to X$ factorizes through the closed point, and $S$ is minimal for this property.

\begin{claim}\label{claim:fields_of_definition_are_equal}
In this set-up, the fields of definition of the geometric connected components of $S$ are the same, and it is a Galois extension of $k(x)$.
\end{claim}
\begin{proof}
\renewcommand{\qedsymbol}{$\lozenge$}
Indeed, let $K^s$ be a separable closure of $k(x)$. Then $(Y,\Delta_Y)\times_{k(x)}K^s$ is a dlt $\bQ$-factorial pair, and every component of $S_{K^s}$ is an lc center that is minimal. Let $W$ be one of them, with field of definition $F$. Then every lc center containing $W$ is also defined over $F$ \cite[4.17]{Kollar_Singularities_of_the_minimal_model_program}, and so any lc center that is weakly $\bP^1$-linked to $W$ is also defined over $F$. By \autoref{prop:minimal_centers_are_linked} (see also \autoref{remark:minimal_centers_are_linked_in_smaller_dimensions}) we obtain that all the components of $S_{K^s}$ are defined over $F$. If $\sigma$ is an element of the Galois group of the Galois closure of the field extension $k(x)\subset F$, then $W^\sigma$ is a minimal lc center defined over the conjugate field $F^\sigma$. Therefore $F^\sigma=F$, so we see that $k(x)\subset F$ is a Galois field extension. 
\end{proof}

\begin{claim}\label{claim:global_sections_is_separable}
$H^0(S,\sO_S)$ is a separable field extension of $k(x)$.
\end{claim}
\begin{proof}
\renewcommand{\qedsymbol}{$\lozenge$}
If $(x\in X)$ has dimension $3$, this holds because $k(x)$ is perfect (since in this case we localized at a closed point). If $(x\in X)$ has dimension $2$, then $(x\in X,\Delta)$ is a local lc surface singularity with the property that $\Delta^{=1}\neq 0$. (By \cite[2.28]{Kollar_Singularities_of_the_minimal_model_program}, these $(x\in X)$ are rational surface singularities.) The possible dual graphs of the minimal resolutions $(T,\Gamma)$ of such pairs are classified, see for example \cite[3.31]{Kollar_Singularities_of_the_minimal_model_program}. Inspecting them, we see that if $C\subset T$ is an exceptional proper curve above $x$, then $\dim_{k(x)}H^0(C,\sO_C)\leq 4$, and if $C'$ is another exceptional proper curve, then $C\cdot C'=\length_{k(x)}C\cap C'\leq 4$. Hence if $E\subset T$ is a lc place of $(x\in X,\Delta)$ then $H^0(E,\sO_E)$ is a separable extension of $k(x)$ provided that $\Char k(x)\geq 5$.
\end{proof}

\begin{claim}\label{claim:global_sections_is_Galois}
$H^0(S,\sO_S)$ is a Galois extension of $k(x)$.
\end{claim}
\begin{proof}
\renewcommand{\qedsymbol}{$\lozenge$}
By \autoref{lemma:Galois_closure_is_a_field_of_definition} and \autoref{claim:global_sections_is_separable}, the field of definition of any connected component of $S_{K^s}$ contains $H^0(S,\sO_S)$ and is contained in the Galois closure $K$ of the field extension $k(x)\subset H^0(S,\sO_S)$. On the other hand, 
the field of definition of these components is a Galois extension of $k(x)$ by \autoref{claim:fields_of_definition_are_equal}. So we deduce that $K$ is the field of definition of every connected component of $S_{K^s}$.

It remains to show that $H^0(S,\sO_S)=K$. Consider the
Cartesian diagram
		$$\begin{tikzcd}
		S\arrow[d] & S_K=\bigcup_i S_K^{(i)}\arrow[l]\arrow[d] \\
		k(x) & K\arrow[l]
		\end{tikzcd}$$
where the $S_K^{(i)}$ are the irreducible components of $S_K$. As before, $(Y,\Delta_Y)\times_{k(x)} K$ is a dlt pair above $K$, and each $S_K^{(i)}$ is an lc center. Thus the intersection of the $S_K^{(i)}$ are also lc centers. By minimality of $S$, we obtain that the $S_K^{(i)}$ are disjoint. So the $S_K^{(i)}$ are also the connected components of $S_K$. Since $K$ is Galois over $k(x)$, the Galois group $G=\Gal(K/k(x))$ acts on $S_K$. Assume that some $S_K^{(j)}$ is stable under the action of a non-trivial subgroup of $G$. Then by \autoref{lemma:Galois_descent_of_schemes}, $S_K^{(j)}$ would be defined over a proper sub-extension $K'$ of $k(x)\subset K$. Then every $S_K^{(i)}$ would be defined over $K'$, which is a contradiction with the previous paragraph. Thus $G$ acts freely on the set of $S_K^{(i)}$'s. The action is also transitive, since $S_K\to S$ is the geometric quotient by $G$.

By flat base-change, we have
		$$H^0(S_K,\sO_{S_K})=H^0(S,\sO_S)\otimes_{k(x)}K$$
and the $G$-action is given by the action on $K$. On the other hand, since $S_K$ is the disjoint union of the $S_K^{(i)}$ and $G$ permutes them freely, we also have
		$$H^0(S_K,\sO_{S_K})=\prod_{\sigma\in G} K$$
where $G$ acts by permuting the factors. Taking $G$-invariants, we obtain
		$$H^0(S,\sO_S)=H^0(S,\sO_S)\otimes_{k(x)}K^G=\left(H^0(S,\sO_S)\otimes_{k(x)}K\right)^G=\left(\prod_{\sigma\in G}K\right)^G=K$$
as desired. 
\end{proof}

\bigskip
\textsc{Step 4: Galois group and crepant birational maps.}
It remains to show that every element of the Galois group $\Gal(Z_S/Z)$ is induced by a crepant birational self-map of $(S,\Delta_S)$. 

There is nothing to prove if $\dim Z=2$, as noticed at the beginning of the previous step.

From now on assume that $\dim Z\leq 1$. We have proved in the previous step that $K:=k(Z_S)$ is Galois over $k(Z)$. Let $W$ be a component of $S\times_{k(Z)}K$ and pick $\sigma\in \Gal(K/k(Z))$. The proof of \autoref{claim:global_sections_is_Galois} shows that $W$ is defined over $K$, so it has a conjugate $W^\sigma\subset S\times_{k(Z)}K$. Fix a weak $\bP^1$-link between $W$ and $W^\sigma$ inside $Y\times_X K$: then the $\Gal(K/k(Z))$-orbit of this link descends to an element of $\Bir^c_{k(Z)}((S,\Delta_S)\times_{Z}k(Z))$, and in turn this crepant birational map induces $\sigma$ on $H^0(S\times_{k(Z)}K,\sO)=K$. 

This proves that Galois automorphisms of $k(Z)\subset k(Z_S)$ are induced by birational self-maps of $S$ which are generically $(K_S+\Delta_S)$-crepant over $Z$. We need to show that these maps are crepant, not only generically crepant. If $\dim Z=0$ there is nothing to show, so assume $\dim Z=1$. Then we make the following observation: $Z$ is contained in a component $D$ of $\Delta^{=1}$ and by adjunction any preimage of $Z$ in $D^n$ is a codimension one lc center of $(D^n,\Diff_{D^n}(\Delta-D))$. By the classification of lc surface singularities \cite[2.31]{Kollar_Singularities_of_the_minimal_model_program}, we see that $D$ is regular at the generic point of $Z$. Hence the strict transform of $Z$ must appear in $\lfloor \Diff_{D^n}(\Delta-D)\rfloor$, and by adjunction we get a natural boundary $\Gamma$ on $Z^n$. Moreover the equality $\dim S=\dim Z$ implies $S=Z_S$, so $f_S^n\colon S\to Z^n$ is a finite morphism. By the definition of $(S,\Delta_S)$ we must have
		\begin{equation}\label{eqn:source_boundary_in_finite_case}
		(f^n_S)^*(K_{Z^n}+\Gamma)=K_S+\Delta_S.
		\end{equation}
On the other hand, the fact that $S=Z_S$ shows that the birational self-maps of $S$ we found are just the Galois automorphisms of $Z_S$ over $Z^n$. Since $\Delta_S$ can be defined by \autoref{eqn:source_boundary_in_finite_case}, we see that Galois automorphisms are crepant. 

The proof is complete.
\end{proof}


\begin{corollary}[Adjunction]\label{cor:adjunction_for_reduced_boundary}
Let $(X,D+\Delta)$ be a quasi-projective lc threefold pair above a perfect field of characteristic $>5$, where $D$ is a reduced divisor with normalization $n\colon D^n\to D$. Let $Z\subset D$ be an lc center of $(X,D+\Delta)$, and $Z_D\subset D^n$ be an irreducible variety such that $n(Z_D)=Z$. Then:
	\begin{enumerate}
		\item $Z_D$ is an lc center of $(D^n,\Diff_{D^n}\Delta)$,
		\item there is a commutative diagram
				$$\begin{tikzcd}
				\src(Z_D, D^n,\Diff_{D^n}^*\Delta) \arrow[rr, dashed, "\emph{cbir}"] \arrow[d] && \src(Z,X,D+\Delta)\arrow[d] \\
				D^n \arrow[rr, "n"] && D
				\end{tikzcd}$$
		\item There is an isomorphism $\spr(Z_D,D^n,\Diff_{D^n}^*\Delta)\cong \spr(Z,X,D+\Delta)$.
	\end{enumerate}
\end{corollary}
\begin{proof}
By enlarging $\Delta$, we may assume that $D$ is irreducible. Choose a $\bQ$-factorial dlt crepant blow-up $(Y,\Delta_Y)$ of $(X,D+\Delta)$ and let $D_Y$ be the strict transform of $D$. By \autoref{proposition: lc centers on dlt 3-folds} the pair $(D_Y,\Diff^*_{D_Y}\Delta_Y)$ is dlt. In particular it is normal, so the Stein factorization of $D_Y\to D$ is the normalization $D^n$. So by construction $(D_Y,\Diff_{D_Y}^*\Delta_Y)\to (D^n,\Diff_{D^n}\Delta)$ is a crepant dlt blow-up. By \autoref{cor:adjunction_for_lc_centers} there is a lc center $W\subset D_Y$ of $(Y,\Delta_Y)$ that dominates $Z_D\subset D^n$. By \autoref{proposition: lc centers on dlt 3-folds} $W$ is also an lc center of $(D_Y,\Diff_{D_Y}^*\Delta_Y)$, which proves the first item. Moreover $W$ is a representative for both $\src(Z_D, D^n,\Diff_{D^n}\Delta)$ and $\src(Z,Y,\Delta_Y)$, and the second item follows by uniqueness of the source up to crepant birational map. The third item is a consequence of the second item.
\end{proof}

\begin{remark}\label{remark:sources_for_general_lc_centers}
The same method gives sources, springs, crepant log structure, adjunction and birational invariance of arbitrary lc centers of $(X,\Delta)$. However:
	\begin{enumerate}
		\item The Galois property is problematic: if $Z$ is $1$-dimensional and not contained in any component of $\Delta^{=1}$, then $\sO_{X,Z}$ might by an elliptic or a cusp singularity. Since the residue field $k(Z)$ is not perfect, the degrees of the exceptional curves on the resolution can in theory be arbitrarily large \cite[3.27.3]{Kollar_Singularities_of_the_minimal_model_program}, and inseparable field extensions might appear. It is interesting whether a variant of \autoref{claim:global_sections_is_Galois} still holds (for example: is the maximal separable intermediate extension of $k(X)\subset H^0(S,\sO_S)$ Galois over $k(x)$?), but we shall not study this question here.
		\item The Poincaré residue map can be defined, but it is not clear how to get rid of the $\Bir^c_Z(S,\Delta_S)$-ambiguity. Indeed, in constrast to \autoref{claim:dimension_of_sources}, one more case can show up, namely $\dim S=2$ and $\dim Z=1$, and I was not able to show finiteness of representations with sufficient generality in this case. 
		
	\end{enumerate}
\end{remark}

\subsection{Gluing theorems for threefolds in characteristic $>5$}\label{section: proof of main result}

We now prove the gluing theorems for lc threefolds. Our proofs follow closely those of Koll\'{a}r \cite[\S 5]{Kollar_Singularities_of_the_minimal_model_program}.

\begin{lemma}\label{lemma:isomorphism_preserves_sources}
Let $(X,\Delta)$ and $(X',\Delta')$ be quasi-projective lc threefold pairs over a perfect field of characteristic $>5$, and $\phi\colon (X,\Delta)\cong (X',\Delta')$ a log isomorphism. Let $Z\subset X$ be an lc center of $(X,\Delta)$. Then:
	\begin{enumerate}
		\item $Z':=\phi(Z)$ is an lc center of $(X',\Delta')$, and
		\item we have a commutative square
				$$\begin{tikzcd}
				\src(Z,X,\Delta) \arrow[d]\arrow[r, dotted, "\emph{cbir}"] & \src(Z',X',\Delta')\arrow[d] \\
				Z\arrow[r, "\phi"] & Z'
				\end{tikzcd}$$
	\end{enumerate}
\end{lemma}
\begin{proof}
Since $\phi$ is a log isomorphism, $Z'$ is an lc center of $(X',\Delta')$. The second point follows from the birational invariance property proved in \autoref{theorem:spring_and_sources}. The source is defined up to crepant birational map over the lc center, so there is no ambiguity about the commutativity of the diagram.
\end{proof}

First we show that the geometric quotient exists.

\begin{proposition}\label{prop:quotient_of_lc_pairs_dim_3}
Let $(X,D+\Delta)$ be a projective lc threefold pair over a perfect field $k$ of characteristic $>5$. Let $\tau\colon (D^n, \Diff_{D^n}\Delta)\cong (D^n, \Diff_{D^n}\Delta)$ be a generically fixed point free involution. Assume that $K_X+D+\Delta$ is ample. Then the induced equivalence relation $R(\tau)\rightrightarrows X$ is finite, the geometric quotient $X/R(\tau)$ exists, and $X/R$ is proper over $k$.
\end{proposition}
\begin{proof}
By \cite[Theorem 6]{Kollar_Quotients_by_finite_equivalence_relations}, the geometric quotient exists as soon as $R:=R(\tau)$ is finite. Notice that, by construction, a point of $\Supp(D)$ and a point of $X\setminus \Supp(D)$ cannot be in the same equivalence class. Moreover, $R$ restricts to the identity relation away from $\Supp(D)$. Thus we only need to prove that $R$ is finite on $\Supp(D)$.

For every lc center $Z$ of $(X,D+\Delta)$ contained in $D$, let $Z^0:= Z\setminus \text{(lower dimensional lc centers)}$ and let $\spr^0(Z,X,D+\Delta)$ be the preimage of $Z^0$ through the finite morphism $\spr(Z,X,D+\Delta)\to Z$. Set
		$$p\colon \spr^0(X,D+\Delta,\subseteq D):=\bigsqcup_{Z\subseteq D}\spr^0(Z,X,D+\Delta)\longrightarrow X,$$
	it is a quasi-finite morphism mapping surjectively onto $D$.
	
	The equivalence relation $R\rightrightarrows D$ pullbacks to an equivalence relation $(p\times p)^*R\rightrightarrows\spr^0(X,D+\Delta,\subseteq D)$ that commutes with the projections to $D$, and since $p$ is surjective onto $D$ it is sufficient to show that $(p\times p)^*R$ is finite. We describe how the generators of $R$ pullbacks to $\spr^0(X,D+\Delta,\subseteq D)$. To make book-keeping easier, we let $\{Z_j\}_j$ be the set of lc centers of $(X,D+\Delta)$ contained in $D$.
	
	By the Galois property of \autoref{theorem:spring_and_sources}, over the normal locus of $Z_j^0$ the morphism $\spr^0(Z_j,X,D+\Delta)\to Z_j^0$ is the quotient by the Galois group $G_j:=\Gal(\spr(Z_j,X,D+\Delta)/Z_j)$. Thus the preimage of the diagonal $Z_j^0\times Z_j^0$ under $p$ is a union of the graphs of the $G_j$-action, together with other components that do not dominate $Z_j^0$ (their images are contained in the locus where $Z_j^0$ is not normal).
	
	Next we understand the pullback of $\tau$.
	
	\begin{claim}\label{claim:involution_lifts_to_crepant_map_of_sources}
	Let $Z_{jh}\subset D^n$ be a subvariety dominating $Z_j\subset X$. Let $Z_l:=n(\tau(Z_{jh}))$. Then $Z_{jh}$ is a lc center of $(D^n,\Diff_{D^n}\Delta)$. Moreover
	\begin{enumerate}
		\item $\tau$ induces a crepant birational map
				$$\tilde{\tau}_{jhl}\colon\src(Z_j,X,D+\Delta)\overset{\emph{cbir}}{\sim} \src(Z_l,X,D+\Delta)$$
			determined up to the left and right action of $\Bir^c_{Z_j}\src(Z_j,X,D+\Delta)$ and $\Bir^c_{Z_l}\src(Z_l,X,D+\Delta)$;
		\item $\tilde{\tau}_{jhl}$ induces an isomorphism
				$$\tau_{jhl}\colon \spr^0(Z_j,X,D+\Delta)\cong \spr^0(Z_l,X,D+\Delta)$$
			determined up to left and right multiplication by $G_j$ and $G_l$.
	\end{enumerate}		
	\end{claim}
	\begin{proof}
	\renewcommand{\qedsymbol}{$\lozenge$}
	By \autoref{cor:adjunction_for_reduced_boundary}, $Z_{jh}$ is an lc center of $(D^n,\Diff_{D^n}\Delta)$ and
			\begin{equation}\label{eqn:adjunction_for_sources}
			\src(Z_{jh},D^n,\Diff_{D^n}\Delta)\overset{\text{cbir}}{\sim}\src(Z_j,X,D+\Delta).
			\end{equation}
			By \autoref{lemma:isomorphism_preserves_sources} the automorphism $\tau$ of $(D^n,\Diff_{D^n}\Delta)$ induces a crepant birational map between $\src(Z_{jh},D^n,\Diff_{D^n}\Delta)$ and $\src(\tau(Z_{jh}),D^n,\Diff_{D^n}\Delta)$. Then by \autoref{eqn:adjunction_for_sources} we obtain crepant birational maps
			$$\tilde{\tau}_{jkl}\colon \src(Z_j,X,D+\Delta)\overset{\text{cbir}}{\sim} \src(Z_l,X,D+\Delta).$$
		Since $\tilde{\tau}_{jhl}$ preserves the non-klt locus and since $\spr^0(Z_j,X,D+\Delta)$ is precisely the image of the klt locus of $\src(Z_j,X,D+\Delta)$, the $\tilde{\tau}_{jhl}$ descend to an isomorphism
		$$\tau_{jhl}\colon \spr^0(Z_j,X,D+\Delta)\cong \spr^0(Z_l,X,D+\Delta)$$
	as claimed. However $\tau_{jhl}$ is not uniquely defined, since $\tilde{\tau}_{jhl}$ is determined up to left and right multiplication by $\Bir^c_{Z_j}\src(Z_j,X,D+\Delta)$ and $\Bir^c_{Z_l}\src(Z_l,X,D+\Delta)$. By \autoref{theorem:spring_and_sources} we obtain that $\tau_{jhl}$ is determined up to left and right multiplication by $G_j$ and $G_l$.
	\end{proof}

	Since $n(n^{-1}Z_j)=\bigcup_i D_i\cap Z_j$ and each $D_i\cap Z_j$ is a union of lc centers by \autoref{corollary:intersection_of_lc_centers}, we see that each component of $n^{-1}Z_j$ dominates a lc center of $(X,D+\Delta)$. So, thanks to \autoref{claim:involution_lifts_to_crepant_map_of_sources}, we have found all the generators of $(p\times p)^*R$.
	
	To show that $(p\times p)^*R$ is finite, by \cite[9.55]{Kollar_Singularities_of_the_minimal_model_program} it is sufficient to show that it is finite over the generic point of every $Z_j^0$. Therefore we may assume that $(p\times p)^*R$ is the groupoid generated by the $G_j$ and the $\tau_{jhl}$, and the stabilizer of $\spr^0(Z_j,X.D+\Delta)$ is generated by the sets $\{\tau^{-1}_{jh'l}G_l\tau_{jhl}\}_{h,h',l}$.
	
	The Galois property of \autoref{theorem:spring_and_sources} shows that $G_j$ is a subgroup of
			$$\Aut^s_k\spr(Z_j,X,D+\Delta):= \im\left[ \Bir^c_{k}\src(Z_j,X,D+\Delta)\to \Aut_k\spr(Z_j,X,D+\Delta)\right].$$
	By \autoref{claim:involution_lifts_to_crepant_map_of_sources}, the $\tau^{-1}_{jh'l}G_l\tau_{jhl}$ are also subgroups of $\Aut^s_k\spr(Z_j,X,D+\Delta)$. To complete the proof we will show that these groups of $k$-automorphisms are finite. This is where the ampleness assumption on $K_X+D+\Delta$ will be used.	
		\begin{enumerate}
			\item If $\dim Z_j=2$ then $\src(Z_j,X,D+\Delta)\to Z_j$ is birational and therefore $\Aut^s_k(Z_j,X,D+\Delta)$ is finite by \autoref{proposition:pluricanonical_representation_for_surfaces}.
			\item If $\dim Z_j=1$ then $\dim \src(Z_j,X,D+\Delta)=1$ by \autoref{claim:linked_lc_centers_above_a_fixed_point}. Hence $\src(Z_j,X,D+\Delta)$ is a $1$-dimensional pair of general type, since its log canonical divisor is the pullback of the ample divisor $K_X+D+\Delta$ through the composition of finite morphisms
					$$\src(Z_j,X,D+\Delta)\to Z_j\hookrightarrow X.$$
			Thus $\Bir^c\src(Z_j,X,D+\Delta)$ is finite by \autoref{proposition:pluricanonical_representations_for_curves_II}.
			\item If $\dim Z_j=0=\dim \src(Z_j,X,D+\Delta)$ then finiteness is clear.
			\item The only case left, according to \autoref{claim:dimension_of_sources}, is $\dim Z_j=0$ and $\dim \src(Z_j,X,D+\Delta)=1$. Then $\src(Z_j,X,D+\Delta)$ is a Calabi-Yau curve and finiteness follows from \autoref{proposition:pluricanonical_representations_for_curves_I}.
		\end{enumerate}
Thus $X/R$ exists and is a $k$-scheme. Since $X\to X/R$ is finite, $X/R$ is proper over $k$ \cite[09MQ,03GN]{Stacks_Project}. This completes the proof.
\end{proof}

\begin{remark}
With the notations of the proof of \autoref{prop:quotient_of_lc_pairs_dim_3}, let us observe that the stabiliser of $\spr(Z_j^0,X,D+\Delta)$ is not contained in the smaller group
		\begin{equation*}\label{eqn:source_automorphisms_over_the_base}
		\Aut_{Z_j}\spr(Z_j,X,D+\Delta):=\im\left[ \Bir^c_{Z_j}\src(Z_j,X,D+\Delta)\to\Aut_{Z_j}\spr(Z_j,X,D+\Delta)\right]
		\end{equation*}
	even though $G_j$ belongs to it and each $\tau_{jhl}$ commutes with the projections to $Z_j$ and $Z_l$. Indeed the stabiliser is generated by the groups $\tau_{jh'l}^{-1}G_l\tau_{jhl}$ where it may happen that $h\neq h'$. In this case the corresponding automorphisms of $\spr(Z_j,X,D+\Delta)$ may not commute with the projection to $Z_j$. This happens if $Z_j$ is dominated by several lc centers of $(D^n,\Diff_{D^n}\Delta)$ whose images under $\tau$ dominate the same lc center $Z_l\subset X$.
	
	Notice that $\Aut_{Z_j}\spr(Z_j,X,D+\Delta)$ is a finite group since $\spr(Z_j,X,D+\Delta)\to Z_j$ is a finite morphism. On the other hand, it is not clear that $\Aut^s_k\spr(Z_j,X,D+\Delta)$ should be finite, and this is where the ampleness assumption comes into the picture.
\end{remark}


Now we show that the log canonical divisor descends to the geometric quotient.

\begin{proposition}\label{prop:descent_of_canonical_bundle_in_dim_3}
Let $(X,D+\Delta)$ be a quasi-projective lc threefold pair over a perfect field $k$ of characteristic $>5$. Let $\tau\colon (D^n, \Diff_{D^n}\Delta)\cong (D^n, \Diff_{D^n}\Delta)$ be an involution. Assume that $R(\tau)\rightrightarrows X$ is finite, let $q\colon X\to Y:=X/R(\tau)$ be the geometric quotient, and let $\Delta_Y:=q_*\Delta$. Then $K_Y+\Delta_Y$ is $\bQ$-Cartier.
\end{proposition}
\begin{proof}
By \autoref{corollary:descent_in_codim_2} the $\bQ$-divisor $K_Y+\Delta_Y$ is $\bQ$-Cartier in codimension $2$. Hence we may localize over a closed point of $Y$, and assume that $Y$ is local with closed point $y$, and that $K_Y+\Delta_Y$ is $\bQ$-Cartier on $Y^0:=Y\setminus\{y\}$. Since $q\colon X\to Y$ is an isomorphism away from $D\to q(D)$, we may assume that $y$ belongs to the nodal locus. Since an $\sO_Y$-module is locally free if and only if it is locally free after an \'{e}tale base-change, we may base-change along the strict henselization of $Y$, and assume that $k(y)$ is separably closed. Since $k$ is perfect, we may therefore assume that $k(y)$ is algebraically closed.

We are going to descend the total space of a multiple of $K_X+D+\Delta$ to $Y$, and use the theory of Seifert bundles (see \cite[\S 9.3]{Kollar_Singularities_of_the_minimal_model_program}) to conclude that it defines a line bundle on $Y$. 

Choose an even integer $m>0$ such that $m\Delta$ is a $\bZ$-divisor, and both $\omega_X^{[m]}(mD+m\Delta)$ and $\omega_{Y^0}^{[rm]}(rm\Delta_Y|_{Y^0})$ are invertible sheaves. We consider the $\bA^1$-bundle over $X$ given by
		$$X_L:=\Spec_X\sum_{r\geq 0}\omega_X^{[rm]}(rmD+rm\Delta)\overset{p}{\longrightarrow} X.$$
	Set $D_L:=p^{-1}D$ and $\Delta_L:=p^{-1}\Delta$. Clearly $(X_L,D_L+\Delta_L)$ is lc. Since $X_L\to X$ is an $\bA^1$-bundle, we see that the normalization $D^n_L\to D_L$ is equal to $D^n\times_DD_L$. By adjunction and functoriality of the relative spectrum, this gives the alternative description
			$$D^n_L=\Spec_{D^n}\sum_{r\geq 0} \omega_{D^n}^{[rm]}(rm\Diff_{D^n}\Delta).$$
	The fiber product description shows that $\Diff_{D^n_L}\Delta_L=p^*\Diff_{D^n}\Delta$ and that the lc centers of $(D^n_L,\Diff_{D^n_L}\Delta_L)$ are the preimages of the lc centers of $(D^n,\Diff_{D^n}\Delta)$. As $\Diff_{D^n}\Delta$ is $\tau$-invariant, the relative spectrum description shows that $\tau$ lifts to an involution $\tau_L$ of the pair $(D^n_L,\Diff_{D^n_L}\Delta_L)$.
	
	Now we wish to show that the induced equivalence relation $R_L:=R(\tau_L)\rightrightarrows X_L$ is finite, so that we can form the quotient $X_L/R_L$.
	
	Denote by $p'\colon Y^0_L\to Y^0$ the total space of the invertible sheaf $\omega_{Y^0}^{[rm]}(rm\Delta_Y|_{Y^0})$, $X^0:=q^{-1}Y^0$ and  $X^0_L:=(p\circ q)^{-1}Y^0$. Then we have a natural finite morphism of $\bA^1$-bundles $q'\colon X^0_L\to Y^0_L$ making the diagram
			$$\begin{tikzcd}
			X^0_L\arrow[r, "p"]\arrow[d, "q'"] & X^0\arrow[d, "q"] \\
			Y^0_L\arrow[r, "q'"]& Y^0
			\end{tikzcd}$$
	commutative. Since $X^0_L$ is the total space of the line bundle $\omega_{X^0}^{[m]}(mD|_{X^0}+m\Delta|_{X^0})$ which descends to $Y^{0}$, we have that $Y^0_L=X^0_L/R_L^0$ where $R^0_L$ is the restriction of $R_L$ to $X^0_L$ \cite[9.48]{Kollar_Singularities_of_the_minimal_model_program}. Therefore $R^0_L$ is finite, and we only need to prove the finiteness of $R_L$ over the complement of $X^0_L$.
	
	Let $x_1,\dots,x_s\in X$ be the preimages of $y$. Since $y$ belongs to the nodal locus of $Y$, we have $x_1,\dots,x_s\in D$. If none of the $x_i$ are lc centers of $(X,D+\Delta)$, then every lc center of $(X_L,D_L+\Delta_L)$ intersects $X^0_L$ and therefore $R_L$ is finite by \cite[9.55]{Kollar_Singularities_of_the_minimal_model_program}.
	
	Assume that one of the $x_i$ is an lc center. Then every $x_i$ is an lc center: for the $x_i$ form an equivalence class of $\tau$, by adjunction one of them corresponds to an lc center of $(D^n,\Diff_{D^n}\Delta)$ and $\tau$ permutes these lc centers.
	
	Since $k(y)$ is algebraically closed, we have $k(x_i)=k(y)$ for each $i$. The fiber of $p\colon X_L\to X$ above $x_i$ is the spectrum of the symmetric algebra of the $1$-dimensional $k(y)$-vector space
			$$V_i:=\omega_X^{[m]}(mD+m\Delta)\otimes_{\sO_X} k(x_i)$$
	and $X_L\setminus X^0_L=\bigcup_i\Spec(\Sym V_i)$. If $x'_i\in D^n$ is a preimage of $x_i$, then the data of $\tau(x_i')=x_j'$ defines an isomorphism $\tau_{ijl}\colon V_i\to V_j$ (the index $l$ accounts for the fact that $x_i$ might have several preimages in $D^n$). The collection of isomorphisms $\{\tau_{ijl}\colon V_i\to V_j\}$ generates a groupoid, and $R_L$ is finite if and only if each group $\Stab(V_i)\subset \Aut(V_i)$ of all possible compositions $\tau_{ij_1l_1}\circ\dots\circ\tau_{j_nil_n}\colon V_i\to V_i$, is finite.
	
	To show this property, consider the sources $(S_i,\Delta_i):=\src(x_i,X,D+\Delta)$. These are proper Calabi-Yau varieties over $k(y)$. If we pullback $\omega_X^{[m]}(mD+m\Delta)$ to a crepant dlt blow-up, restrict it to (a model of) $S_i$ and take global sections, we obtain $V_i$. Thus the Poincaré residue maps constructed in \autoref{theorem:spring_and_sources} give canonical isomorphisms
			$$V_i\cong H^0(S_i,\omega_{S_i}^{[m]}(m\Delta_i)).$$
	Moreover \autoref{claim:involution_lifts_to_crepant_map_of_sources} shows that each isomorphism $\tau_{ijl}\colon V_i\to V_j$ is induced by a crepant birational map $\phi_{ijl}\colon (S_j,\Delta_j)\dashrightarrow (S_i,\Delta_i)$. Hence we conclude that
		\begin{equation}\label{eqn:sources_and_descent_of_canonical_bundle}
		\Stab(V_i)\subseteq \im\left[ \Bir^c(S_i,\Delta_i)\to \Aut_{k(y)}H^0(S_i,\omega_{S_i}^{[m]}(m\Delta_i))\right].
		\end{equation}
	Now observe that the $S_i$ are at most $1$-dimensional by \autoref{claim:dimension_of_sources}, and therefore by \autoref{proposition:pluricanonical_representations_for_curves_I} the right-hand side in \autoref{eqn:sources_and_descent_of_canonical_bundle} is finite. 
	
	It follows that $R_L$ is finite, and thus the quotient $X_L/R_L$ exists. By \cite[9.48]{Kollar_Singularities_of_the_minimal_model_program} the complement of the zero section is a Seifert bundle over $Y$, and by \cite[9.53]{Kollar_Singularities_of_the_minimal_model_program} this implies that its define a $\bQ$-line bundle on $Y$. By construction this $\bQ$-line bundle is equal to $\omega_{Y^0}^{[m]}(m\Delta_Y|_{Y^0})$ over $Y^0$. Hence some reflexive power of $\omega_Y^{[m]}(m\Delta_Y)$ is invertible, as was to be shown.
\end{proof}

\begin{theorem}\label{theorem:gluing_for_threefolds}
Let $k$ be a perfect field of characteristic $> 5$. Then normalization gives a one-to-one correspondence
	\begin{equation*}
		(\Char k > 5) \quad
		\begin{pmatrix}
		\text{Proper slc threefold pairs}\\
		(X,\Delta) \text{ such that}\\
		K_X+\Delta\text{ is ample}
		\end{pmatrix}
		\overset{1:1}{\longrightarrow}
		\begin{pmatrix}
		\text{Proper lc threefold pairs } (\bar{X},\bar{D}+\bar{\Delta})  \\
		\text{plus a generically fixed point free}\\
		\text{involution }\tau\text{ of }(\bar{D}^n,\Diff_{\bar{D}^n}\bar{\Delta}) \\
		\text{such that }K_{\bar{X}}+\bar{D}+\bar{\Delta} \text{ is ample.}
		\end{pmatrix}
		\end{equation*}
\end{theorem}
\begin{proof}
Given $(\bar{X},\bar{D}+\bar{\Delta},\tau)$ as in the right-hand side, by \autoref{prop:quotient_of_lc_pairs_dim_3} the equivalence relation $R(\tau)$ is finite and we can form the geometric quotient $q\colon \bar{X}\to X:=\bar{X}/R(\tau)$. Set $\Delta:=q_*\bar{\Delta}$. By \autoref{prop:finite_equivalence_relation_give_nodes} the scheme $X$ is demi-normal, and by \autoref{prop:descent_of_canonical_bundle_in_dim_3} the $\bQ$-divisor $K_X+\Delta$ is $\bQ$-Cartier. Therefore $(X,\Delta)$ is slc. This gives a map $\{(\bar{X},\bar{D}+\bar{\Delta},\tau)\}\to \{(X,\Delta)\}$. It is an inverse to the normalization map by \autoref{prop:finite_equivalence_relation_give_nodes} and \cite[5.3]{Kollar_Singularities_of_the_minimal_model_program}.
\end{proof}

\section{Application to the moduli theory of stable surfaces}\label{section:applications_to_moduli_spaces}
In this section we apply the theory of gluing to the moduli functor of stable varieties. Our discussion will be conditional, since some technical results are not known yet in positive characteristic.

Let $k$ be an algebraically closed field. We define families of stable log varieties and the moduli functor of stable varieties following \cite{Patakfalvi_Projectivity_moduli_space_of_surfaces_in_pos_char}.

\begin{definition}\label{definition:stable_surfaces}
A projective connected pure-dimensional $k$-scheme $X$ together with a $\bQ$-divisor is a \textbf{stable log pair} if $(X,D)$ is slc and $K_X+D$ is and ample. If $D=0$, we simply say that $X$ is a \textbf{stable variety}.

Let $T$ be a $k$-scheme. A \textbf{family of pairs} over $T$ is a flat morphism of $k$-schemes $X\to T$ together with a $\bQ$-divisor $D$ on $X$ such that: for every $t\in T$, the fiber $X_t$ is geometrically reduced and pure dimensional, none of the irreducible components of $X_t$ is contained in $\Supp D$, and none of the irreducible components of $X_t\cap \Supp D$ is contained in $\Sing X_t$. This allows us to define a restricted divisor $D_t$ on $X_t$.

Let $T$ be a $k$-scheme. A \textbf{family of stable log pairs} over $T$ is a family of pairs $f\colon (X,D)\to T$ such that $K_{X/T}+D$ is $\bQ$-Cartier and the geometric fiber $(X_{\bar{t}}, D_{\bar{t}})$ is a stable log pair for every $t\in T$.

The \textbf{moduli functor} $\overline{\sM}_{n,v,k}$, where $v\in \bQ_+$, is defined on $\text{Sch}_k$ by the values
		$$
		\overline{\sM}_{n,v,k}(T)=\left\{
		\begin{tabu}{c | l l}
		\begin{tikzcd}
		X \arrow[dd, "f" left] \\  \\T
		\end{tikzcd}
		& &
		\begin{aligned}
		1) \; &f \text{ is a flat morphism of }k\text{-schemes},\\
		2) \; &\left(\omega_{X/T}^{[m]}\right)_S\cong \omega_{X_S/S}^{[m]} \text{ for every }S\to T\text{ and }m\in \bN,\\
		3) \; & \text{for every }t\in T, X_{\bar{t}} \text{ is a stable variety of dimension } n\\ & \text{ with } \Vol(K_{X_{\bar{t}}})=v.
		\end{aligned}
		\end{tabu}
		\right\}
		$$
and for $T'\to T$, the corresponding map $\overline{\sM}_{n,v}(T)\to \overline{\sM}_{n,v}(T')$ is given by pullbacks.
\end{definition}

\begin{remark}\label{remark:different_definitions_of_families}
There are subtle differences between families of stable varieties and families parametrized by the functor $\overline{\sM}_{n,v,k}$:
	\begin{enumerate}
		\item At least when $n=2$, if $(X\to T)\in \overline{\sM}_{2,v,k}(T)$ and $T$ is normal, then $X\to T$ is a family of stable surfaces, see \cite[Lemma 2.3]{Patakfalvi_Projectivity_moduli_space_of_surfaces_in_pos_char}.
		\item If $(X\to T)\in \overline{\sM}_{n,v,k}(T)$, then $X\to T$ need not be a family of stable varieties, see \cite[Remark 1.6]{Patakfalvi_Projectivity_moduli_space_of_surfaces_in_pos_char}.
	\end{enumerate}
\end{remark}

\begin{remark}
We have only defined the moduli functor in the boundary-free case. To define a moduli functor of stable log pairs we need a good notion of family of divisors above arbitrary bases. A good notion, at least in characteristic $0$, is developed in \cite{Kollar_Families_of_divisors}. However, to avoid technical difficulties, we will restrict ourselves to the boundary-free case in what follows. This way, the only pairs we will have to deal with are the one arising from the normalization of demi-normal varieties.
\end{remark}

From now on we consider the case of stable surfaces (that is, $n=2$) over an algebraically closed field $k$ of characteristic $p>5$. Then it is known that $\overline{\sM}_{2,v,k}$ is a separated Artin stack of finite type over $k$ with finite diagonal \cite[Theorem 9.7]{Patakfalvi_Projectivity_moduli_space_of_surfaces_in_pos_char}. We discuss the valuative criterion of properness using the methods of \cite[\S 2.4]{Kollar_Families_of_varieties_of_general_type}. 

We are interested in the following situation. Let $T$ be an affine one-dimensional regular scheme of finite type over $k$, $t\in T$ a closed point and $T^0:=T\setminus\{t\}$. Suppose we are given a family $(f^0\colon X^0\to T^0)\in \overline{\sM}_{2,v}(T^0)$. Then we are looking for a finite morphism $\pi\colon T'\to T$ and a family $(f'\colon X'\to T')\in \overline{\sM}_{2,v}(T')$ such that the pullback family $f^0\times_T \pi\colon X^0\times_T T'\to T'$ is isomorphic to $X'\to T'$ over $\pi^{-1}T^0$.

The method of \cite[\S 2.4]{Kollar_Families_of_varieties_of_general_type} in characteristic $0$ can be outlined as follows:
	\begin{enumerate}
		\item[(1)] Let $X_1\to T$ be the flat compactification of $X^0\to T^0$, then take a log resolution $(Y_1,E_1)\to X_1$ such that $(Y_1,E_1+\red(Y_{1,t}))$ is snc for every $t\in T$.
		\item[(2)] After a finite base-change $T'\to T$, we may assume that the fibers of $Y_1\to T$ are reduced and that $(Y_1,E_1+Y_{1,t})$ is lc for every $t\in T$. This is based on local toric computations, see \cite[2.52]{Kollar_Families_of_varieties_of_general_type}.
		\item[(3)] Now take the relative canonical model of $(Y_1,E_1)$ over $T'$. Over the preimage of $T^0$, we get back the original family $X^0\to T^0$. By adjunction, the new central fibers are stable varieties.
	\end{enumerate}
In general the scheme $X^0$ is only demi-normal, which is inconvenient for many steps in the proof. Thus we normalize, and drag along the conductor divisor. In the end we get a canonical model $(X_\text{can},D_\text{can})\to T'$ together with an involution $\tau^0$ of $D_\text{can}^n\times_T T^0$. By separatedness of the moduli, the involution extends to an involution $\tau$ of $D_\text{can}^n$. It turns out that gluing along the equivalence classes of $\tau$ is not a problem because none of the lc centers is contained in the new central fibers (see \cite[9.55]{Kollar_Singularities_of_the_minimal_model_program}), but the descent of the log canonical $\bQ$-line bundle to the quotient is not as easy. One has to apply \cite[5.38]{Kollar_Singularities_of_the_minimal_model_program}, for which the full strength of the theory of sources and springs is necessary.

\bigskip\noindent One must solve several problems to carry out this program in positive characteristic:
	\begin{enumerate}
		\item The adjunction dictionnary between fiberwise and global properties works well in one direction. Suppose $f\colon (X,D)\to T$ is a family of stable log surfaces over a one-dimensional normal base. Since $T$ and every fiber $X_t$ are $S_2$, we see that $X$ is $S_2$. Points of codimension one of $X$ that do not dominate $T$ are regular, and those which dominate $T$ are Gorenstein since the generic fiber is demi-normal. Thus $X$ is demi-normal. Now inversion of adjunction implies that $(X,D+X_t)$ is slc for every $t\in T$.
		
		Notice that if $(f\colon X\to T)\in\overline{\sM}_{2,v,k}(T)$ and $(\bar{X},\bar{D})\to X$ is the normalization, then it follows from \autoref{remark:different_definitions_of_families} and \autoref{lemma: basic properties of normalization of demi-normal scheme} that the morphism $(\bar{X},\bar{D})\to T$ is a family of stable log surfaces (with $\bar{D}$ dominating $T$ if $X$ is not already normal).
		
		Problems appear with the converse implication. Assume that $(X,D+X_t)$ is slc. Then the deformation theory of nodes show that $X_t$ is Gorenstein in codimension one \cite[2.33]{Kollar_Singularities_of_the_minimal_model_program}, and adjunction would then imply that $(X_t,D_t)$ is slc, provided that $X_t$ is $S_2$. That $X_t$ is $S_2$ in characteristic $0$ follows from a non-trivial result of Alexeev (see \cite[7.21]{Kollar_Singularities_of_the_minimal_model_program}). This is not known at the moment in positive characteristic, so let us formulate the condition
		
		\begin{tabular}{p{1cm} l p{12cm}}
		\textbf{(S2)} & & If $(X,D)\to T$ is a flat family of geometrically reduced surface pairs over a one-dimensonal normal base such that $(X,D+X_t)$ is slc for every $t\in T$, then every $X_t$ is $S_2$.
		\end{tabular} 
		
		\item To produce the completed family, one first extends $X^0\to T^0$ to a flat family $X_1\to T$. In general the central fiber is not even reduced: so one looks for a base-change along a finite $T'\to T$ such that the fibers of $X'_1:=X_1\times_TT'\to T'$ are reduced. If we find one such base-change, we have to make sure that $(X_1',(X_1')_{t'})$ is still slc for every $t'\in T'$. In positive characteristic, this is a problem if $T'\to T$ is wildly ramified or inseparable, see \cite[2.14.5-6]{Kollar_Singularities_of_the_minimal_model_program}. So we formulate the following condition of semi-stable reduction:
		
		\begin{tabular}{p{1cm} l p{12cm}}
		\textbf{(SSR)} & & Let $X\to T$ is a flat morphism where $X$ is a regular threefold and $T$ a one-dimensional curve. Let $E$ be a reduced effective divisor on $X$ such that $(X,E+\red X_t)$ is snc for every closed $t\in T$. Then there exists a finite morphism $T'\to T$ such that: if $Y$ is the normalization of $X\times_T T'$ and $E_Y$ is the pullback divisor, then every closed fiber $Y_{t'}$ is reduced and every $(Y, E_Y+Y_{t'})$ is lc.
		\end{tabular}
	\end{enumerate}
Modulo these two conditions, we can prove the valuative criterion of properness for $\overline{\sM}_{2,v,k}$, following the method of \cite[2.49]{Kollar_Families_of_varieties_of_general_type}.

\begin{lemma}\label{lemma:adjunction_and_stability}
Let $(X,D+\Delta)\to T$ be a family of stable log surfaces over a normal one-dimensional base, where $D$ is a reduced divisor with normalization $n\colon D^n\to D$. Then every componenent of $D^n$ dominates $T$ and $(D^n,\Diff_{D^n}\Delta)\to T$ is a (disjoint union of) families of stable log curves.
\end{lemma}
\begin{proof}
By definition of families of pairs, every irreducible component of $D$ dominates $T$. In particular every component of $D^n$ dominates $T$ and $D^n\to T$ is flat. So the fibers $D^n_t:=(D^n)_t$ are of pure dimension one. By inversion of adjunction the pair $(X,D+\Delta+X_t)$ is slc. Passing to the normalization (recall that $X$ is normal at the generic points of $D$), we may assume it is lc. Since $X_t$ is Cartier we have
		$$\Diff_{D^n}(\Delta+X_t)=\Diff_{D^n}(\Delta)+X_t|_{D^n}=\Diff_{D^n}+D^n_t.$$
	By adjunction, this implies that $(D^n,\Diff_{D^n}\Delta+D^n_t)$ is lc. Since $D^n$ is a surface, classification of surface lc singularities show that $D^n_t$ is Gorenstein. Moreover, as $D^n$ is $S_2$ (it is normal), the $D^n_t$ are $S_1$. Thus $D^n_t$ is a demi-normal curve, and by adjunction we deduce that $(D^n,\Diff_{D^n}\Delta)\to T$ is a family of log curves. As 
			$$K_{D^n/T}+\Diff_{D^n}\Delta=n^*(K_{X/T}+D+\Delta),$$
	we obtain that $K_{D^n/T}+\Diff_{D^n}\Delta$ is ample over $T$.
\end{proof}

\begin{theorem}\label{theorem:Properness_of_moduli_of_stable_surfaces}
We work over an algebraically closed field $k$ of characteristic $>5$. Assume that the conditions \textbf{(S2)} and \textbf{(SSR)} hold. Then $\overline{\sM}_{2,v,k}$ is proper.
\end{theorem}
\begin{proof}
We consider anew an affine one-dimensional regular $k$-scheme of finite type $T$, a closed point $t\in T$ and a family $(f_0\colon X^0\to T^0:=T\setminus\{t\})\in \overline{\sM}_{2,v,k}(T^0)$. Since separatedness holds, we have to show that there exists a finite morphism $\pi\colon T'\to T$ such that the pullback family $f^0\times_T\pi\colon X^0\times_TT'\to \pi^{-1}T^0$ extends to a family in $\overline{\sM}_{2,v,k}(T')$. By \autoref{remark:different_definitions_of_families}, we may think of $f^0$ as a family of stable surfaces.

\bigskip
\textsc{Step 1: Normal case.} First assume that $X^0$ is actually normal. We can extend $f^0$ to a flat morphism $f_1\colon X_1\to T$. Let $g_1\colon Y_1\to X_1$ be a log resolution with $E_1:=\Exc(g_1)$ such that $(Y_1,E_1+\red (Y_1)_t)$ is snc for every $t\in T$. Such a resolution exists for threefolds: the proof is the same as in \cite[10.46]{Kollar_Singularities_of_the_minimal_model_program}, using \autoref{fact:resolution_of_singularities_for_threefolds} at the appropriate places.

By \textbf{(SSR)}, there is a finite morphism $\pi\colon T'\to T$ such that: if $Y_2$ is the normalization of $Y_1\times_TT'$ with induced morphism $f_2\colon Y_2\to T'$ and $E_2$ is the pullback divisor of $E_1$, then every fiber $(Y_2)_{t'}$ is reduced and $(Y_2,E_2+(Y_2)_{t'})$ is lc for every closed $t'\in T'$.

By \cite{Waldron_MMP_for_3_folds_in_char_>5}, the family $(Y_2,E_2)\to T'$ admits a relative canonical model $(X_{\can}, E_{\can})\to T'$. Notice that the pullback family $X^0\times_TT'\to \pi^{-1}T^0$ is a relative canonical model of $(Y_2,E_2)\to T'$ over a dense open susbset of $T'$. By uniqueness of canonical models, these two families are generically isomorphic. By separatedness \cite[Lemma 9.4]{Patakfalvi_Projectivity_moduli_space_of_surfaces_in_pos_char}, the isomorphism extends to the whole $\pi^{-1}T^0$, and this implies that $E_{\can}=0$. Using \cite{Kollar_Singularities_of_the_minimal_model_program} we see that the pair $(X_{\can}, (X_{\can})_{t'})$ is lc for every closed $t'\in T'$. The condition \textbf{(S2)} then ensures that $(X_{\can}\to T')\in\overline{\sM}_{2,v,k}(T')$.

\bigskip
\textsc{Step 2: Demi-normal case.} Now we consider the case where $X^0$ is only demi-normal. Let $(\bar{X}^0,\bar{D}^0)\to T^0$ be the normalization, with induced involution $\tau^0$. Then we can apply the same argument as before (all the results we used are available for log pairs). We obtain a finite morphism $\pi\colon T'\to T$ and a family of surface pairs $\bar{f}_{\can}\colon (\bar{X}_{\can},\bar{D}_{\can})\to T'$ which extends the pullback of $(\bar{X}^0,\bar{D}^0)\to T^0$ and such that $(\bar{X}_{\can},\bar{D}_{\can}+(\bar{X}_{\can})_{t'})$ is lc for every closed $t'\in T'$.

By \autoref{lemma:adjunction_and_stability}, the morphism $(\bar{D}^n_{\can},\Diff_{\bar{D}^n_{\can}}(0))\to T'$ is a family of stable log curves. Hence by \cite[Lemma 9.4]{Patakfalvi_Projectivity_moduli_space_of_surfaces_in_pos_char}, the pullback of the involution $\tau^0$ extends to an involution $\tau_{\can}$ on $(\bar{D}^n_{\can},\Diff_{\bar{D}^n_{\can}}(0))$. 

By \cite[2.14, 2.15.3]{Kollar_Families_of_varieties_of_general_type}, none of the lc centers of $(\bar{X}_{\can},\bar{D}_{\can})$ are disjoint from $\bar{X}^0\times_T T'$. Moreover, since $R(\tau^0)\rightrightarrows \bar{X}^0$ is a finite equivalence relation and $T'\to T$ is finite as well, then $R(\tau_{\can})\rightrightarrows \bar{X}_{\can}$ is finite on $\bar{f}_{\can}(\pi^{-1}T^0)$. Thus \cite[9.55]{Kollar_Singularities_of_the_minimal_model_program} implies that $R(\tau_{\can})$ is finite. Thus there exists a geometric quotient $\bar{X}_{\can}\to X_{\can}$. Since $\tau^0$ commutes with the projection to $T$, the involution $\tau_{\can}$ commutes with the projection to $T'$, and therefore $\bar{f}_{\can}$ factorises through a morphism $f_{\can}\colon X_{\can}\to T'$.

By \autoref{prop:finite_equivalence_relation_give_nodes} the scheme $X_{\can}$ is demi-normal, and by \autoref{prop:descent_of_canonical_bundle_in_dim_3} the divisor $K_{X_{\can}}$ is $\bQ$-Cartier. Since $(X_{\can},(X_{\can})_{t'})$ pullbacks to $(\bar{X}_{\can}, \bar{D}_{\can}+(\bar{X}_{\can})_{t'})$, we see that $(X_{\can},(X_{\can})_{t'})$ is slc for every closed $t'\in T'$. Since $\bar{X}^0\times_T T'\to \bar{X}^0$ is flat and the formation of geometric quotient commutes with flat base-change \cite[9.11]{Kollar_Singularities_of_the_minimal_model_program}, we see that the pullback of $X^0$ along $\pi^{-1}T^0\to T^0$ is isomorphic to $f_{\can}^{-1}(\pi^{-1}T^0)$. The condition \textbf{(S2)} ensures that $(f_{\can}\colon X_{\can}\to T')\in \overline{\sM}_{n,v,k}(T')$. This finishes the proof.
\end{proof}







\bibliographystyle{alpha}
\bibliography{Bibliography}

\begin{thebibliography}{HMX18}

\bibitem[ABL22]{Arvidsson_Bernasconi_Lacini_KVV_for_log_dP_surfaces_in_pos_char}
Emelie Arvidsson, Fabio Bernasconi, and Justin Lacini.
\newblock On the {K}awamata-{V}iehweg vanishing theorem for log del {P}ezzo
  surfaces in positive characteristic.
\newblock {\em Compos. Math.}, 158(4):750--763, 2022.

\bibitem[AH11]{Abramovich_Hassett_Stable_varieties_with_a_twist}
Dan Abramovich and Brendan Hassett.
\newblock Stable varieties with a twist.
\newblock In {\em Classification of algebraic varieties}, EMS Ser. Congr. Rep.,
  pages 1--38. Eur. Math. Soc., Z\"{u}rich, 2011.

\bibitem[Ale94]{Alexeev_Boundedness_and_K2_for_log_surfaces}
Valery Alexeev.
\newblock Boundedness and {$K^2$} for log surfaces.
\newblock {\em Internat. J. Math.}, 5(6):779--810, 1994.

\bibitem[Art75]{Artin_Wildly_ramified_Z/2_action_in_dim_2}
M.~Artin.
\newblock Wildly ramified {$\mathbb{Z}/2$} actions in dimension two.
\newblock {\em Proc. Amer. Math. Soc.}, 52:60--64, 1975.

\bibitem[AT51]{Artin_Tate_Note_on_finite_extensions}
Emil Artin and John~T. Tate.
\newblock A note on finite ring extensions.
\newblock {\em J. Math. Soc. Japan}, 3:74--77, 1951.

\bibitem[Ber14]{Berquist_Embedding_of_demi_normal_varieties}
Jeremy Berquist.
\newblock Embeddings of {D}emi-{N}ormal {V}arieties.
\newblock {\em ArXiv e-print, arXiv:1411.2264v1}, 2014.

\bibitem[Ber19]{Bernasconi_Non_normal_plt_centers_in_pos_char}
Fabio Bernasconi.
\newblock Non-normal purely log terminal centres in characteristic
  {$p\geqslant3$}.
\newblock {\em Eur. J. Math.}, 5(4):1242--1251, 2019.

\bibitem[BK21]{Bernasconi_Kollar_vanishing_theorems_for_threefolds_in_char_>5}
Fabio Bernasconi and J\'{a}nos Koll\'{a}r.
\newblock Vanishing theorems for threefolds in characteristic $p>5$.
\newblock {\em ArXiv e-print, arXiv:2012.08343v3}, 2021.

\bibitem[B\u01]{Badescu_Agebraic_surfaces}
Lucian B\u{a}descu.
\newblock {\em Algebraic surfaces}.
\newblock Universitext. Springer-Verlag, New York, 2001.
\newblock Translated from the 1981 Romanian original by Vladimir Ma\c{s}ek and
  revised by the author.

\bibitem[CP08]{Cossart_Piltant_Resolution_of_singularities_for_3_folds_in_pos_char_I}
Vincent Cossart and Olivier Piltant.
\newblock Resolution of singularities of threefolds in positive characteristic.
  {I}. {R}eduction to local uniformization on {A}rtin-{S}chreier and purely
  inseparable coverings.
\newblock {\em J. Algebra}, 320(3):1051--1082, 2008.

\bibitem[CT19]{Cascini_Tanaka_Plt_threefolds_with_non_normal_centres_in_char_2}
Paolo Cascini and Hiromu Tanaka.
\newblock Purely log terminal threefolds with non-normal centres in
  characteristic two.
\newblock {\em Amer. J. Math.}, 141(4):941--979, 2019.

\bibitem[DH16]{Das_Hacon_Adjunction_for_3_folds_in_char_>5}
Omprokash Das and Christopher~D. Hacon.
\newblock On the adjunction formula for 3-folds in characteristic {$p>5$}.
\newblock {\em Math. Z.}, 284(1-2):255--269, 2016.

\bibitem[FG14]{Fujino_Gongyo_Log_pluricanonical_representations_and_abundance}
Osamu Fujino and Yoshinori Gongyo.
\newblock Log pluricanonical representations and the abundance conjecture.
\newblock {\em Compos. Math.}, 150(4):593--620, 2014.

\bibitem[Fli92]{Flips_and_abundance_for_3folds}
{\em Flips and abundance for algebraic threefolds}.
\newblock Soci\'{e}t\'{e} Math\'{e}matique de France, Paris, 1992.
\newblock Papers from the Second Summer Seminar on Algebraic Geometry held at
  the University of Utah, Salt Lake City, Utah, August 1991, Ast\'{e}risque No.
  211 (1992) (1992).

\bibitem[Fuj18]{Fujino_Semipositivity_theorems_for_moduli_problems}
Osamu Fujino.
\newblock Semipositivity theorems for moduli problems.
\newblock {\em Ann. of Math. (2)}, 187(3):639--665, 2018.

\bibitem[FW21]{Filipazzi_Waldron_Connectedness_principle_for_3folds_pos_char}
Stefano Filipazzi and Joe Waldron.
\newblock Connectedness principle for $3$-folds in characteristic $p>5$.
\newblock {\em ArXiv e-print, arXiv:2010.08414v3}, 2021.

\bibitem[GM78]{Greco_Marinari_Nagata_criterion}
Silvio Greco and Maria~Grazia Marinari.
\newblock Nagata's criterion and openness of loci for {G}orenstein and complete
  intersection.
\newblock {\em Math. Z.}, 160(3):207--216, 1978.

\bibitem[Gro65]{EGA_IV.2}
A.~Grothendieck.
\newblock \'{E}l\'{e}ments de g\'{e}om\'{e}trie alg\'{e}brique. {IV}. \'{E}tude
  locale des sch\'{e}mas et des morphismes de sch\'{e}mas. {II}.
\newblock {\em Inst. Hautes \'{E}tudes Sci. Publ. Math.}, (24):231, 1965.

\bibitem[Gro66]{EGA_IV.3}
A.~Grothendieck.
\newblock \'{E}l\'{e}ments de g\'{e}om\'{e}trie alg\'{e}brique. {IV}. \'{E}tude
  locale des sch\'{e}mas et des morphismes de sch\'{e}mas. {III}.
\newblock {\em Inst. Hautes \'{E}tudes Sci. Publ. Math.}, (28):255, 1966.

\bibitem[GT80]{Greco_Traverso_On_seminormal_schemes}
S.~Greco and C.~Traverso.
\newblock On seminormal schemes.
\newblock {\em Compositio Math.}, 40(3):325--365, 1980.

\bibitem[Har77]{Hartshorne_Algebraic_Geometry}
Robin Hartshorne.
\newblock {\em Algebraic geometry}.
\newblock Springer-Verlag, New York-Heidelberg, 1977.
\newblock Graduate Texts in Mathematics, No. 52.

\bibitem[Har07]{Hartshorne_Generalized_divisors_and_biliaison}
Robin Hartshorne.
\newblock Generalized divisors and biliaison.
\newblock {\em Illinois J. Math.}, 51(1):83--98, 2007.

\bibitem[Har10]{Hartshorne_Deformation_Theory}
Robin Hartshorne.
\newblock {\em Deformation theory}, volume 257 of {\em Graduate Texts in
  Mathematics}.
\newblock Springer, New York, 2010.

\bibitem[HK04]{Hassett_Kovacs_Reflexive_pullbacks_and_base_extension}
Brendan Hassett and S\'{a}ndor~J. Kov\'{a}cs.
\newblock Reflexive pull-backs and base extension.
\newblock {\em J. Algebraic Geom.}, 13(2):233--247, 2004.

\bibitem[HMX18]{Hacon_McKernan_Xu_Boundedness_of_moduli_of_varieties_of_general_type}
Christopher~D. Hacon, James McKernan, and Chenyang Xu.
\newblock Boundedness of moduli of varieties of general type.
\newblock {\em J. Eur. Math. Soc. (JEMS)}, 20(4):865--901, 2018.

\bibitem[HNT20]{Hashizume_Nakamura_Tanaka_MMP_for_lc_threefolds_in_char_p}
Kenta Hashizume, Yusuke Nakamura, and Hiromu Tanaka.
\newblock Minimal model program for log canonical threefolds in positive
  characteristic.
\newblock {\em Math. Res. Lett.}, 27(4):1003--1054, 2020.

\bibitem[HW19]{Hacon_Witaszek_On_the_relative_MMP_for_threefolds_in_low_char}
Christopher Hacon and Jakub Witaszek.
\newblock On the relative {M}inimal {M}odel {P}rogram for threefolds in low
  characteristics.
\newblock {\em ArXiv e-print, arXiv:1909.12872v2}, 2019.

\bibitem[HX13]{Hacon_Xu_Existence_of_lc_closures}
Christopher~D. Hacon and Chenyang Xu.
\newblock Existence of log canonical closures.
\newblock {\em Invent. Math.}, 192(1):161--195, 2013.

\bibitem[HX15]{Hacon_Xu_On_the_3dim_MMP_in_pos_char}
Christopher~D. Hacon and Chenyang Xu.
\newblock On the three dimensional minimal model program in positive
  characteristic.
\newblock {\em J. Amer. Math. Soc.}, 28(3):711--744, 2015.

\bibitem[HX16]{Hacon_Xu_Finiteness_of_B_representations_and_slc_abundance}
Christopher~D. Hacon and Chenyang Xu.
\newblock On finiteness of {B}-representations and semi-log canonical
  abundance.
\newblock In {\em Minimal models and extremal rays ({K}yoto, 2011)}, volume~70
  of {\em Adv. Stud. Pure Math.}, pages 361--377. Math. Soc. Japan, [Tokyo],
  2016.

\bibitem[Kar00]{Karu_Minimal_models_and_boundedness_of_stable_varieties}
Kalle Karu.
\newblock Minimal models and boundedness of stable varieties.
\newblock {\em J. Algebraic Geom.}, 9(1):93--109, 2000.

\bibitem[Kee99]{Keel_basepoint_freeness_in_positive_char}
Se\'{a}n Keel.
\newblock Basepoint freeness for nef and big line bundles in positive
  characteristic.
\newblock {\em Ann. of Math. (2)}, 149(1):253--286, 1999.

\bibitem[KK]{Kollar_Kovacs_Birational_Geometry_of_log_surfaces}
J\'{a}nos Koll\'{a}r and S\'{a}ndor Kov\'{a}cs.
\newblock Birational geometry of log surfaces.
\newblock {\em Available at
  \url{https://sites.math.washington.edu/~kovacs/pdf/BiratLogSurf.pdf}}.

\bibitem[KM98]{Kollar_Mori_Birational_geometry_of_algebraic_varieties}
J{\'a}nos Koll{\'a}r and Shigefumi Mori.
\newblock {\em Birational geometry of algebraic varieties}, volume 134 of {\em
  Cambridge Tracts in Mathematics}.
\newblock Cambridge University Press, Cambridge, 1998.
\newblock With the collaboration of C. H. Clemens and A. Corti, Translated from
  the 1998 Japanese original.

\bibitem[Kol90]{Kollar_Projectivity_complete_moduli}
J\'{a}nos Koll\'{a}r.
\newblock Projectivity of complete moduli.
\newblock {\em J. Differential Geom.}, 32(1):235--268, 1990.

\bibitem[Kol96]{Kollar_Rational_curves}
J\'{a}nos Koll\'{a}r.
\newblock {\em Rational curves on algebraic varieties}, volume~32 of {\em
  Ergebnisse der Mathematik und ihrer Grenzgebiete. 3. Folge. A Series of
  Modern Surveys in Mathematics [Results in Mathematics and Related Areas. 3rd
  Series. A Series of Modern Surveys in Mathematics]}.
\newblock Springer-Verlag, Berlin, 1996.

\bibitem[Kol11]{Kollar_Two_examples_of_nc_surfaces}
J\'{a}nos Koll\'{a}r.
\newblock Two examples of surfaces with normal crossing singularities.
\newblock {\em Sci. China Math.}, 54(8):1707--1712, 2011.

\bibitem[Kol12]{Kollar_Quotients_by_finite_equivalence_relations}
J\'{a}nos Koll\'{a}r.
\newblock Quotients by finite equivalence relations.
\newblock In {\em Current developments in algebraic geometry}, volume~59 of
  {\em Math. Sci. Res. Inst. Publ.}, pages 227--256. Cambridge Univ. Press,
  Cambridge, 2012.
\newblock With an appendix by Claudiu Raicu.

\bibitem[Kol13]{Kollar_Singularities_of_the_minimal_model_program}
J{\'a}nos Koll{\'a}r.
\newblock {\em Singularities of the Minimal Model Program}, volume 200 of {\em
  Cambridge Tracts in Mathematics}.
\newblock 2013.

\bibitem[Kol16]{Kollar_Sources_of_lc_centers}
J\'{a}nos Koll\'{a}r.
\newblock Sources of log canonical centers.
\newblock In {\em Minimal models and extremal rays ({K}yoto, 2011)}, volume~70
  of {\em Adv. Stud. Pure Math.}, pages 29--48. Math. Soc. Japan, [Tokyo],
  2016.

\bibitem[Kol18]{Kollar_Log_plurigenera_for_stable_surface_families}
J\'{a}nos Koll\'{a}r.
\newblock Log-plurigenera in stable families of surfaces.
\newblock {\em Peking Math. J.}, 1(1):109--124, 2018.

\bibitem[Kol19]{Kollar_Families_of_divisors}
J\'{a}nos Koll\'{a}r.
\newblock Families of divisors.
\newblock {\em ArXiv e-print, arXiv:1910.00937v1}, 2019.

\bibitem[Kol21]{Kollar_Families_of_varieties_of_general_type}
J\'anos Koll\'ar.
\newblock Families of varieties of general type.
\newblock {\em Available at
  \url{https://web.math.princeton.edu/~kollar/FromMyHomePage/modbook.pdf}},
  2021.
\newblock With the collaboration of {K}. {A}ltmann and {S}. {K}ov\'{a}cs.

\bibitem[KP17]{Kovacs_Patakfalvi_Projectivity_of_moduli_space_of_stable_log_varieties}
S\'{a}ndor~J. Kov\'{a}cs and Zsolt Patakfalvi.
\newblock Projectivity of the moduli space of stable log-varieties and
  subadditivity of log-{K}odaira dimension.
\newblock {\em J. Amer. Math. Soc.}, 30(4):959--1021, 2017.

\bibitem[KSB88]{Kollar_Shepherd_Barron_3folds_and_deformations_of_surfaces_singularities}
J.~Koll\'{a}r and N.~I. Shepherd-Barron.
\newblock Threefolds and deformations of surface singularities.
\newblock {\em Invent. Math.}, 91(2):299--338, 1988.

\bibitem[Lan02]{Lang_Algebra}
Serge Lang.
\newblock {\em Algebra}, volume 211 of {\em Graduate Texts in Mathematics}.
\newblock Springer-Verlag, New York, third edition, 2002.

\bibitem[LV81]{Leahy_Seminormal_rings}
John~V. Leahy and Marie~A. Vitulli.
\newblock Seminormal rings and weakly normal varieties.
\newblock {\em Nagoya Math. J.}, 82:27--56, 1981.

\bibitem[Mat89]{Matsumura_Commutative_Ring_Theory}
Hideyuki Matsumura.
\newblock {\em Commutative ring theory}, volume~8 of {\em Cambridge Studies in
  Advanced Mathematics}.
\newblock Cambridge University Press, Cambridge, second edition, 1989.
\newblock Translated from the Japanese by M. Reid.

\bibitem[NT20]{Nakamura_Tanaka_Witt_Nadel_vanishing_for_threefolds}
Yusuke Nakamura and Hiromu Tanaka.
\newblock A {W}itt {N}adel vanishing theorem for threefolds.
\newblock {\em Compos. Math.}, 156(3):435--475, 2020.

\bibitem[Pat17]{Patakfalvi_Projectivity_moduli_space_of_surfaces_in_pos_char}
Zsolt Patakfalvi.
\newblock On the projectivity of the moduli space of stable surfaces in char
  $>5$.
\newblock {\em ArXiv e-print, arXiv:1710.03818v3}, 2017.

\bibitem[Pos21]{Posva_Gluing_in_mixed_char}
Quentin Posva.
\newblock Gluing for stable families of surfaces in mixed characteristic.
\newblock {\em ArXiv e-print, arXiv:2110.09845v3}, 2021.

\bibitem[PW22]{Patakfalvi_Waldron_Singularities_of_general_fibers}
Zsolt Patakfalvi and Joe Waldron.
\newblock Singularities of general fibers and the {LMMP}.
\newblock {\em Amer. J. Math.}, 144(2):505--540, 2022.

\bibitem[PZ20]{Patakfalvi_Zdanowicz_Beauville_Bogomolov_in_pos_char}
Zsolt Patakfalvi and Maciej Zdanowicz.
\newblock On the {B}eauville--{B}ogomolov decomposition in characteristic
  $p\geq 0$.
\newblock {\em ArXiv e-print, arXiv:1912.12742v2.}, 2020.

\bibitem[Sak84]{Sakai_Divisors_normal_surfaces}
Fumio Sakai.
\newblock Weil divisors on normal surfaces.
\newblock {\em Duke Math. J.}, 51(4):877--887, 1984.

\bibitem[SGA71]{SGA6}
{\em Th\'{e}orie des intersections et th\'{e}or\`eme de {R}iemann-{R}och}.
\newblock Lecture Notes in Mathematics, Vol. 225. Springer-Verlag, Berlin-New
  York, 1971.
\newblock S\'{e}minaire de G\'{e}om\'{e}trie Alg\'{e}brique du Bois-Marie
  1966--1967 (SGA 6), Dirig\'{e} par P. Berthelot, A. Grothendieck et L.
  Illusie. Avec la collaboration de D. Ferrand, J. P. Jouanolou, O. Jussila, S.
  Kleiman, M. Raynaud et J. P. Serre.

\bibitem[Sil09]{Silverman_Arithmetic_of_elliptic_curves}
Joseph~H. Silverman.
\newblock {\em The arithmetic of elliptic curves}, volume 106 of {\em Graduate
  Texts in Mathematics}.
\newblock Springer, Dordrecht, second edition, 2009.

\bibitem[Sta]{Stacks_Project}
The {S}tacks {P}roject.
\newblock {\em \url{https://stacks.math.columbia.edu}}.

\bibitem[Tan14]{Tanaka_MMP_and_abundance_for_positive_characteristic_log_surfaces}
Hiromu Tanaka.
\newblock Minimal models and abundance for positive characteristic log
  surfaces.
\newblock {\em Nagoya Math. J.}, 216:1--70, 2014.

\bibitem[Tan16]{Tanaka_Abundance_for_slc_surfaces}
Hiromu Tanaka.
\newblock Abundance theorem for semi log canonical surfaces in positive
  characteristic.
\newblock {\em Osaka J. Math.}, 53(2):535--566, 2016.

\bibitem[Tan18]{Tanaka_MMP_for_excellent_surfaces}
Hiromu Tanaka.
\newblock Minimal model program for excellent surfaces.
\newblock {\em Ann. Inst. Fourier (Grenoble)}, 68(1):345--376, 2018.

\bibitem[Vie95]{Viehweg_Qproj_Moduli_for_polarized_manifolds}
Eckart Viehweg.
\newblock {\em Quasi-projective moduli for polarized manifolds}, volume~30 of
  {\em Ergebnisse der Mathematik und ihrer Grenzgebiete (3) [Results in
  Mathematics and Related Areas (3)]}.
\newblock Springer-Verlag, Berlin, 1995.

\bibitem[vS87]{van_Straten_weakly_normal_surfaces}
Duco van Straten.
\newblock {\em Weakly normal surface singularities and their improvements}.
\newblock PhD thesis, Leiden, 1987.

\bibitem[Wal18]{Waldron_MMP_for_3_folds_in_char_>5}
Joe Waldron.
\newblock The {LMMP} for log canonical 3-folds in characteristic {$p>5$}.
\newblock {\em Nagoya Math. J.}, 230:48--71, 2018.

\bibitem[Yan83]{Yanagihara_Weakly_normal_ring_extensions}
Hiroshi Yanagihara.
\newblock Some results on weakly normal ring extensions.
\newblock {\em J. Math. Soc. Japan}, 35(4):649--661, 1983.

\end{thebibliography}

\end{document}